\documentclass[aap]{imsart}

\RequirePackage{amsthm,amsmath,amsfonts,amssymb}
\RequirePackage[numbers]{natbib}
\RequirePackage[colorlinks,citecolor=blue,urlcolor=blue]{hyperref}
\usepackage{amsmath,amsthm,amssymb,latexsym,enumerate, mathtools,mathrsfs,todonotes}

\startlocaldefs
\numberwithin{equation}{section}
\theoremstyle{plain}

\newtheorem{lemma}{Lemma}
\newtheorem{proposition}[lemma]{Proposition}
\newtheorem{theorem}[lemma]{Theorem}
\newtheorem{corollary}[lemma]{Corollary}

\theoremstyle{remark}
\newtheorem{definition}[lemma]{Definition}
\newtheorem{remark}[lemma]{Remark}

\numberwithin{lemma}{section}

\newcommand{\RR}{{\mathbb R}}
\newcommand{\NN}{{\mathbb N}}

\newcommand{\del}{\partial}

\newcommand{\dist}{\on{dist}}
\newcommand{\eps}{\varepsilon}
\newcommand{\Id}{\mathrm{Id}}

\newcommand{\norm}[1]{ \left\| #1 \right\| }
\newcommand{\oline}[1]{\overline{#1}}
\newcommand{\osc}{\on{osc}}
\newcommand{\oo}{\infty}
\newcommand{\sgn}{\on{sgn}}

\newcommand{\tr}{\on{tr}}
\newcommand{\uline}[1]{\underline{#1}}

\DeclareMathOperator*{\argmax}{arg\,max}
\DeclareMathOperator*{\argmin}{arg\,min}
\DeclareMathOperator*{\esssup}{ess\,sup}

\newcommand{\mcl}{\mathcal}
\newcommand{\mbb}{\mathbb}

\newcommand{\on}{\operatorname}

\endlocaldefs

\begin{document}

\begin{frontmatter}
\title{The Neumann problem for fully nonlinear SPDE}
\runtitle{The Neumann problem for fully nonlinear SPDE}

\begin{aug}
\author[A]{\fnms{Paul} \snm{Gassiat}\ead[label=e1]{gassiat@ceremade.dauphine.fr}},
\author[B]{\fnms{Benjamin} \snm{Seeger}\ead[label=e2]{seeger@math.utexas.edu}}
\address[A]{Universit\'e Paris-Dauphine, PSL University, UMR 7534, CNRS, CEREMADE, 75775 Paris Cedex 16, France, \printead{e1}}

\address[B]{University of Texas at Austin, 2515 Speedway, PMA 8.100, Austin, TX 78712, \printead{e2}}
\end{aug}

\begin{abstract}
We generalize the notion of pathwise viscosity solutions, put forward by Lions and Souganidis to study fully nonlinear stochastic partial differential equations, to equations set on a sub-domain with Neumann boundary conditions. Under a convexity assumption on the domain, we obtain a comparison theorem which yields existence and uniqueness of solutions as well as continuity with respect to the driving noise. As an application, we study the long time behaviour of a stochastically perturbed mean-curvature flow in a cylinder-like domain with right angle contact boundary condition.
\end{abstract}



\end{frontmatter}
\tableofcontents



\section{Introduction} 

The focus of this paper is the Neumann problem for the equation given by
\begin{equation}\label{E:neumann}
	\begin{dcases}
		du = F(D^2 u, Du,u,x,t) dt + \sum_{i=1}^m H^i(Du) \cdot d\zeta^i & \text{in } \Omega \times (0,T],\\
		Du \cdot n = 0 & \text{on } \del \Omega \times [0,T],
	\end{dcases}
\end{equation}
where $\Omega \subset \RR^d$ is a given subdomain, $F: \mbb S^d \times \RR^d \times \RR \times \RR^d \times [0,T] \to \RR$ is degenerate elliptic, $H: \RR^d \to \RR^m$ is sufficiently regular, and $\zeta: [0,T] \to \RR^m$ is a fixed continuous path. We provide more precise assumptions in what follows.

In the case when $\zeta$ is sufficiently regular (say $C^1$), this equation is covered by the classical framework of Crandall-Lions viscosity solutions (e.g. \cite{CIL}). Our interest in this paper is to treat the ``rough'' case, where we only assume that $\zeta$ is a continuous function, so that $\dot{\zeta}(t)$ does not make sense as a pointwise function (this is also known as the ``stochastic'' case, since a major motivation is to apply the theory for cases when $\zeta$ is the realisation of a stochastic process such as Brownian motion, which is naturally an irregular object).

{
These equations were introduced in a series of works by Lions and Souganidis \cite{LSfirst, LSnonsmooth, LSsemilinear, LSunique}, who explained how to extend the theory of viscosity solutions to deal with such irregular terms; see \cite{Snotes} for a comprehensive overview. A number of applications were also discussed in these works, among them the level-set formulation of the motion of hypersurfaces when the dynamics are perturbed by a stochastic noise (this includes in particular the case of stochastic mean curvature flow). Other recent developments include the analysis of qualitative behavior, for instance, long-time behavior, regularity/regularization by noise, and finite/infinite speed of propagation \cite{G,GG19, GGLS, LSreg}; the construction of numerical schemes \cite{Seegerschemes}; and applications to stochastically perturbed mean curvature flow \cite{SY,LSac}.

Several new tools and techniques are required to extend the already well-developed viscosity solution theory to the setting of equations driven by rough multiplicative noise. Many of the aforementioned works take advantage of the spatial homogeneity of the Hamiltonians $H^i$ and the translation invariance of the spatial domain, which is taken to be the whole space $\RR^d$ or the torus. Some recent works treat equations in which this homogeneity is broken by considering $x$-dependent Hamiltonians \cite{FGLS, Shomog, SPerrons}. The treatment of equations with $(x,t)$-dependence is by now standard in the classical viscosity solution theory, with well-posedness holding under quite general structural assumptions on the nonlinearities. By contrast, if the noise coefficients $H^i$ in \eqref{E:neumann} depend on $x$, the analysis must be completely revisited on account of the wild behavior of the term $d \zeta$. As a consequence, quite particular restrictions are put on the Hamiltonians in those works, for instance a separated structure $h(Du) + f(x)$, or a ``metric'' structure in which $H$ is scalar valued, convex, and coercive.

The main purpose of the present paper is to further the scope of the pathwise viscosity solution theory to ``inhomogenous'' settings, by considering in particular the Neumann problem for \eqref{E:neumann} on a subdomain $\Omega \subset \RR^d$. To our knowledge, this is the first treatment of fully nonlinear SPDEs with boundary conditions (apart from the periodic boundary conditions on the torus, which may be recast as a problem on the whole space). In analogy to the setting of $x$-dependent Hamiltonians, many of the simplifications that are used in the analysis of the homogenous problem on the whole space are no longer available. We are able to prove well-posedness results that are, while new and sufficiently general for the applications we have in mind, somewhat more restrictive than in the standard, ``non-rough'' setting. The most notable such restriction is that we only treat convex $\Omega$. The Hamiltonian, meanwhile, require certain structural assumptions.

Boundary value problems for (deterministic) fully nonlinear equations are by now a classical topic. In particular, the definition of Neumann boundary conditions in the viscosity sense goes back to Lions \cite{Lions85}, who obtained a comparison principle for first-order equations. This was then extended by various authors to more general cases, such as second-order equations (in particular of geometric type) and fully nonlinear boundary conditions, see for instance \cite{Ishii91,GS93, Bar93,Bar99, IS04, Bou08}. 

The Neumann boundary condition for \eqref{E:neumann} is relevant for the case of geometric equations, in which case the level sets of the solutions $u$ model the motion of hypersurfaces with a prescribed normal velocity that is rough in time. The imposition of the boundary condition is then equivalent to requiring that the contact angle of the hypersurface with the boundary be a right angle. 
}

\subsection{Well-posedness results: method of proof and main difficulties} We define a notion of (sub/super) solution for the boundary value problem \eqref{E:neumann}. Following the Lions-Souganidis theory, we make use of specific test functions in order to deal with the singularity of $d\zeta$. Crucially, just as in the classical viscosity solution theory, the boundary condition has to be understood in a weak sense (see the discussion at the beginning of subsection \ref{subsec:definitions} below). We show that our definition is consistent with the classical one when $\zeta$ is $C^1$, and satisfies a stability property under half-relaxed limits.

We then proceed to prove well-posedness results for this notion of solution. The domain $\Omega$ is assumed to be $C^1$ and convex (see below for a discussion of this second assumption), and $F$ is taken to satisfy ``standard'' assumptions from the theory of viscosity solutions, which encompass, for example, the nonlinearity arising from the level-set equation for mean curvature flow. As is usual in the Lions-Souganidis theory of ``rough'' viscosity solutions, the Hamiltonian $H$ requires more assumptions than in the classical case. We assume either that $H \in C^2$ with $D^2 H$ bounded, or each $H^i$ is equal to a difference of convex functions, a condition which is introduced already in \cite{LSnonsmooth}. In the latter case, we also need some stronger conditions: either $\Omega$ is a half-space, the convexity of $\Omega$ satisfies a nondegeneracy condition (see \eqref{A:uniformlyconvexdomain} below), or $H$ is radial. A polynomial growth assumption for $H$ is needed as well; in this paper, we assume quadratic growth for simplicity, but this can be generalized (see Remark \ref{R:growth} below).

We then prove the following:
\begin{itemize}

\item Comparison: if  $u$ and $v$ are respectively a sub- and super-solution of \eqref{E:neumann}, then
	\[		\sup_{(x,t) \in \oline{\Omega} \times [0,T]} \Big\{ u(x,t) - v(x,t)\Big\} \le \sup_{x \in \oline{\Omega}} \Big\{ u(x,0) - v(x,0) \Big\}. \]
\item Continuity with respect to the noise: if, for $n \in \NN$, $(u^n)_{n \in \NN}$ are solutions of \eqref{E:neumann} driven by given signals $\zeta^n \in C([0,T],\RR^m)$ which converge uniformly, as $n \to \oo$, to $\zeta \in C([0,T],\RR^m)$, then $u^n \xrightarrow{n \to \oo} u$ locally uniformly, where $u$ is a solution of \eqref{E:neumann} corresponding to $\zeta$.
\end{itemize}

These statements imply the existence of a unique solution to the initial value problem for \eqref{E:neumann} for arbitrary continuous $\zeta$ and initial datum $u_0$. As is usual in the viscosity solution theory, the proof of the comparison principle also yields some additional information on the solutions. For instance, we also obtain in the case where $F$ is independent of $(u,x,t)$:

\begin{itemize}
\item Spatial continuity: if $u_0$ is uniformly continuous, any solution $u$ of \eqref{E:neumann}  is uniformly continuous in $x$, with a modulus which does not depend on $\zeta$. 

\item Monotonicity in the path variable: if the $H^i$ are convex and, $\zeta^1(0)=\zeta^2(0)$ and $\zeta^1 \le \zeta^2$ on $[0,T]$, then the corresponding solutions satisfy $u^1 \leq u^2$ on $\Omega \times [0,T]$. 
\end{itemize}

{
In order to explain the specific difficulties and the need for our assumptions, due to the presence of both the boundary condition and the irregularity of $\zeta$, we describe now the strategy of proof for the comparison principle. As usual in the viscosity solution theory, we double the variables and apply maximum principle arguments to a quantity of the form
\[\sup_{x,y \in \Omega } \left\{ u(x,t) - v(y,t) - \Phi(x,y,t) \right\}   \]
for well-chosen test functions $\Phi$. The fact that $\zeta$ is irregular means there is little flexibility in choosing $\Phi$, and, in order to ``cancel out'' the rough term in the dynamics, $\Phi$ must be smooth in the $(x,y)$-variable and a solution of the doubled equation
\begin{equation}\label{introdoubled}
d\Phi = \sum_{i=1}^m \left( H^i\left(D_x \Phi\right) - H^i \left(- D_y \Phi\right) \right) \cdot d\zeta^i.
\end{equation}

The family of test functions that we use are indexed by two small parameters $\delta,\eps > 0$, which each play a different role. We will have, roughly,
\begin{equation}\label{Phibehavior}
	\Phi(x,y,t) \approx \frac{|x-y|^2}{2\delta} + \eps d_\Omega(x) + \eps d_\Omega(y),
\end{equation}
where $d_\Omega$ is the signed distance to $\Omega$. Thus, $\delta$ corresponds to a penalization outside of the diagonal, and the $\varepsilon$-term ensures that the test function is a strict super-solution of the boundary condition on $\partial \Omega$. A large part of this work is devoted to constructing a test function $\Phi$ that solves \eqref{introdoubled}, while also behaving as in \eqref{Phibehavior} on a sufficiently long time interval.

The main difficulty is that the time interval on which the desired constraints hold shrinks with $\delta$, and, therefore, in the well-posedness proofs, we take first $\eps \to 0$ and then $\delta \to 0$. This is contrary to the classical setting, where the parameters are balanced in some way. 
It is for this reason that we must assume $\Omega$ is convex, and why we can only deal with the case of the homogeneous Neumann boundary condition. Note, in particular, that, if $\Omega$ is convex, then $x \mapsto |x-y|^2/2\delta$ is always a super-solution of the homogenous Neumann boundary condition.

The construction of the test functions $\left(\Phi_{\delta, \varepsilon}\right)_{\delta,\eps > 0}$ is determined by the precise assumptions on the Hamiltonian $H$. When $H \in C^2$, their analysis, via the method of characteristics, is rather straightforward (see Section \ref{sec:smoothH}). 

In the case when each $H^i$ is a difference of convex functions, we use $C^{1,1}$-regular test functions defined by Hopf-type formulae, in which case proving that they satisfy all the properties that we need requires a more involved analysis. To explain the difficulty, let us assume for simplicity here that $m = 1$ and $H$ is convex. Then a candidate for the test function $\Phi_{\delta,\eps}$, defined for $t$ in a neighborhood of some fixed $t_0 \in [0,T]$, is
\begin{equation}\label{Phicandidate}
\begin{split}
	\sup_{p,u,v \in \RR^d}& \left\{ (p+u) \cdot x - (p-v) \cdot y - \delta \frac{|p|^2}{2} - \eps \psi^* \left( \frac{u}{\eps} \right) - \eps \psi^* \left( \frac{v}{\eps} \right) \right.\\
	&+ \left. (\zeta_t - \zeta_{t_0}) \left( H(p + u) - H(p - v) \right) - \gamma \left( H(p + u) + H(p - v)\right) \right\}.
\end{split}
\end{equation}
Here, $\gamma > 0$ is some small parameter and $\psi$ is a certain uniformly convex regularization of the signed distance function to $\Omega$ (see \eqref{psi} below). 

In view of the uniform concavity in $(p,u,v)$ of the function inside the supremum, one can check that \eqref{Phicandidate} is $C^{1,1}$ in $(x,y)$ and solves \eqref{introdoubled} as long as $|\zeta_t - \zeta_{t_0}| \le \gamma$. The main difficulty is in proving that this test function is a supersolution of the boundary condition, for an appropriate choice of $\gamma$ depending on $\delta$. As it turns out, however, this may not necessarily be the case, even when $\zeta_t = \zeta_{t_0}$. 

We overcome this issue by ``pushing'' the supremum $(p,u,v)$ attained in \eqref{Phicandidate} in a convenient direction. This is done by adding to $|p|^2/2$ some convex function $\ell(p)$. The precise choice of $\ell(p)$ depends on the more particular cases described above, that is, when $\Omega$ is a half space or nondegenerately convex, or when $H$ is radial. The precise construction of the test functions, and the proofs of their specific properties, are laid out in Section \ref{sec:testfn}.
}

\subsection{Long-time behavior for stochastic mean curvature flow}
As an application of our comparison and stability estimates, we study the qualitative long-time behaviour of a stochastically perturbed mean-curvature flow in a convex cylinder-like domain $\Omega = D \times \RR$ . Namely, we consider a bounded hypersurface $\Gamma(t) \subset \Omega$ which satisfies a right angle boundary condition on $\partial \Omega$ and evolves according to the normal velocity
$$ V(t) = - \kappa + dB(t)$$
where $\kappa$ is the mean curvature and $B=(B(t))_{t \geq 0}$ is a scalar Brownian motion. We show that, for large times, $\Gamma$ is arbitrarily close in the Hausdorff distance to a (possibly fat) hyperplane of the form $D \times [a+B(t), b+B(t)]$ (see Section \ref{sec:mcf} below for the precise results). This result is an extension of a deterministic result obtained by Giga, Ohnuma and Sato \cite{GOS99}. In fact, our proof crucially relies on their result, which we combine with the stability properties of the level set PDE to obtain our convergence (in particular, both the spatial modulus and monotonicity results are needed). { Note that the presence of the noise term allows to slightly improve their result, since they only obtained convergence in Hausdorff distance under an additional condition. }Similar results have also been obtained in the stochastic PDE literature in the simpler case of periodic graphs and by different methods in \cite{ER12,DHR21}.

\subsection{Organization of the paper}

In Section \ref{sec:prelim}, we introduce some preliminary objects, state the main assumptions, and define the notion of viscosity solution to \eqref{E:neumann}. Section \ref{sec:smoothH} is devoted to the proof of the comparison principle when $H$ is smooth ($C^2$). The case of non-smooth $H$ is dealt with in the following sections: in Section \ref{sec:testfn}, we define a family of test functions and derive their properties, and in Section \ref{sec:nonsmoothwellposed} we use these results to prove the well-posedness theorems. The case of geometric equations requires a slightly different family of test functions to deal with the singularity in the second-order part of the equation, and is treated in Section \ref{sec:geo}. In Section \ref{sec:mcf} we apply our results to the long-time behaviour of a stochastically perturbed mean curvature flow in a domain. Section \ref{sec:cl} contains a discussion of some possible extensions and open questions. Finally, in the Appendix, we prove some auxiliary results from the pathwise viscosity solution theory used throughout the paper, such as consistency and stability properties.

\subsection{Notation}
We denote by $\cdot$ the usual scalar product  in $\RR^d$ and $|\cdot|$ the associated Euclidean norm.

$\mbb S^d$ is the space of $d\times d$ symmetric matrices, equipped with the norm $\norm{X} := \max_{|v| = 1} |Xv \cdot v|$. $\Id$ is the identity matrix (its dimension will be clear from the context).

Given a convex function $\psi: \RR^d \to \RR$, we let $\partial \psi(x)$ be its sub-differential at a point $x \in \RR^d$, and $\psi^{\ast}(p) := \sup_{x \in \RR^d} \left( \left \langle p, x \right\rangle - \psi(x) \right)$ its convex conjugate.

We let $D f, D^2 f$ denote respectively the gradient and Hessian of a function $f$. If the arguments of $f$ are $(x,t)$, $Df$ and $D^2 f$ will always be understood as being only w.r.t. the $x$ variable.

$\|\cdot\|_{\infty,A}$ is the supremum norm for functions defined over a set $A$.

$(B)USC(A)$, $(B)LSC(A)$, $(B)UC(A)$ are respectively (bounded) upper semi-continuous, lower semi-continuous and uniformly continuous from $A$ to $\RR$. $F^\ast$, $F_{\ast}$ are the upper and lower semi-continuous envelope of a function $F$ (there is a minor clash of notations with the convex conjugate, which one we mean will be clear from the context)  .

We let $C_0([0,T],\RR^m) =\left\{ \zeta \in C([0,T],\RR^m) \mbox{ s.t. }\zeta(0)=0\right\}$.

\section{Preliminaries}\label{sec:prelim}

\subsection{The domain}


We assume throughout that
\begin{equation}\label{A:C1convexdomain}
	\left\{
	\begin{split}
	&\Omega \subset \RR^d \quad \text{is convex and open with a $C^1$-boundary, and }\\
	&\text{the outward normal $n(x)$ is uniformly continuous in $x \in \partial \Omega$.}
	\end{split}
	\right.
\end{equation}
At times, we will impose a stronger convexity assumption on $\Omega$ (see \eqref{A:uniformlyconvexdomain} for instance).

In view of \eqref{A:C1convexdomain}, there exists $\psi: \RR^d \to \RR$ such that
\begin{equation}\label{psi}
	\left\{
	\begin{split}
	&\psi \in C^2(\RR^d) \text{ is convex, } \norm{D\psi}_{\oo} + \norm{D^2 \psi}_{\oo} < \oo,\\
	&\inf \psi = -1, \quad \lim_{|x| \to +\oo} \psi(x) = +\oo, \quad \text{and }\\
	&D\psi(x) \cdot n(x) \ge 1 \quad \text{for all } x \in \partial \Omega.
	\end{split}
	\right.
\end{equation}
Indeed, let $\rho: \RR^d \to \RR$ be the signed distance function to $\Omega$, negative in the interior of $\Omega$ and positive in the exterior. Then we may take $\psi$ to be a regularization (by convolution with a standard mollifier, for example) of the function $2(\rho \vee -1)$. If $\Omega$ is unbounded, the penalization property for large $x$ can be ensured by adding a term of the form $\eps \sqrt{1 + |x|^2}$, with $\eps$ taken sufficiently small that the desired inequality on $\partial \Omega$ is still satisfied by $\psi$.

The global bounds on the gradient and Hessian of $\psi$ in \eqref{psi} imply that its convex conjugate satisfies
\begin{equation}\label{psistar}
	\left\{
	\begin{split}
	&\psi^* \text{ is finite only in a fixed compact set $K \subset \RR^d$, and, for some fixed $\kappa > 0$,}\\
	&D^2\psi^*  \ge \kappa \Id \text{ in the sense of distributions on $K$.}
	\end{split}
	\right.
\end{equation}

{
\subsection{The nonlinearity $F$}

The second order dependence in the equations will be determined by a function $F$ satisfying
\begin{equation}\label{A:Fcts}
	F \in C(\mbb S^d, \RR^d , \RR , \RR^d, [0,T]) \text{ is uniformly continuous in } (r,x,t) \text{ for bounded } (X,p),		
\end{equation}
\begin{equation}\label{A:Fmonotone}
	r \mapsto F(X,p,r,x,t) \text{ is non-increasing for all } (X,p,x,t) \in \mbb S^d \times \RR^d \times \RR^d \times [0,T],
\end{equation}
and
\begin{equation}\label{A:F}
	\left\{
	\begin{split}
	&\text{for all $R,C > 0$, there exists $\omega_{R,C}: [0,\oo) \to [0,\oo)$ such that $\lim_{r \to 0^+} \omega_{R,C}(r) = 0$}\\
	&\text{and, for all $p,q \in \RR^d$, $-R \le r \le s \le R$, $x,y \in \RR^d$, $t \in [0,T]$, $\delta,\eta > 0$,}\\
	&\text{and $X,Y \in \mbb S^d$ satisfying}\\
	&- \frac{C}{\delta}
		\begin{pmatrix}
			\Id & 0 \\
			0 & \Id
		\end{pmatrix}
		\le
		\begin{pmatrix}
			X & 0 \\
			0 & - Y
		\end{pmatrix}
		\le \frac{C}{\delta}
		\begin{pmatrix}
			\Id & -\Id \\
			-\Id & \Id
		\end{pmatrix}
		+ \eta
		\begin{pmatrix}
			\Id & 0 \\
			0 & \Id
		\end{pmatrix}, \\
		&\text{we have } F(X,p,r,x,t) - F(Y,q,s,y,t)\\
		&\quad \le \omega_{R,C}\left( \frac{|x-y|^2 + |p-q|^2}{\delta} + (1 + |p| + |q|)|x-y| + (s-r) + \eta\right).
	\end{split}
	\right.
\end{equation}
The condition \eqref{A:F} is a general assumption on the coupled dependence of $F$ on the derivative variables $(X,p)$ and the environment variables $(x,t)$. Such a requirement is usual in the theory of fully nonlinear second order equations; cf. \cite[(3.10), (3.14)]{CIL}. One difference is that the modulus in \eqref{A:F} accounts for different values in the gradient variable $p$.

We note our assumptions encompass all ``standard'' examples of nonlinearities that we list now, referring to \cite{CIL} and the references therein for details. For instance, \eqref{A:F} is clearly satisfied if $F$ is independent of $p$ and $x$, nondecreasing in $X \in \mbb S^d$, and nonincreasing in $u$, or, more generally for $F$ with a separated dependence, that is, if $F$ is given, for $F_0$ and $F_1$ satisfying appropriate conditions, by
\[
	F(X,p,r,x,t) = F_1(X,r,t) + F_0(p,r,x,t).
\]
A more nontrivial example of coupling is given by
\[
	F(X,p,x,t) = \tr[a(p,x,t)X],
\]
where $a = \sigma \sigma^t$ for some $\sigma \in C^{0,1}(\RR^d \times \RR^d; \RR^{d \times m})$. Finally, other examples can be generated by observing that \eqref{A:Fcts} - \eqref{A:F} are closed under the $\inf$ and $\sup$ operattions. That is, if $(F_{\alpha\beta})_{\alpha \in \mcl A, \, \beta \in \mcl B}$ satisfy the assumptions uniformly in $\alpha$ and $\beta$, then the same is true for
\[
	\inf_{\alpha \in \mcl A} \sup_{\beta \in \mcl B} F_{\alpha\beta} \quad \sup_{\beta \in \mcl B} \inf_{\alpha \in \mcl A} F_{\alpha\beta}.
\]

Later in the paper, motivated by geometric examples in which $F$ has a singularity at $p = 0$, we allow for $F$ to be discontinuous at $p = 0$. Although general dependence on all variables can be treated with similar assumption to those above, for brevity and simplicity of presentation, we focus then on $F$ depending only on $X$ and $p$, which still allows for the treatment of the nonlinearity arising from perturbed mean curvature flow, i.e.
\[
	F(X,p) = \tr\left[ \left(\Id - \frac{p \otimes p}{|p|^2} \right)X \right].
\]
}

\subsection{Definition of solutions}\label{subsec:definitions}

{
An important feature in the study of fully nonlinear equations on domains is the need to understand the boundary conditions in a weak sense. Indeed, consider the initial value problem
\begin{equation}\label{E:sampleboundary}
	\begin{dcases}
		\del_t u = |\del_x u| & \text{in }(-1,1) \times [0,T],\\
		\del_x u \cdot \nu = 0 & \text{on } \del(-1,1) \times [0,T], \text{ and}\\
		u(x,0) = x & \text{in } (-1,1)
	\end{dcases}
\end{equation}
(here $\nu(1) = 1$ and $\nu(-1) = -1$). It turns out that the unique viscosity solution (obtained, for instance, from a vanishing viscosity limit) is given by $u(x,t) = \min\{x + t, 1 \}$, which does not satisfy the Neumann condition at the left endpoint $-1$. Instead, on the boundary $\{-1,1\}$, $u$ satisfies the subsolution property
\[
	\min\left\{ \del_t u - |\del_x u|, \del_x u \cdot \nu \right\}  \le 0
\]
and the supersolution property
\[
	\max \left\{ \del_t u - |\del_x u|, \del_x u \cdot \nu \right\} \ge 0.
\]
Observe that, although $\del_x u(-1,t)\cdot \nu(-1) = -1 < 0$, the boundary supersolution property is still satisfied, because $(\del_t u - |\del_x u|)(-1,t) = 0$.

The nature of the definition of pathwise viscosity solutions rests on appropriate classes of test functions, which, in turn, is determined by the regularity of the Hamiltonian $H$. 
}

\subsubsection{Smooth Hamiltonians}

We shall first assume that
\begin{equation}\label{A:HC2}
	H \in C^2(\RR^d;\RR^m) \quad \text{and} \quad \sup_{p \in \RR^d} |D^2 H(p)| < \oo.
\end{equation}
The method of characteristics then yields, for any $\phi \in C^2(\RR^d)$, a number $h > 0$, depending only on $\norm{D^2H}_\oo$ and $\norm{D^2\phi}_\oo$, and, for any $t_0 > 0$, a solution $\Phi \in C((t_0 - h, t_0+h) \cap [0,T] , C^2(\RR^d))$ of 
\begin{equation}\label{E:local}
	d\Phi = \sum_{i=1}^m H^i(D\Phi) \cdot d\zeta^i \quad \text{in } \RR^d\times (t_0 - h, t_0 + h)
\end{equation}
satisfying $\Phi(\cdot,t_0) = \phi$. The solution of \eqref{E:local} is understood in the sense of continuous extension of the map $\zeta \mapsto \Phi$, and in fact, because $H$ is independent of $x$, $\Phi$ can be obtained by composing the solution operators for the Hamilton-Jacobi equations associated to the Hamiltonians $(H^i)_{i=1}^m$ with increments of the path $\zeta$ (see Lemma \ref{L:solutionops} below).

Following \cite{LSfirst}, we use test functions satisfying \eqref{E:local} in order to adapt the notion of viscosity solutions for the Neumann problem to the ``rough'' setting:
\begin{definition}\label{D:smoothH}
	A function $u \in USC( \oline{\Omega} \times [0,T])$ (resp. $LSC$) is a sub- (resp. super-) solution of \eqref{E:neumann} if, whenever $\mu \in C^1((0,T))$, $\Phi$ is a smooth-in-space solution of \eqref{E:local}, and $u(x,t) - \Phi(x,t) - \mu(t)$ achieves a local maximum (minimum) at some $(x_0,t_0) \in \bar{\Omega} \times (0,T]$, then
	\[
		\left\{
		\begin{split}
		&\mu'(t_0) - F(D^2\Phi(x_0,t_0), D\Phi(x_0,t_0), u(x_0,t_0),x_0,t_0) \le 0 \quad \text{if $x_0 \in \Omega$ and}\\
		&\min\Big\{ \mu'(t_0) - F(D^2\Phi(x_0,t_0), D\Phi(x_0,t_0), u(x_0,t_0),x_0,t_0), \\
		& \qquad \qquad \qquad D\Phi(x_0,t_0) \cdot n(x_0)\Big\} \le 0 \quad \text{if $x_0 \in \del \Omega$}
		\end{split}
		\right.
	\]
	(resp.
	\[
		\left\{
		\begin{split}
		 &\mu'(t_0) - F(D^2\Phi(x_0,t_0), D\Phi(x_0,t_0), u(x_0,t_0),x_0,t_0) \ge 0 \quad\text{if $x_0 \in  \Omega$ and}\\
		&\max\Big\{ \mu'(t_0) - F(D^2\Phi(x_0,t_0), D\Phi(x_0,t_0), u(x_0,t_0),x_0,t_0), \\
		& \qquad \qquad \qquad D\Phi(x_0,t_0) \cdot n(x_0)\Big\} \ge 0  \quad \text{if }x_0 \in \del \Omega.
		\end{split}
		\right.
	\]
	A function $u \in UC(\oline{\Omega} \times [0,T])$ is called a solution if it is both a sub- and super-solution.
\end{definition}

\subsubsection{Nonsmooth Hamiltonians}

We next relax \eqref{A:HC2} and assume that
\begin{equation}\label{A:DCH}
	\text{for each $i = 1,2,\ldots,m$, there exist convex $H^i_1,H^i_2: \RR^d \to \RR$ such that } H^i = H^i_1 - H^i_2.
\end{equation}
Even when $\Omega = \RR^d$, \eqref{A:DCH} is the weakest possible assumption for which \eqref{E:neumann} is well-posed for any given initial datum and path; see \cite{LSnonsmooth,Snotes}. 

For such Hamiltonians, it is no longer possible, in general, to find local-in-time solutions of \eqref{E:local} that are $C^2$ in space. However, through a different procedure, one can build local-in-time, $C^{1,1}$-in-space solutions (see Lemma \ref{L:C11solutions} below). If $H$ satisfies \eqref{A:DCH} and $F \equiv 0$, then Definition \ref{D:smoothH} can still be used to define solutions using such test functions. However, when $F$ is nontrivial, we are forced to generalize the definition.

We use the fact (see \cite[Theorem 7.1]{Snotes}) that, if $H$ satisfies \eqref{A:DCH}, then, for any $\phi \in UC(\RR^d)$, there exists a unique pathwise viscosity solution $\Phi \in UC(\RR^d \times [0,T])$, satisfying $\Phi(\cdot,0) = \phi$, of
\begin{equation}\label{E:pathwiseHJnonsmooth}
	d\Phi = \sum_{i=1}^m H^i(D\Phi) \cdot d\zeta^i \quad \text{in } \RR^d \times [0,T].
\end{equation}

Below, we also allow for discontinuous $F$. We do this to allow for the study of geometric equations in Section \ref{sec:geo}.

\begin{definition}\label{D:nonsmoothH}
	A function $u \in USC(\oline{\Omega} \times (0,T))$ is called a sub-solution of \eqref{E:neumann} if, whenever $\Phi \in UC(\RR^d \times (0,T))$ is a pathwise viscosity solution of \eqref{E:pathwiseHJnonsmooth}, the function
	\[
		\overline{u}(\xi,t) := \max_{x \in \oline{\Omega}} \left( u(x,t) - \Phi(x-\xi,t)\right), \quad (\xi,t) \in \RR^d \times (0,T)
	\]
	is a viscosity sub-solution of 
	\begin{align*}
		\sup_{x \in A^+_{\xi,t}} 
		\min \Big\{ &\overline{u}_t(\xi,t) - F^*\left(D^2_{\xi\xi} \oline{u}(\xi,t) , D_\xi \oline{u}(\xi,t),\oline{u}(\xi,t) + \Phi(x-\xi,t),x,t \right), \\
		&\qquad D_\xi \oline{u}(\xi,t) \cdot n(x) \Big\} \le 0, \quad (\xi,t) \in \RR^d \times (0,T),
	\end{align*}
	where $A^+_{\xi,t} := \argmax_{ \oline{\Omega}} \left( u(\cdot,t) - \Phi(\cdot-\xi,t)\right)$.
		
	A function $u \in LSC(\oline{\Omega} \times (0,T))$ is called a super-solution of \eqref{E:neumann} if, whenever $\Phi \in UC(\RR^d \times (0,T))$ is a pathwise viscosity solution of \eqref{E:pathwiseHJnonsmooth}, the function
	\[
		\underline{u}(\xi,t) := \min_{x \in \oline{\Omega}} \left( u(x,t) - \Phi(x-\xi,t)\right), \quad (\xi,t) \in \RR^d \times (0,T)
	\]
	is a viscosity super-solution of 
	\begin{align*}
		\inf_{x \in A^-_{\xi,t}}
		\max\Big\{ &\underline{u}_t(\xi,t) -  F_*\left(D^2_{\xi\xi} \underline{u}(\xi,t), D_\xi \underline{u}(\xi,t), \underline{u}(\xi,t) + \Phi(x - \xi,t),x,t\right),\\
		& \qquad D_\xi \uline{u}(\xi,t) \cdot n(x) \Big\} \ge 0,\quad (\xi,t) \in \RR^d \times (0,T),
	\end{align*}
	where $A^-_{\xi,t} := \argmin_{ \oline{\Omega}} \left( u(\cdot,t) - \Phi(\cdot-\xi,t)\right)$.
	
	Finally, $u \in UC(\oline{\Omega} \times [0,T])$ is called a solution if it is both a sub- and super-solution.
\end{definition}

\begin{remark}
	If $A^+_{\xi,t}$ or $A^-_{\xi,t}$ are empty, then the corresponding inequality is vacuous. Moreover, if, for fixed $(\xi,t) \in \RR^d \times [0,T]$, the set $A^+_{\xi,t} \cap \partial \Omega$ (resp. $A^-_{\xi,t} \cap \partial \Omega$) is empty, then the quantity $D\oline{u}(\xi,t) \cdot n(x)$ is understood to be $+\oo$ (resp. $-\oo$) in the verification of the sub-(resp. super-) solution inequality; that is, the boundary condition is not checked in that case.
\end{remark}

\subsection{Construction of test functions for non-smooth Hamiltonians}
We now describe the general method for constructing $C^{1,1}$-test functions for use in Definition \ref{D:nonsmoothH}, which is based on the Hopf formula for solutions of Hamilton-Jacobi equations; see \cite{LR}. This construction below is also used in \cite{Seegerschemes}. We also record an important Hessian bound for such solutions, which will be a key part of the proofs for second-order equations.

Throughout the paper, for $H$ satisfying \eqref{A:DCH}, we fix convex $H^i_1,H^i_2$ such that $H^i = H^i_1 - H^i_2$, and we define the convex function
\begin{equation}\label{Gfn}
	G := \sum_{i=1}^m \left( H^i_1 + H^i_2\right).
\end{equation}
Observe that
\begin{equation}\label{Gprop}
	G + \sum_{i=1}^m \theta^i H^i \quad \text{is convex whenever }\theta^1,\theta^2,\ldots,\theta^m \in [-1,1].
\end{equation}

\begin{lemma}\label{L:C11solutions}
	Assume that $\gamma > 0$, $A$ is a positive symmetric matrix, and $L: \RR^d \to \RR$ is such that $D^2 L \ge A$ in the sense of distributions. For $(x,\tau) \in \RR^d \times \RR^m$, define
	\[
		\Phi(x,\tau) := \sup_{p \in \RR^d} \left\{ p \cdot x - L(p) - \gamma G(p) + \sum_{i=1}^m \tau_i H^i(p) \right\}.
	\]
	Then $\Phi \in C^{1,1}(\RR^d \times (-\gamma,\gamma)^m)$, $0 \le D^2 \Phi \le A^{-1}$ in the sense of distributions, and
	\[
		\del_{\tau_i} \Phi(x,\tau) = H^i(D_x \Phi(x,\tau)) \quad \text{for } i = 1,2,\ldots,m, \quad (x,\tau) \in \RR^d \times (-\gamma,\gamma)^m.
	\]
\end{lemma}

{
When constructing particular test functions in Section \ref{sec:testfn}, the specific uniformly convex function $L$ is chosen so that $\Phi$ has similar behavior to the convex conjugate $L^*(x) = \sup_{p \in \RR^d} \left\{ p \cdot x - L(p) \right\}$.
}

\begin{proof}[Proof of Lemma \ref{L:C11solutions}]
	For $\tau \in (-\gamma,\gamma)^m$, set
	\[
		\tilde L(p,\tau) := L(p) + \gamma G(p) - \sum_{i=1}^m \tau_i H^i(p).
	\]
	Then, in view of \eqref{Gprop}, $\tilde L(\cdot,\tau)$ is uniformly convex and $D^2 \tilde L(\cdot,\tau) \ge A$. It follows that, for every $x \in \RR^d$, the supremum in the definition of $\Phi(x,\tau)$ is attained for a unique value $p(x,\tau) \in \RR^d$, which is equivalent to $x \in \del \tilde L(p(x,\tau),\tau)$, as well as $p(x,\tau) \in \del \Phi(x,\tau)$. The uniqueness of $p(x,\tau)$ them implies that $\Phi$ is differentiable in $x \in \RR^d$, with $D_x \Phi(x,\tau) = p(x,\tau)$. 
	
	Fix $x,\hat x \in \RR^d$. Then
	\begin{align*}
		\tilde L(p(\hat x,\tau),\tau) &\ge \tilde L(p(x,\tau),\tau) + x \cdot \left( p(\hat x,\tau) - p(x,\tau) \right) \\
		&+ \frac{1}{2} A\left(p(x,\tau) - p(\hat x,\tau)\right) \cdot \left(p(x,\tau) - p(\hat x,\tau)\right)
	\end{align*}
	and
	\begin{align*}
		\tilde L(p(x,\tau),\tau) &\ge \tilde L(p(\hat x,\tau),\tau) + \hat x \cdot \left( p(x,\tau) - p(\hat x,\tau) \right) \\
		&+ \frac{1}{2} A\left(p(x,\tau) - p(\hat x,\tau)\right) \cdot \left(p(x,\tau) - p(\hat x,\tau)\right).
	\end{align*}
	Adding the two inequalities yields
	\[
		 A\left(p(x,\tau) - p(\hat x,\tau)\right) \cdot \left(p(x,\tau) - p(\hat x,\tau)\right) \le \left( p(x,\tau) - p(\hat x,\tau) \right) \cdot (x - \hat x).
	\]
	The method of Lagrange multipliers gives, for any $X \in \RR^d$,
	\[
		\max \left\{ P \cdot X : P \in \RR^d \text{ and } AP \cdot P \le P \cdot X \right\} = A^{-1} X \cdot X,
	\]
	and therefore
	\[
		\left( p(x,\tau) - p(\hat x,\tau) \right) \cdot (x - \hat x) \le A^{-1} (x - \hat x) \cdot (x - \hat x).
	\]
	This implies that $D^2 \phi \le A^{-1}$, as desired.
	
	We next claim that $p$ is Lipschitz in $\tau$. Let $\theta > 0$ be the smallest eigenvalue of $A$. Fix $h \in \RR^m$. Then
	\[
		\tilde L(p(x,\tau +h), \tau) \ge \tilde L(p(x,\tau),\tau) + x \cdot (p(x,\tau+h) - p(x,\tau)) + \frac{\theta}{2} |p(x,\tau+h) - p(x,\tau) |^2
	\]
	and
	\begin{align*}
		\tilde L(p(x,\tau), \tau+h) &\ge \tilde L(p(x,\tau+h),\tau+h) + x \cdot (p(x,\tau) - p(x,\tau+h)) \\
		&+ \frac{\theta}{2} |p(x,\tau+h) - p(x,\tau) |^2.
	\end{align*}
	Adding the two inequalities gives
	\[
		\theta |p(x,\tau+h) - p(x,\tau)|^2 \le h \cdot \left[ H(p(x,\tau+h)) - H(p(x,\tau)) \right].
	\]
	We conclude that $p$ is locally Lipschitz in $\tau$, in view of the fact that $H$ is locally Lipschitz in $p$.
	
	Finally, the uniform convexity of $\tilde L$ and the definition of $\Phi$ imply
	\begin{align*}
		\Phi(x,\tau) &+ h \cdot H(p(x,\tau+h)) - \frac{\theta}{2} |p(x,\tau+h) - p(x,\tau)|^2
		\le \Phi(x, \tau+h)\\
		&\le \Phi(x,\tau) + h \cdot H(p(x,\tau+h)).
	\end{align*}
	Because $p$ is Lipschitz in $\tau$, this implies that $\del_\tau \Phi(x,\tau) = H(p(x,\tau)) = H(D_x \Phi(x,\tau))$.
\end{proof}

%
%

\section{Smooth Hamiltonians}\label{sec:smoothH}

Throughout this section, we assume that $\Omega$ satisfies \eqref{A:C1convexdomain}, $H$ satisfies \eqref{A:HC2}, and $\zeta \in C([0,T],\RR^m)$, and we study the well-posedness of the initial value problem for \eqref{E:neumann}. 

\subsection{A particular test function}\label{subsec:smoothtest}

For $i = 1,2,\ldots,m$, let $(S^i_+(\tau))_{\tau \in \RR}$ and $(S^i_-(t))_{t \in \RR}$ be the solution operators for respectively
\[
	\frac{\partial U}{\partial \tau} = H^i(D_x U) \quad \text{and} \quad \frac{\partial U}{\partial \tau} = - H^i(-D_y U),
\]
that is, $U(x,y,\tau) = S^i_{\pm}(\tau)U_0(x,y)$ solves the corresponding equation with $U(x,y,0) = U_0(x,y)$. Define also $S^i_d(\tau) = S^i_+(\tau)\circ  S^i_-(\tau)$.

By the method of characteristics, for any $\phi \in C^2(\RR^d \times \RR^d)$ with $\norm{D^2\phi}_{\oo} < \oo$, there exists $\tau_0 > 0$, depending only on $\norm{D^2\phi}_\oo$ and $\max_{i=1}^m \norm{D^2H^i}_\oo$, such that, for all $\tau \in (-\tau_0,\tau_0)$ and $i = 1,2,\ldots, m$, $S^i_{\pm}(\tau)\phi \in C^2(\RR^d \times \RR^d)$. Moreover, the Poisson bracket of any two of the Hamiltonians
\[
	\left\{ (p,q) \mapsto H^i(p), H^i(-q), i = 1,2,\ldots,m \right\}
\]
is $0$. This implies that the corresponding Hamiltonian flows commute, and, in particular, gives the following result.

\begin{lemma}\label{L:solutionops}
	The solution operators $S^1_{\pm}, S^2_{\pm}, \ldots, S^m_{\pm}$ all commute for smooth initial data and sufficiently short time. In particular, if $\phi \in C^2(\RR^d \times \RR^d)$, $\norm{D^2 \phi}_\oo < \oo$, and
	\[
		\Phi(x,y,\sigma,\tau) = \prod_{i=1}^m S^i_+(\sigma_i) S^i_-(\tau_i)\phi(x,y), \quad x,y \in \RR^d, \quad \sigma,\tau \in \RR^m,
	\]
	then there exists $\tau_0 > 0$, depending only on $\norm{D^2H}_\oo$ and $\norm{D^2\phi}_\oo$, such that, for $i = 1,2,\ldots,m$,
	\[
		\del_{\sigma_i} \Phi = H^i(D_x \Phi) \quad \text{and} \quad \del_{\tau_i} \Phi = -H^i(-D_y \Phi) \quad \text{in } \RR^d \times \RR^d \times (-\tau_0,\tau_0)^m \times (-\tau_0,\tau_0)^m.
	\]
\end{lemma}

We now define, for $x,y \in \RR^d$ and $\delta,\eps > 0$,
\[
	\psi_{\delta,\eps}(x,y) := \frac{1}{2\delta} |x-y|^2 + \eps \psi(x) + \eps \psi(y),
\]
where $\psi$ is as in \eqref{psi}, and, for $\sigma,\tau \in \RR^m$,
\begin{equation}\label{testfnsmoothH}
	\Phi_{\delta,\eps}(x,y,\sigma,\tau) := \prod_{i=1}^m S^i_+(\sigma_i)S^i_-(\tau_i) \psi_{\delta,\eps}(x,y).
\end{equation}

\begin{lemma}\label{L:testfunctionsmoothH}
	There exists $c_0 > 0$ depending only on $\norm{\psi}_{C^2}$ and $\norm{D^2 H}_\oo$ such that, if $\gamma_\delta = c_0 \delta$, then the following hold:
	\begin{enumerate}[(i)]
	\item\label{smoothtestfneqs} For all $0 < \delta,\eps < 1$, the function \eqref{testfnsmoothH} belongs to $C^2(\RR^d \times \RR^d \times (-\gamma_\delta,\gamma_\delta)^m \times (-\gamma_\delta,\gamma_\delta)^m)$, and, for all $i = 1,2,\ldots,m$,
	\[
		\left\{
		\begin{split}
		&\del_{\sigma_i} \Phi_{\delta,\eps} = H^i(D_x \Phi_{\delta,\eps}) \quad \text{and} \quad
		\del_{\tau_i} \Phi_{\delta,\eps} = -H^i(-D_y\Phi_{\delta,\eps}) \\
		& \text{in } \RR^d \times \RR^d \times (-\gamma_\delta,\gamma_\delta)^m \times (-\gamma_\delta,\gamma_\delta)^m.
		\end{split}
		\right.
	\]
	\item\label{smoothpenalizing} There exists $C > 0$, independent of $\delta$ or $\eps$, such that, for all $(x,y,\sigma) \in \RR^d \times \RR^d \times (-\gamma_\delta,\gamma_\delta)^{m}$,
	\[
		\Phi_{\delta,\eps}(x,y,\sigma,\sigma) \ge -C\eps + \frac{|x-y|^2}{2\delta} \quad \text{and} \quad
		\Phi_{\delta,\eps}(x,y,\sigma,\sigma) \ge -C\eps + \eps \psi(x) + \eps \psi(y).
	\]
	\item\label{smoothepsto0} If $r_\eps > 0$ satisfies $\lim_{\eps \to 0} r_\eps = 0$ and $R > 0$, then, for fixed $\delta > 0$,
	\begin{align*}
		\lim_{\eps \to 0} \sup \Bigg\{ \left| \Phi_{\delta,\eps}(x,y,\sigma,\sigma) - \frac{|x-y|^2}{2\delta} \right| &: \sigma \in (-\gamma_\delta,\gamma_\delta)^m, \; |x-y| \le R, \\
		& \eps \psi(x) + \eps \psi(y)  \le r_\eps \Bigg\} = 0.
	\end{align*}
	\item\label{smoothtestfnboundary} There exists $c_1 > 0$ such that, for all $\sigma \in (-\gamma_\delta,\gamma_\delta)$, $0 < \delta,\eps < 1$, and $x,y \in \oline{\Omega}$ with $|x-y| \le 1$,
	\[
		\left\{
		\begin{split}
		&D_x \Phi_{\delta,\eps}(x,y,\sigma,\sigma) \cdot n(x) \ge c_1 \eps \quad \text{if } (x,y) \in \partial \Omega \times \oline{\Omega} \quad \text{and}\\
		&D_y \Phi_{\delta,\eps}(x,y,\sigma,\sigma) \cdot n(y) \ge c_1 \eps \quad \text{if } (x,y) \in \oline{\Omega} \times \partial \Omega.
		\end{split}
		\right.
	\]
	\end{enumerate}
\end{lemma}

\begin{remark}\label{R:xydistance}
	In the proof of the comparison principle below, the test function will be evaluated at points where $x$ and $y$ are close, depending on $\delta$. Therefore, for $\delta$ sufficiently small, the condition $|x-y| \le 1$ is easily satisfied. 
\end{remark}

\begin{proof}[Proof of Lemma \ref{L:testfunctionsmoothH}]
	The regularity claimed in part \eqref{smoothtestfneqs} is a consequence of the method of characteristics and the fact that $\norm{D^2 H}_{\oo} < \oo$ and, for $0 < \eps < 1$,
	\[
		\norm{D^2 \psi_{\delta,\eps}}_{\oo} \le \frac{1}{\delta} + 2\norm{D^2 \psi}_{\oo}.
	\]
	The fact that $\Phi_{\delta,\eps}$ satisfies the equations in part \eqref{smoothtestfneqs} follows from Lemma \ref{L:solutionops}. In particular, if we set
	\[
		\tilde \Phi(x,y,\sigma) := \Phi_{\delta,\eps}(x,y,\sigma,\sigma),
	\]
	then, for all $i = 1,2,\ldots,m$,
	\[
		\frac{\del \tilde \Phi}{\del \sigma_i}  = H^i(D_x \tilde \Phi) - H^i(-D_y \tilde \Phi).
	\]
	The inequalities in part \eqref{smoothpenalizing} are then an easy consequence of the comparison principle and the fact that $|x-y|^2/2\delta$ exactly solves the equation.
	
	To prove the next two parts, we look more closely at the formula from the solution arising from the method of characteristics. Namely, for any fixed $x,y \in \oline{\Omega}$, there exist unique $\tilde x, \tilde y \in \RR^d$ such that
	\[
	x = \tilde x - \sum_{i=1}^m DH^i\left( \frac{\tilde x - \tilde y}{\delta} + \eps D\psi(\tilde x)\right)\sigma_i,
	\quad
	y = \tilde y - \sum_{i=1}^m DH^i\left(\frac{\tilde x - \tilde y}{\delta} - \eps D\psi(\tilde y)\right)\sigma_i,
	\]
	\begin{align*}
		D_x \Phi_{\delta,\eps}(x,y,\sigma,\sigma) &= D_x \psi_{\delta,\eps}(\tilde x, \tilde y)
		= \frac{\tilde x - \tilde y}{\delta} + \eps D\psi(\tilde x),\\
		D_y \Phi_{\delta,\eps}(x,y,\sigma,\sigma) &= D_y \psi_{\delta,\eps}(\tilde x, \tilde y) = - \frac{\tilde x - \tilde y}{\delta} + \eps D\psi(\tilde y),
	\end{align*}
	and
	\begin{align*}
		&\Phi_{\delta,\eps}(x,y,\sigma,\sigma) = \psi_{\delta,\eps}(\tilde x, \tilde y) + \sum_{i=1}^m \sigma_i \bigg[ H^i\left( D_x \psi_{\delta,\eps}(\tilde x,\tilde y) \right) - H^i\left(- D_y \psi_{\delta,\eps}(\tilde x, \tilde y)\right) \\
		& \qquad - DH^i\left( D_x \psi_{\delta,\eps}(\tilde x, \tilde y) \right) \cdot D_x \psi_{\delta,\eps}(\tilde x, \tilde y) +  DH^i\left( D_y \psi_{\delta,\eps}(\tilde x, \tilde y) \right) \cdot D_y \psi_{\delta,\eps}(\tilde x, \tilde y)\bigg].
	\end{align*}
	If, as in part \eqref{smoothpenalizing}, $|x-y| \le R$ and $\eps \psi(x) + \eps \psi(y) \le r_\eps$, then, for some $C = C_\delta > 0$,
	\[
		|\tilde x - \tilde y| \le R + C\eps, \quad |(\tilde x - \tilde y) - (x-y)| \le C \eps, \quad \text{and} \quad - 2\eps \le \eps \psi(\tilde x) + \eps \psi(\tilde y) \le r_\eps + C \eps.
	\]
	Using this in the formula for $\Phi_{\delta,\eps}(x,y,\sigma,\sigma)$ then gives the desired limit.
	
	Suppose now that, as in part \eqref{smoothtestfnboundary}, we assume $x \in \del \Omega$ and $y \in \oline \Omega$ with $|x-y| \le 1$, and let $\sigma \in (-\gamma_\delta,\gamma_\delta)^m$. Then, for some constant $C$ depending on $\norm{D\psi}_\oo$ and the local bounds for $DH$ and $D^2 H$,
	\[
		|x - \tilde x| \vee |y - \tilde y| \le \frac{C \gamma_\delta}{\delta}
		\quad \text{and} \quad
		|x - y - (\tilde x- \tilde y)| \le \frac{C\gamma_\delta \eps}{\delta}.
	\]
	The property \eqref{psi} and the convexity of $\Omega$ yield, for some further $\tilde C > 0$,
	\begin{align*}
		D_x \Phi_{\delta,\eps}(x,y,\sigma,\sigma) \cdot n(x) 
		\ge \frac{x -y}{\delta} \cdot n(x) + \eps D\psi(x) \cdot n(x) - \frac{\tilde C\gamma_\delta \eps}{\delta}
		\ge \eps \left( 1 - \tilde C c_0\right).
	\end{align*}
	Shrinking $c_0$ further if necessary gives the result for $D_x\Phi_{\delta,\eps}$, and the argument for $D_y \Phi_{\delta,\eps}$ is similar.
\end{proof}

\subsection{The comparison principle: first order equations}

We now prove a comparison principle for the first order problem
\begin{equation}\label{E:smoothHJ}
	\begin{dcases}
		du = \sum_{i=1}^m H^i(Du) \cdot d\zeta^i & \text{in } \Omega \times (0,T], \\
		Du \cdot n = 0 & \text{on } \del \Omega \times (0,T].
	\end{dcases}
\end{equation}

\begin{theorem}\label{T:firstorderconvexdomain}
	Let $u \in BUSC(\oline{\Omega} \times [0,T])$ and $v \in BLSC(\oline{\Omega} \times [0,T])$ be respectively a bounded sub- and super-solution of \eqref{E:smoothHJ} in the sense of Definition \ref{D:smoothH}. Then
	\[
		\sup_{(x,t) \in \oline{\Omega} \times [0,T]} \left( u(x,t) - v(x,t)\right) \le \sup_{x \in \oline{\Omega}} \left( u(x,0) - v(x,0)\right).
	\]
\end{theorem}

An important role in the proof of Theorem \ref{T:firstorderconvexdomain} is played by the doubled equation
\begin{equation}\label{E:doubled}
	d\Phi = \sum_{i=1}^m \left( H^i(D_x \Phi) - H^i(-D_y \Phi) \right) \cdot d\zeta^i \quad \text{in } \RR^d \times \RR^d \times [0,T].
\end{equation}

\begin{lemma}\label{L:doubled}
	Assume that $u\in USC(\oline{\Omega} \times [0,T])$ and $v \in LSC(\oline{\Omega} \times [0,T])$ are respectively a sub- and super-solution of \eqref{E:smoothHJ}. Fix $0 \le a < b \le T$ and let $\Phi \in C([a,b],C^2(\oline{\Omega}))$ be a solution of \eqref{E:doubled} such that
	\begin{equation}\label{penalize}
		\lim_{|x-y| \to +\oo} \min_{t \in [a,b]} \Phi(x,y,t) = +\oo,
	\end{equation}
	\begin{equation}\label{boundedx-y}
		\begin{split}
		\sup_{x,y \in \oline{\Omega}}& \left( u(x,t) - v(y,t) - \Phi(x,y,t) \right)\\
		&> 
		\sup_{x,y \in \oline{\Omega}, \; |x-y| > 1/2} \left( u(x,t) - v(y,t) - \Phi(x,y,t) \right) \quad \text{for all } t \in [a,b],
		\end{split}
	\end{equation}
	and, uniformly for $t \in [a,b]$ and $x,y \in \oline{\Omega}$ with $|x-y| \le 1$,
	\begin{equation}\label{boundary}
		D_x \Phi(x,y,t) \cdot n(x) > 0  \quad \text{if } x \in \del \Omega \quad \text{and} \quad
		D_y \Phi(x,y,t) \cdot n(y) > 0 \quad \text{if } y \in \del \Omega.
	\end{equation}
	Then $[a,b] \ni t \mapsto \sup_{x,y \in \oline{\Omega}} \left( u(x,t) - v(y,t) - \Phi(x,y,t) \right)$ is nonincreasing.
\end{lemma}

\begin{proof}
	Fix $t_0 \in (a,b]$ and set $\phi := \Phi(\cdot,\cdot,t_0)$. Then, by Lemma \ref{L:solutionops}, for $t$ sufficiently close to $t_0$,
	\[
		\Phi(\cdot,\cdot,t) = \prod_{i=1}^m S^i_d(\zeta^i_t - \zeta^i_{t_0})\phi = \prod_{i=1}^m S^i_+(\zeta^i_t - \zeta^i_{t_0})S^i_-(\zeta^i_t - \zeta^i_{t_0})\phi = \prod_{i=1}^m S^i_-(\zeta_t - \zeta_{t_0})S^i_+(\zeta_t - \zeta_{t_0})\phi.
	\]
	Define
	\[
		\Psi(x,y,s,t) := \sum_{i=1}^m S^i_+(\zeta^i_s - \zeta^i_{t_0})S^i_-(\zeta^i_t - \zeta^i_{t_0})\phi(x,y).
	\]
	By Lemma \ref{L:solutionops}, there exists $\nu > 0$ sufficiently small such that $\Psi$ is $C^2$ in $x$ and $y$ for $s,t \in (t_0 - \nu, t_0 + \nu)$.
	
	To establish the result, it suffices to show that
	\begin{equation}\label{nonincreasingclaim}
		\left\{
		\begin{split}
		&(s,t) \mapsto \sup_{x,y \in \oline{\Omega}} \left( u(x,s) - v(y,t) - \Psi(x,y,s,t) \right)\\
		&\text{is nonincreasing in both } s,t \in (t_0 - \nu, t_0 + \nu) \text{ if $\nu$ is sufficiently small.}
		\end{split}
		\right.
	\end{equation}
	By \eqref{boundedx-y}, if $\nu$ is sufficiently small, then the supremum in \eqref{nonincreasingclaim} may be restricted to $|x-y| \le 1$. It then suffices to show that, for fixed $(y,t) \in \oline{\Omega} \times (t_0 - \nu, t_0 + \nu)$ with $|x-y| \le 1$,
	\[
		s \mapsto \sup_{x \in \oline{\Omega}} \left( u(x,s) - \Phi(x,y,s,t) \right) \quad \text{is nonincreasing}
	\]
	and, for fixed $(x,s) \in \oline{\Omega} \times (t_0 -\nu,t_0 + \nu)$ with $|x-y| \le 1$,
	\[
		t \mapsto \inf_{y \in \oline{\Omega}} \left( v(y,t) + \Phi(x,y,s,t) \right) \quad \text{is nondecreasing}.
	\]
	We note that the above supremum and infimum are both attained, in view of \eqref{penalize}. Also, if $\nu$ is sufficiently small, then the same boundary behavior in \eqref{boundary} is satisfied. The claim then follows from Definition \ref{D:smoothH}.
\end{proof}

\begin{proof}[Proof of Theorem \ref{T:firstorderconvexdomain}]
	Assume by contradiction that the claim is false. Then, for sufficiently small $\delta > 0$ and $\mu > 0$,
	\[
		[0,T] \ni t \mapsto \sup_{x,y \in \oline{\Omega} } \left( u(x,t) - v(y,t) - \frac{1}{2\delta} |x - y|^2\right) - \mu t
	\]
	attains a maximum at some $t_0 \in (0,T]$.
	
	Because $(1/2\delta)|x-y|^2$ solves \eqref{E:doubled}, we have, for all $x,y, t$,
	\[
		\frac{1}{2\delta} |x-y|^2 = \prod_{i=1}^m S^i_d(\zeta^i_t - \zeta^i_{t_0})\left( (1/2\delta)|\cdot - \cdot|^2\right)(x,y).
	\]
	Let $\Phi_{\delta,\eps}$ be defined as in \eqref{testfnsmoothH}, and let $h_\delta > 0$ be such that
	\[
		\max_{|t - t_0| < h_\delta} \max_{i=1,2,\ldots,m} \left| \zeta^i_t - \zeta^i_{t_0} \right| < \gamma_\delta,
	\]
	where $\gamma_\delta$ is as in Lemma \ref{L:testfunctionsmoothH}. The inequalities in Lemma \ref{L:testfunctionsmoothH}\eqref{smoothpenalizing} imply that, for any $t \in (t_0 - h_\delta, t_0 + h_\delta)$, there exists a maximum point $(x_\eps,y_\eps) \in \oline{\Omega} \times \oline{\Omega}$ of
	\[
		(x,y) \mapsto u(x,t) - v(y,t) - \Phi_{\delta,\eps}(x,y,\zeta_t - \zeta_{t_0}, \zeta_t - \zeta_{t_0}) 
	\]
	such that, if $\delta$ is sufficiently small, then
	\begin{align*}
		&\sup_{x,y \in\oline{\Omega}, \; |x-y| > 1/2} \big\{ u(x,t) - v(y,t) - \Phi_{\delta,\eps}(x,y,\zeta_t - \zeta_{t_0}, \zeta_t - \zeta_{t_0}) \big\} \\
		&< u(x_\eps,t) - v(y_\eps,t) - \Phi_{\delta,\eps}(x_\eps,y_\eps,\zeta_t - \zeta_{t_0}, \zeta_t - \zeta_{t_0}),
	\end{align*}
	and there exists $r_\eps \xrightarrow{\eps \to 0} 0$, independent of $t \in (t_0 - h_\delta,t_0 + h_\delta)$, such that $\eps  \psi(x_\eps) + \eps \psi(y_\eps) ) \le r_\eps$.
	
	Lemma \ref{L:testfunctionsmoothH}\eqref{smoothepsto0} then implies that, for sufficiently small $\eps$ depending on $\delta$,
	\[
		[0,T] \ni t \mapsto \max_{x,y \in \oline{\Omega} } \left( u(x,t) - v(y,t) - \Phi_{\delta,\eps}(x,y,\zeta_t - \zeta_{t_0}, \zeta_t - \zeta_{t_0})\right) - \mu t - \frac{|t - t_0|^2}{2}
	\]
	achieves a maximum at some $t_\eps$ such that $\lim_{\eps \to 0} t_\eps = t_0$. Shrinking $\eps$ if necessary, again depending on $\delta$, we have $t_\eps \in (t_0 - h_\delta, t_0 + h_\delta)$. It is now a consequence of Lemmas \ref{L:doubled} and \ref{L:testfunctionsmoothH}\eqref{smoothtestfneqs},\eqref{smoothtestfnboundary} that $\mu + t_\eps - t_0 \le 0$. Sending $\eps \to 0$ gives the contradiction $\mu \le 0$.	
\end{proof}

\subsection{The comparison principle: second order equations}

We now turn to the comparison principle for the second order problem
\begin{equation}\label{E:smoothgeneraleq}
	\begin{dcases}
		du = F(D^2 u, Du, u, x,t)dt + \sum_{i=1}^m H^i(Du) \cdot d\zeta^i & \text{in } \Omega \times (0,T], \\
		Du \cdot n = 0 & \text{on } \del \Omega \times (0,T],
	\end{dcases}
\end{equation}
where $F$ satisfies \eqref{A:Fcts} - \eqref{A:F}.

\begin{theorem}\label{T:secondorderconvexdomain}
	Let $u \in BUSC(\oline{\Omega} \times [0,T])$ and $ v \in BLSC(\oline{\Omega} \times [0,T])$ be respectively a bounded sub- and super-solution of \eqref{E:smoothgeneraleq}. Then
	\[
		\max_{(x,t) \in \oline{\Omega} \times [0,T]} \left( u(x,t) - v(x,t)\right) \le \max_{x \in \oline{\Omega}} \left( u(x,0) - v(x,0)\right).
	\]
\end{theorem}

An important ingredient in the proof of this comparison principle is a generalization of the so-called ``Theorem of Sums'' \cite[Theorem 3.2]{CIL} from the theory of viscosity solutions to the pathwise setting. This is Proposition \ref{P:thmofsums} in the Appendix below.

\begin{proof}[Proof of Theorem \ref{T:secondorderconvexdomain}]
	We assume without loss of generality that $u(\cdot,0) \le v(\cdot,0)$. The general case can always be reduced to this setting because, in view of \eqref{A:Fmonotone}, $\tilde u(\cdot,t) = u(\cdot,t) - \sup_x (u(x,0) - v(x,0))_+$ is a sub-solution of \eqref{E:general} and $\tilde u(\cdot,0) \le v(\cdot,0)$.

	We now assume by contradiction that the claim is false. Then, for sufficiently small $\mu > 0$ and $\delta > 0$,
	\[
		[0,T] \ni t \mapsto \sup_{x,y \in \oline{\Omega} } \left( u(x,t) - v(y,t) - \frac{1}{2\delta} |x-y|^2 \right) - \mu t
	\]
	attains a positive maximum at some $t_0 \in (0,T]$. Invoking Lemma \ref{L:testfunctionsmoothH} and arguing as in the proof of Theorem \ref{T:firstorderconvexdomain}, we have that, for sufficiently small $\eps > 0$,
	\[
		\oline{\Omega} \times \oline{\Omega} \times [0,T] \ni (x,y,t) \mapsto u(x,t) - v(y,t) - \Phi_{\delta,\eps}(x,y,\zeta_t - \zeta_{t_0}, \zeta_t - \zeta_{t_0}) - \mu t - \frac{|t - t_0|^2}{2}
	\]
	achieves a maximum at some $(x_\eps,y_\eps,t_\eps)$ such that $\lim_{\eps \to 0} t_\eps = t_0$, $u(x_\eps,t_\eps) > v(y_\eps,t_\eps)$, and $|x_\eps - y_\eps| \le 1$. Lemma \ref{L:testfunctionsmoothH}\eqref{smoothtestfnboundary} implies that if $x_\eps \in \del \Omega$, then 
	\[
		D_x\Phi_{\delta,\eps}(x_\eps,y_\eps,\zeta_{t_\eps} - \zeta_{t_0}, \zeta_{t_\eps} - \zeta_{t_0}) \cdot n(x_\eps) > 0,
	\]
	and if $y_\eps \in \del \Omega$, then
	\[
		D_y\Phi_{\delta,\eps}(x_\eps,y_\eps,\zeta_{t_\eps} - \zeta_{t_0}, \zeta_{t_\eps} - \zeta_{t_0}) \cdot n(y_\eps) > 0.
	\]
	For sufficiently small $\eps$ (depending on $\delta$) we also have $|t_\eps - t_0| \le h_\delta$, where $h_\delta$ is such that
	\[
		\max_{t \in (t_0 - h_\delta, t_0 + h_\delta)} \max_{i=1,2,\ldots,m} |\zeta^i_t - \zeta^i_{t_0}| < \gamma_\delta
	\]	
	and $\gamma_\delta$ is as in Lemma \ref{L:testfunctionsmoothH}. 

	By Lemma \ref{L:testfunctionsmoothH} and Proposition \ref{P:thmofsums}, there exist $X_{\delta,\eps},Y_{\delta,\eps} \in \mbb S^d$ such that
	\begin{align*}
		&- \left( \norm{D^2 \Phi_{\delta,\eps}(x_\eps,y_\eps,t_\eps)} + \frac{1}{\delta}\right)
		\begin{pmatrix}
			\Id & 0 \\
			0 & \Id
		\end{pmatrix}
		\\
		&\le
		\begin{pmatrix}
			X_{\delta,\eps} & 0 \\
			0 & - Y_{\delta,\eps}
		\end{pmatrix}
		\le
		D^2 \Phi_{\delta,\eps}(x_\eps,y_\eps,t_\eps) + \delta \left[ D^2 \Phi_{\delta,\eps}(x_\eps,y_\eps,t_\eps) \right]^2
	\end{align*}
	and
	\begin{align*}
		\mu + t_\eps - t_0 &\le F(X_{\delta,\eps}, D_x \Phi_{\delta,\eps}(x_\eps,y_\eps,t_\eps), u(x_\eps,t_\eps), x_\eps,t_\eps) \\
		&- F(Y_{\delta,\eps}, -D_y \Phi_{\delta,\eps}(x_\eps,y_\eps,t_\eps), v(y_\eps,t_\eps), y_\eps,t_\eps).
	\end{align*}
	Sending $\eps \to 0$ along some sub-sequence, we have $t_\eps \to t_0$ and $(X_{\delta,\eps}, Y_{\delta,\eps}) \xrightarrow{\eps \to 0} (X_\delta,Y_\delta)$, where $X_\delta,Y_\delta \in \mbb S^d$ satisfy
	\[
		- \frac{3}{\delta}
		\begin{pmatrix}
			\Id & 0 \\
			0 & \Id
		\end{pmatrix}
		\le
		\begin{pmatrix}
			X_{\delta} & 0 \\
			0 & - Y_{\delta}
		\end{pmatrix}
		\le \frac{3}{\delta}
		\begin{pmatrix}
			\Id & -\Id \\
			-\Id & \Id
		\end{pmatrix}.
	\]
	Then, for some modulus $\rho_\delta: [0,\oo) \to [0,\oo)$ depending on $\delta$, in view of \eqref{A:Fcts}, \eqref{A:Fmonotone}, and \eqref{A:F},
	\begin{align*}
		\mu &\le F(X_\delta, D_x \Phi_{\delta,\eps}(x_\eps,y_\eps,t_\eps),u(x_\eps,t_\eps), x_\eps,t_\eps)\\
		& \qquad - F(Y_\delta, -D_y \Phi_{\delta,\eps}(x_\eps,y_\eps,t_\eps), v(y_\eps,t_\eps),y_\eps,t_\eps) + \rho_\delta(\eps)\\
		&\le \omega_R\Big( \big| D_x \Phi_{\delta,\eps}(x_\eps,y_\eps,t_\eps) + D_y \Phi_{\delta,\eps}(x_\eps,y_\eps,t_\eps) \big| \\
		& \qquad + \big(1+ |D_x \Phi_{\delta,\eps}(x_\eps,y_\eps,t_\eps) | + |D_y \Phi_{\delta,\eps}(x_\eps,y_\eps,t_\eps)| \big)|x_\eps - y_\eps| \Big) + \rho_\delta(\eps),
	\end{align*}
	where $R > 0$ is a bound for $u$ and $v$. By the method of characteristics, there exist unique $\tilde x_\eps,\tilde y_\eps \in \RR^d$ such that
	\[
		D_x \Phi_{\delta,\eps}(x_\eps,y_\eps,t_\eps) = \frac{\tilde x_\eps - \tilde y_\eps}{\delta} + \eps D\psi(\tilde x_\eps) \quad \text{and} \quad
		D_y \Phi_{\delta,\eps}(x_\eps,y_\eps,t_\eps) = - \frac{\tilde x_\eps - \tilde y_\eps}{\delta} - \eps D\psi(\tilde y_\eps),
	\]
	and therefore, for fixed $\delta$, as $\eps \to 0$, $D_x \Phi_{\delta,\eps}(x_\eps,y_\eps,t_\eps) + D_y \Phi_{\delta,\eps}(x_\eps,y_\eps,t_\eps) \to 0$.
	
	Define
	\[
		A_\delta := \argmax_{\oline{\Omega} \times \oline{\Omega}} \left\{ (x,y) \mapsto u(x,t_0) - v(y,t_0) - \frac{1}{2\delta}|x-y|^2 \right\}.
	\]
	Standard arguments from the theory of viscosity solutions (see for instance \cite[Lemma 3.1]{CIL}) yield that $\lim_{\delta \to 0} \sup_{(x,y) \in A_\delta} \frac{|x-y|^2}{\delta} = 0$. Upon sending $\eps \to 0$, we reach the contradiction
	\[
		0 < \mu \le \sup_{x,y \in A_\delta} \omega_R\left( \frac{|x-y|^2}{\delta} + |x-y| \right) \xrightarrow{\delta \to 0} 0.
	\]
\end{proof}

\subsection{Existence and path stability}

We now establish the existence of a (unique) solution of either of the initial value problems for \eqref{E:smoothHJ} or \eqref{E:smoothgeneraleq}. 

\begin{theorem}\label{T:smoothexistence}
	Assume $\Omega$ satisfies \eqref{A:C1convexdomain}, $H$ satisfies \eqref{A:HC2}, $\zeta \in C([0,T],\RR^m)$, and $F$ satisfies \eqref{A:Fcts}, \eqref{A:Fmonotone}, and \eqref{A:F}. Then, for any $u_0 \in BUC(\oline{\Omega})$, there exists a unique solution of \eqref{E:smoothgeneraleq} with $u(\cdot,0) = u_0$ in the sense of Definition \ref{D:smoothH}.
\end{theorem}

\begin{proof}
	Let $(\zeta^n)_{n \in \NN} \subset C^1([0,T],\RR^m)$ be a smooth approximation of $\zeta$ as $n \to \oo$, and, for $n \in \NN$, let $u^n$ be the unique viscosity solution of the boundary/initial-value problem
	\begin{equation}\label{E:smoothapprox}
		\begin{dcases}
			\frac{\partial u^n}{\partial t} = F(D^2 u^n, Du^n, u^n,x,t)  + H(Du^n) \cdot \dot \zeta^n_t & \text{in } \Omega \times (0,T],\\
			Du^n \cdot n = 0 & \text{on } \partial \Omega \times (0,T], \\
			u^n(\cdot,0) = u_0 & \text{on } \oline{\Omega} \times \{0\}.
		\end{dcases}
	\end{equation}
	By Proposition \ref{P:nonsmoothdefconsistent}, $u^n$ is a solution of \eqref{E:smoothapprox} in the sense of Definition \ref{D:smoothH}. The comparison principle Theorem \ref{T:secondorderconvexdomain} implies that, for some $C > 0$ and for all $n \in \NN$ and $(x,t) \in \oline{\Omega} \times [0,T]$,
	\[
		u^n(x,t) \le \norm{u_0}_{\oo,\oline\Omega} + Ct + \sum_{i=1}^m H^i(0) \zeta^{n,i}(t),
	\]
	because, in view of \eqref{A:Fcts}, the right-hand side is a super-solution of \eqref{E:smoothgeneraleq} for sufficiently large $C > 0$. Therefore,
	\[
		u^\star(x,t) := \limsup_{(x',t') \to (x,t), n \to \oo} u^n(x',t')
	\]
	is bounded and upper-semicontinuous, and, in view of Proposition \ref{P:halfrelaxed}, is a sub-solution of \eqref{E:smoothgeneraleq} in the sense of Definition \ref{D:smoothH}, and $u^\star(\cdot,0) \ge u_0$ on $\oline{\Omega}$. We claim that $u^\star(\cdot,0) = u_0$. To see this, let $\phi \in C^2_b(\oline{\Omega})$ be such that $u_0 \le \phi$ in $\oline{\Omega}$ and $D\phi \cdot n > 0$ on $\oline{\Omega}$. The method of characteristics yields $\delta > 0$, depending only on $\phi$ and not on $n$, and $\Phi^n \in C([0,T],C^2(\oline{\Omega}))$ such that
	\[
		\begin{dcases}
		\frac{\partial \Phi^n}{\partial t} = \sum_{i=1}^m H^i(D\Phi^n) \cdot \dot \zeta^{n,i} & \text{in } \oline{\Omega} \times [0,\delta),\\
		D\Phi^n \cdot n > 0 & \text{on } \del \Omega \times [0,\delta),\\
		\Phi^n(\cdot,0) = \phi & \text{on } \oline{\Omega} \times \{0\}.
		\end{dcases}
	\]
	Then, by \eqref{A:Fcts}, for some sufficiently large constant $C > 0$ independent of $n$, $\Phi^n(x,t) + Ct$ is a super-solution of \eqref{E:smoothapprox}, and, therefore,
	\[
		u^n(x,t) \le \Phi^n(x,t) + Ct \quad \text{in } \oline{\Omega} \times [0,\delta).
	\]
	Sending $n \to \oo$ and $(x',t') \to (x,0)$ along an appropriate sequence then yields $u^\star(x,0) \le \phi(x)$ for $x \in \oline{\Omega}$. Because $\phi$ is arbitrary, we conclude that $u^\star(\cdot,0) = u_0$ as desired.
	
	We similarly have that the lower half-relaxed limit $u_\star$ is a bounded super-solution of \eqref{E:smoothgeneraleq} with $u_\star(\cdot,0) = u_0$. The comparison principle now yields $u^\star \le u_\star$, and therefore, because $u_\star \le u^\star$ by definition, we conclude that $u_\star = u^\star =: u$ is a solution.
\end{proof}

\begin{remark}
	The equality of $u_\star$ and $u^\star$ in the above proof implies additionally that $u^n$ converges locally uniformly to $u$.
\end{remark}

With exactly the same argument, we achieve the following stability result:

\begin{theorem}
	For $(\zeta^n)_{n \in \NN} \subset C([0,T],\RR^m)$ and $(u_0^n)_{n \in \NN} \subset BUC(\oline{\Omega})$, let $(u^n)_{n \in \NN} \subset BUC(\oline{\Omega} \times [0,T])$ be the corresponding solutions of \eqref{E:smoothgeneraleq}. If, for some $\zeta \in C([0,T],\RR^m)$ and $u_0 \in BUC(\oline{\Omega})$,
	\[
		\lim_{n \to \oo} \left( \norm{\zeta^n - \zeta}_{\oo,[0,T]} + \norm{u_0^n - u_0}_{\oo,\oline{\Omega}} \right) = 0,
	\]
	then, as $n \to \oo$, $u^n$ converges locally uniformly to the solution $u$ of \eqref{E:smoothgeneraleq}.
\end{theorem}

\section{Nonsmooth Hamiltonians: test functions}\label{sec:testfn}

The focus of this and the next section is the study of \eqref{E:neumann} when $H$ is not $C^2$. In this case, the method of characteristics is not available, and we make use of Lemma \ref{L:C11solutions} to construct test functions with the desired properties. The proofs of well-posedness results, using these test functions, appear in the next section.

It turns out to be very difficult to construct the desired type of penalizing test function under only the assumptions that the domain be convex and the Hamiltonian be, component by component, a difference of convex functions. Therefore, we consider the following three separate settings:
\begin{itemize}
\item $\Omega$ is a half space,
\item $\Omega$ has a quantified convexity assumption, or
\item $H$ is radial.
\end{itemize}
In the last case, we may take $H(p) = |p|$, which is used in the analysis of the geometric equations in Sections \ref{sec:geo} and \ref{sec:mcf}.

Throughout, for $H$ satisfying \eqref{A:DCH} and $i = 1,2,\ldots,m$, we fix convex $H^i_1,H^i_2$ such that $H^i = H^i_1 - H^i_2$. We will also impose the growth assumption
\begin{equation}\label{A:HLip}
	\esssup_{p \in \RR^d} \max_{i=1,2,\ldots,m} \frac{ |DH^i_1(p)| + |DH^i_2(p)| }{1 + |p|} < \oo.
\end{equation}
{
This implies that the convex functions $H^i_1$ and $H^i_2$ grow at most quadratically as $|p| \to \oo$, and are bounded below by affine functions. The same is then also true for the convex function $G$ from \eqref{Gfn}. } In particular, there exists $C > 0$ such that, for all $\theta = \theta_1,\theta_2,\ldots,\theta_m) \in [-1,1]^m$, there exists $\nu_\theta \in S^{d-1}$ such that
\begin{equation}\label{A:Hquadgrowth}
	C(\nu_\theta \cdot p -1) \le G(p) + \sum_{i=1}^m \theta_i H^i(p) \le C\left( 1 + \frac{|p|^2}{2} \right) \quad \text{for all } p \in \RR^d.
\end{equation}

\begin{remark}\label{R:growth}
	\begin{enumerate}[(i)]
	\item The quadratic growth assumption for $H$ can, in principle, be generalized, in which case the quadratic term $|p|^2$ in \eqref{quadphieps=0} below is replaced with a different power. Some sort of growth assumption is required, however, in order to perform a quantitative study of the test functions that are defined in what follows.	
	\item If $H \in C^2(\RR^d)$ with $\norm{D^2H}_{\oo} < \oo$, then $H = H^1 - H^2$ with
	\[
		H^1(p) = H(p) + \frac{\norm{D^2H}_{\oo}}{2} |p|^2,
	\]
	and then $H^1$ and $H^2$ satisfy \eqref{A:HLip}.
	\end{enumerate}
\end{remark}

Throughout this section, we say a constant is \emph{universal} if it depends only on $d$, the constants appearing in \eqref{A:HLip} and \eqref{A:Hquadgrowth}, and the domain.

\subsection{The penalizing test function and properties}\label{sec:testfunctions}

%
%

For $\delta > 0$, let $\gamma_\delta > 0$ satisfy
\begin{equation}\label{intervalofexistence}
	0 < \gamma_\delta \le \delta \quad \text{for all } \delta \in (0,1),
\end{equation}
and let $\ell : \RR^d \to \RR$ be such that
\begin{equation}\label{pushfn}
	\ell \text{ is convex and } 0 \le \ell(p) \le |p| \text{ for all } p \in \RR^d.
\end{equation}
Both $\gamma_\delta$ and $\ell$ will be further specified, depending on the scenario.

For $z \in \RR^d$, and $\rho \in (-3\gamma_\delta,3\gamma_\delta)^m$, we define
\begin{equation}\label{quadphieps=0}
	\phi_\delta(z,\rho) := \sup_{p \in \RR^d} \left\{ p \cdot z - \delta \left( \frac{|p|^2}{2} + \ell(p)\right) +\sum_{i=1}^m \rho_i H^i(p) - 3\gamma_\delta G(p) \right\}.
\end{equation}

\begin{lemma}\label{L:quadphieps=0}
	The function $\phi_\delta$ belongs to $C^{1,1}(\RR^d \times (-3\gamma_\delta,3\gamma_\delta)^m)$ and, for all $i = 1,2,\ldots,m$,
	\[
		\frac{\partial \phi_\delta}{\partial \rho_i} = H^i(D_z \phi_\delta) \quad \text{in } \RR^d \times (-3\gamma_\delta,3\gamma_\delta)^m.
	\]
	Moreover, there exists a universal constant $C > 0$ such that, for all $\delta \in (0,1)$, $z \in \RR^d$, and $\rho \in (-3\gamma_\delta,3\gamma_\delta)^m$,
	\[
		\frac{[(|z| - \delta)_+]^2}{2(1+C)\delta} - C \delta \le \phi_\delta(z,\rho) \le \frac{1}{2\delta}|z|^2 + C\delta
	\]
	and
	\[
		\left| D\phi_\delta(z,\rho) \right| \le C\left( 1 + \frac{|z|}{\delta} \right).
	\]
\end{lemma}

\begin{proof}
	The regularity and the satisfaction of the equations are consequences of Lemma \ref{L:C11solutions}. 
	
	By \eqref{A:Hquadgrowth}, we have, for all $\rho \in (-3\gamma_\delta,3\gamma_\delta)^m$ and $p \in \RR^d$ and for some $C > 0$ and $\nu_\rho \in S^{d-1}$,
	\[
		 - C\gamma_\delta + C \gamma_\delta \nu_\rho \cdot p 
		 \le
		 - \sum_{i=1}^m \rho_i H^i(p)  + 3 \gamma_\delta G(p) 
		 \le C\gamma_\delta  + \frac{1}{2} C \gamma_\delta |p|^2.
	\]
	The upper bound is then proved by computing, with \eqref{intervalofexistence}, for some universal $C > 0$ that may change line to line,
	\begin{align*}
		\phi_\delta(z,\rho) &\le \sup_{p \in \RR^d} \left\{ p \cdot z - \frac{\delta}{2} |p|^2 + C\gamma_\delta \nu_\rho \cdot p \right\} + C \gamma_\delta \\
		&= \frac{1}{2\delta}|z + C \gamma_\delta \nu_\rho|^2 + C\gamma_\delta \le \frac{1}{2\delta}\left( |z| + C \gamma_\delta \right)^2 + C\gamma_\delta \le \frac{|z|^2}{2\delta} + C\delta,
	\end{align*}
	and, for the lower bound,
	\begin{align*}
		\phi_\delta(z,\rho)
		&\ge \sup_{p \in \RR^d} \left\{ p \cdot z - \frac{\delta + C\gamma_\delta}{2}|p|^2 - \delta |p| \right\} - C\gamma_\delta\\
		&\ge \frac{[(|z| - \delta)_+]^2}{2(\delta + C \gamma_\delta)} - C \gamma_\delta
		\ge \frac{[(|z| - \delta)_+]^2}{2(1+C)\delta} - C \delta.
	\end{align*}
	In view of \eqref{A:Hquadgrowth}, we note that, increasing $C$ if necessary, if $|p| > C(1 + |z|/\delta)$, then
	\[
		p \cdot z - \delta \left( \frac{|p|^2}{2} + \ell(p) \right) + \sum_{i=1}^m \rho_i H^i(p) - 3 \gamma_\delta G(p) < \phi_\delta(z,\rho),
	\]
	and so the maximum in the definition of $\phi_\delta$ must be attained for $|p| \le C(1 + |z|/\delta)$. This completes the proof of the bound for $D\phi(z,\rho)$, as the unique maximum $p = p(z,\rho) = D\phi_\delta(z,\rho)$.
\end{proof}

Now, for $\delta,\eps > 0$, $x,y \in \oline{\Omega}$, and $\sigma,\tau,\rho \in (-\gamma_\delta,\gamma_\delta)^m$, we define
\begin{equation}\label{multidtestfn}
	\begin{split}
	\Phi_{\delta,\eps}&(x,y,\sigma,\tau,\rho) \\
	&:= 
	\sup_{p,u,v \in \RR^d} \Bigg\{ (p+u) \cdot x - (p-v) \cdot y  - \delta \left( \frac{|p|^2}{2} + \ell(p) \right) \\
	&- \eps \psi^*\left( \frac{u}{\eps} \right) - \eps \psi^* \left( \frac{v}{\eps} \right) + \sum_{i=1}^m \left( \sigma_i H^i(p+u) - \tau_i H^i(p-v) + \rho_i H^i(p) \right)\\
	& - \gamma_\delta \left( G(p+u) + G(p-v) + G(p) \right) \Bigg\},
	\end{split}
\end{equation}
where $\psi$ is as in \eqref{psi}. 

Recall that $\phi^*= +\oo$ outside of a compact set $K$, and therefore, the supremum in \eqref{multidtestfn} may be restricted to $p,u,v \in \RR^d$ such that, for some universal constant, $|u| \le C\eps$ and $|v| \le C\eps$.

\begin{lemma}\label{L:testfnproperties}
	The function $\Phi_{\delta,\eps}$ defined in \eqref{multidtestfn} satisfies the following:
	\begin{enumerate}[(i)]
	\item\label{A:testfn} For all $\delta,\eps \in (0,1)$, $\Phi_{\delta,\eps} \subset C^{1,1}(\oline{\Omega} \times \oline{\Omega} \times (-\gamma_\delta,\gamma_\delta)^m \times (-\gamma_\delta,\gamma_\delta)^m \times (-\gamma_\delta,\gamma_\delta)^m)$, and, for all $i = 1,2,\ldots,m$,
	\[
		\left\{
		\begin{split}
		&\frac{\del \Phi_{\delta,\eps} }{\del \sigma_i} = H^i\left( D_x \Phi_{\delta,\eps}\right)
		\quad \text{and} \quad
		\frac{\del \Phi_{\delta,\eps}}{\del \tau_i} = - H^i\left( -D_y \Phi_{\delta,\eps} \right)\\
		&\text{in } \oline{\Omega} \times \oline{\Omega} \times (-\gamma_\delta,\gamma_\delta)^m \times (-\gamma_\delta,\gamma_\delta)^m \times (-\gamma_\delta,\gamma_\delta)^m.
		\end{split}
		\right.
	\]
	\item\label{A:xybounded} For some universal $C > 0$ and for all $x,y \in \oline{\Omega}$, $\sigma,\tau,\rho \in (-\gamma_\delta,\gamma_\delta)^m$, and $\delta,\eps \in (0,1)$,
	\[
		\Phi_{\delta,\eps}(x,y,\sigma,\tau,\rho) \ge \eps \psi(x) + \eps \psi(y) - C\gamma_\delta.
	\]
	\item\label{A:Phieps=0} If $r_\eps > 0$ is such that $\lim_{\eps \to 0} r_\eps = 0$ and $R > 0$, then, for all $\delta > 0$,
	\begin{align*}
		&\lim_{\eps \to 0} \sup \Big\{ |\Phi_{\delta,\eps}(x,y,\sigma,\tau,\rho) - \phi_\delta(x-y,\sigma-\tau + \rho)| \\
	&\qquad  : |x-y| \le R, \; \eps\psi(x) + \eps \psi(y) \le r_\eps, \;  \sigma,\tau,\rho \in (-\gamma_\delta,\gamma_\delta)^m\Big\}= 0.
	\end{align*}
	\item\label{A:PhiHessian} For $\delta,\eps \in (0,1)$, $\sigma,\tau,\rho \in (-\gamma_\delta,\gamma_\delta)^m$, and $\kappa > 0$ as in \eqref{psistar},
	\[
		D^2_{(x,y)} \Phi_{\delta,\eps}(\cdot,\cdot,\sigma,\tau,\rho) \le \frac{1}{\delta}
		\begin{pmatrix}
			\Id & -\Id \\
			-\Id & \Id
		\end{pmatrix}
		+ \frac{\eps}{\kappa}
		\begin{pmatrix}
			\Id & 0 \\
			0 & \Id
		\end{pmatrix}.
	\]
	\end{enumerate}
\end{lemma}


\begin{proof}
	Part \eqref{A:testfn} is a consequence of Lemma \ref{L:C11solutions}, and part \eqref{A:xybounded} follows from the estimate, for all $x,y \in \RR^d$ and $\sigma,\tau,\rho \in (-\gamma_\delta,\gamma_\delta)^m$,
	\begin{align*}
		\Phi_{\delta,\eps}(x,y,\sigma,\tau,\rho) &\ge \sup_{u,v \in \eps K} \left\{ u \cdot x + v \cdot y - \eps \psi^*\left( \frac{u}{\eps} \right) - \eps \psi^*\left( \frac{v}{\eps} \right) \right\} - C \gamma_\delta \\
		&= \eps \psi(x) + \eps \psi(y) - C \gamma_\delta.
	\end{align*}
	To prove part \eqref{A:Phieps=0}, we first estimate from below (taking $u = v = 0$ and using $\psi^*(0) = 1$),
	\[
		\Phi_{\delta,\eps}(x,y,\sigma,\tau,\rho)
		\ge \phi_{\delta}(x-y,\sigma- \tau + \rho) - 2\eps.
	\]
	We use the sub-additivity of the supremum to obtain the upper bound
	\begin{equation}\label{Phiub}
		\begin{split}
		\Phi_{\delta,\eps}&(x,y,\sigma,\tau,\rho)
		\le 
		\sup_{p \in \RR^d, \; u,v \in \eps K}
		\Bigg\{ p \cdot (x-y)  - \delta \left( \frac{|p|^2}{2} + \ell(p) \right) \\
	&- \sum_{i=1}^m \left( \sigma_i H^i(p+u) - \tau_i H^i(p-v) + \rho_i H^i(p) \right) - \gamma_\delta \left( G(p+u) + G(p-v) + G(p) \right) \Bigg\} \\
	&+ \eps \psi(x) + \eps \psi(y).
	\end{split}
	\end{equation}
	By Young's inequality, \eqref{A:Hquadgrowth}, and \eqref{intervalofexistence}, if $|x-y| \le R$, $p \in \RR^d$ and $u,v \in \eps K$, then, for some universal $C > 0$,
	\begin{align*}
		&p \cdot (x-y)  - \delta \left( \frac{|p|^2}{2} + \ell(p) \right) 
	- \sum_{i=1}^m \left( \sigma_i H^i(p+u) - \tau_i H^i(p-v) + \rho_i H^i(p) \right) \\
	&- \gamma_\delta \left( G(p+u) + G(p-v) + G(p) \right) 
	\le R|p| - \frac{\delta}{4} |p|^2 + C\delta.
	\end{align*}
	It follows that, for some $C_R > 0$, the supremum on the right-hand side of \eqref{Phiub} may be restricted to $|p| \le C_R$. Therefore,
	\[
		\Phi_{\delta,\eps}(x,y,\sigma,\tau,\rho) \le \phi_\delta(x-y,\sigma-\tau+\rho) + C_R\eps + \eps \psi(x) + \eps \psi(y),
	\]
	and the result easily follows.
	
	Finally, to prove part \eqref{A:PhiHessian}, we rewrite, for some convex $L$,
	\begin{align*}
		\Phi_{\delta,\eps}&(x,y,\sigma,\tau,\rho)  \\
		&= \sup_{p,u,v} 
		\left\{ u \cdot x + v \cdot y - \delta \frac{|p|^2}{2} - \eps \psi^*\left( \frac{u - p}{\eps} \right) - \eps \psi^*\left( \frac{v + p}{\eps} \right) - L(p,u,v) \right\}.
	\end{align*}
	In view of \eqref{psistar}, the map
	\[
		\delta \frac{|p|^2}{2} + \eps \psi^*\left( \frac{u - p}{\eps} \right) + \eps \psi^*\left( \frac{v + p}{\eps} \right) + L(p,u,v)
	\]
	is uniformly convex, and its Hessian is bounded from below by the matrix
	\[
		B_{\delta,\eps} := \delta
		\begin{pmatrix}
			\Id & 0 & 0\\
			0 & 0 & 0\\
			0 & 0 & 0
		\end{pmatrix}
		+ 
		\frac{\kappa}{\eps}
		\left[
		\begin{pmatrix}
			\Id & -\Id & 0\\
			-\Id & \Id & 0\\
			0 & 0 & 0
		\end{pmatrix}
		+
		\begin{pmatrix}
			\Id & 0 & \Id\\
			0 & 0 & 0\\
			\Id & 0 & \Id
		\end{pmatrix}
		\right],
	\]
	whose inverse is
	\[
		C_{\delta,\eps} := B_{\delta,\eps}^{-1} = 
		\frac{1}{\delta}
		\begin{pmatrix}
			\Id & \Id & -\Id \\
			\Id & \Id & -\Id \\
			-\Id & -\Id & \Id
		\end{pmatrix}
		+ \frac{\eps}{\kappa}
		\begin{pmatrix}
			0 & 0 & 0\\
			0 & \Id & 0 \\
			0 & 0 & \Id
		\end{pmatrix}.
	\]
	It therefore follows as in Lemma \ref{L:C11solutions} that, for any $x,\hat x,y,\hat y \in \RR^d$,
	\begin{align*}
		&\left( D_x\Phi_{\delta,\eps}(x,y,\cdot) - D_x \Phi_{\delta,\eps}(\hat x, \hat y, \cdot) \right) \cdot (x - \hat x) 
		+ \left( D_y\Phi_{\delta,\eps}(x,y,\cdot) - D_x \Phi_{\delta,\eps}(\hat x, \hat y, \cdot) \right) \cdot (y - \hat y)\\
		&\le C_{\delta,\eps}
		\begin{pmatrix}
			0 \\
			x - \hat x \\
			y - \hat y
		\end{pmatrix}
		\cdot
		\begin{pmatrix}
			0 \\
			x - \hat x \\
			y - \hat y
		\end{pmatrix}
		= 
		A_{\delta,\eps}
		\begin{pmatrix}
			x - \hat x \\
			y - \hat y
		\end{pmatrix}
		\cdot
		\begin{pmatrix}
			x - \hat x \\
			y - \hat y
		\end{pmatrix},
	\end{align*}
	where the matrix $A_{\delta,\eps}$ is exactly as on the right-hand side of the inequality in part \eqref{A:PhiHessian}.
\end{proof}


We next prove, in the three aforementioned settings, that $\Phi_{\delta,\eps}$ satisfies the strict boundary inequalities
\begin{equation}\label{A:boundarycondition}
	\left\{
	\begin{split}
	&D_x \Phi_{\delta,\eps}(x,y,\sigma,\tau,\rho) \cdot n(x) > 0 \quad \text{if } x \in \del \Omega \text{ and}\\
	&D_y \Phi_{\delta,\eps}(x,y,\sigma,\tau,\rho) \cdot n(y) > 0 \quad \text{if } y \in \del \Omega
	\end{split}
	\right.
\end{equation}
if $\sigma,\tau,\rho$ (which represent the increments of paths) belong to sufficiently small intervals $(-\gamma_\delta,\gamma_\delta)^m$, and if $x$ and $y$ are sufficiently close. It makes sense to consider the latter restriction, because, in view of Lemma \ref{L:quadphieps=0}, in the proof of the comparison principle and related results, we will consider $x,y \in \oline{\Omega}$ satisfying $|x-y| = O(\delta^{1/2})$, with the proportionality constant depending on the bounds for the given sub- and super-solution (in fact we will have $|x-y| = o(\delta^{1/2})$, but it is hard to explicitly control the rate at which $\delta^{-1/2} |x_\delta-y_\delta|$ goes to $0$ for such points).

As we explain with the next lemma, the condition $|x-y|  = O(\delta^{1/2})$ implies some control on the size of the unique maximizer $p$ in \eqref{multidtestfn}.

\begin{lemma}\label{L:boundedp}
	Assume \eqref{A:DCH}, \eqref{A:HLip}, \eqref{intervalofexistence}, and \eqref{pushfn}. Then there exists universal $K > 0$ such that, for all $\delta,\eps \in (0,1)$, $\sigma,\tau,\rho \in (-\gamma_\delta,\gamma_\delta)^m$, and $x,y \in \oline{\Omega}$ satisfying, for some $M \ge 1$, $|x-y| \le M\delta^{1/2}$, if $(p,u,v)$ is the unique maximizer in the formula for $\Phi_{\delta,\eps}(x,y,\sigma,\tau,\rho)$, then $|p| \le KM \delta^{-1/2}$. 
\end{lemma}

\begin{proof}
If, for some $K \ge 2$, $|p| > KM\delta^{-1/2}$, then the lower bound in \eqref{A:Hquadgrowth} gives $C > 0$ such that
\begin{align*}
	\Phi_{\delta,\eps}(x,y,\sigma,\tau,\rho)
	&\le(p+u) \cdot x - (p - v) \cdot y - \frac{\delta}{2} |p|^2 - \eps \psi^*\left(\frac{u}{\eps}\right) - \eps \psi^*\left(\frac{v}{\eps}\right) + C \delta(1 + |p|)\\
	&\le M\delta^{1/2}|p|  - \frac{\delta}{4} |p|^2 + \eps \psi(x) + \eps \psi(y) + C\\
	&< M^2\left(K - \frac{K^2}{4} \right)+ \eps \psi(x) + \eps \psi(y) + C,
\end{align*}
which, in view of Lemma \ref{L:testfnproperties}\eqref{A:xybounded}, is a contradiction for sufficiently large $K$. 
\end{proof}


The next result gives some more detail on the maximizers $p,u,v$ in the definition of $\Phi_{\delta,\eps}$:

\begin{lemma}\label{L:puv}
	Assume \eqref{A:DCH}, \eqref{A:HLip}, \eqref{intervalofexistence}, and \eqref{pushfn}, and let $\delta,\eps \in (0,1)$, $\sigma,\tau,\rho \in (-\gamma_\delta,\gamma_\delta)^m$, $x,y \in \oline{\Omega}$, $M \ge 1$, and $|x-y| \le M \delta^{1/2}$. Then there exists a universal constant $C > 0$ and $w_1,w_2 \in \RR^d$ such that, if $(p,u,v) \in \RR^{3d}$ is the unique maximizer in \eqref{multidtestfn}, then
	\[
		u = \eps D\psi(x) + w_1, \quad v = \eps D\psi(y) + w_2, \quad \text{and} \quad |w_1| + |w_2| \le C \eps M\gamma_\delta \delta^{-1/2}.
	\]
	Moreover, if $\ell$ is differentiable at $p$, then there exists $w_3 \in \RR^d$ such that
	\[
		p = \frac{x-y}{\delta} - D\ell(p) + w_3 \quad \text{and} \quad |w_3| \le M \gamma_\delta \delta^{-3/2}.
	\]
	Finally, if $H^i_1$ and $H^i_2$ are globally Lipschitz for all $i = 1,2,\ldots,m$, then $w_1,w_2,w_3$ satisfy
	\[
		|w_1| + |w_2| \le C\eps\gamma_\delta \quad \text{and} \quad |w_3| \le C \gamma_\delta \delta^{-1}.
	\]
\end{lemma}

\begin{proof}
	We have
	\[
		x \in \del \left( \eps\psi^*\left( \frac{\cdot}{\eps} \right) - \sum_{i=1}^m \sigma_i H^i(p + \cdot) + \gamma_\delta G(p + \cdot) \right)(u).
	\]
	We use the additivity of the sub-differential (see for instance \cite{HL}) to deduce that
	\[
		x \in (\del \psi^*)\left( \frac{u}{\eps} \right) + \del \left( \gamma_\delta G( p + \cdot) - \sum_{i=1}^m \sigma_i H^i(p + \cdot) \right) (u),
	\]
	and so, by \eqref{A:HLip} and Lemma \ref{L:boundedp}, there exists $w \in \RR^d$ with $|w| \le CM \gamma_\delta \delta^{-1/2}$ such that $u = \eps D\psi(x + w)$. The result for $u$ now follows from the fact that $D\psi$ is Lipschitz. The arguments for $v$ and $p$ are similar, and the statement for globally Lipschitz Hamiltonians follows easily from the fact that the bound $p$ in Lemma \ref{L:boundedp} no longer needs to be taken into account above.
\end{proof}

We now prove the strict boundary inequalities for $\Phi_{\delta,\eps}$ in the cases where $\Omega$ is a half space, $\Omega$ is convex in a quantifiable way, or $H$ is radial. Throughout, we always assume \eqref{A:DCH}, \eqref{A:HLip}, \eqref{intervalofexistence}, and \eqref{pushfn}. The last two assumptions are further specified in the three different cases.

\subsection{The half space} \label{subsec:halfspace}

We set
\[
	\Omega = \left\{ x \in \RR^d : x \cdot e_1 > 0 \right\},
\]
where $e_1 = (1,0,0,\ldots,0)$; here $n(x) = -e_1$ for all $x \in \del \Omega = \{ x \in \RR^d : x_1 = 0 \}$. 


\begin{lemma}\label{L:halfspacetestfn}
	Let $M \ge 1$ and set $\ell(p) = |p \cdot e_1|$. Then there exists a universal constant $c_0 \in (0,1)$ such that, if $\gamma_\delta = c_0 M^{-1}\delta^{3/2}$, then, for all $\delta,\eps \in (0,1)$, $\sigma, \tau,\rho \in (-\gamma_\delta,\gamma_\delta)^m$, and $|x-y| \le M \delta^{1/2}$, then \eqref{A:boundarycondition} holds. If $H^i_1$ and $H^i_2$ are globally Lipschitz for $i = 1,2,\ldots,m$, then the same is true if $\gamma_\delta = c_0 \delta$.
\end{lemma}

\begin{proof}
We only prove the first inequality in \eqref{A:boundarycondition}, as the argument for the second is similar.

Fix $x \in \del \Omega$ and $y \in \oline{\Omega}$ be such that $|x-y| \le M\delta^{1/2}$, and fix $\sigma,\tau,\rho \in (-\gamma_\delta,\gamma_\delta)^m$. Let $p,u,v$ uniquely attain the maximum in \eqref{multidtestfn}, so that $D_x \Phi_{\delta,\eps}(x,y,\sigma,\tau,\rho) = p + u$. 

By Lemma \ref{L:puv}, for some $C > 0$,
\[
	u \cdot n(x) \ge \eps D\psi(x) \cdot n(x) - C\eps M \gamma_\delta \delta^{-1/2} \ge \eps(1 - Cc_0 \delta^{1/2}).
\]
Taking $c_0$ sufficiently small yields $u \cdot n(x) > 0$, and so it suffices to prove that $p \cdot n(x) = - p \cdot e_1 \ge 0$. We are done if $p \cdot e_1 = 0$. On the other hand, if $p\cdot e_1 \neq 0$, then Lemma \ref{L:puv} gives
\[
	p = \frac{x-y}{\delta} - \sgn (p \cdot e_1) e_1 + O(M\gamma_\delta \delta^{-1/2}).
\]
Assuming by contradiction that $p \cdot n(x) = -p \cdot e_1 < 0$, we find that
\[
	p = \frac{x-y}{\delta} - e_1 + O(M\gamma_\delta \delta^{-3/2}),
\]
and so, for some $C > 0$, $p \cdot e_1 \le -1 + Cc_0$. This is negative upon shrinking $c_0$, which is a contradiction. If the $H^i_1,H^i_2$ are globally Lipschitz, the argument follows in the same way.
\end{proof}
\begin{remark}
	When $d =1$ (that is, when $\Omega$ is the half-line), we may also take $\gamma_\delta = c_0 \delta$ in general. Indeed, in that case, $\del \Omega = \{0\}$ and $n(0) = -1$, and the maximizer $p$ satisfies $p \cdot n(x) = -p$, and thus, if $p \cdot n(x) < 0$, then
	\[
		|p| = p = \frac{x-y}{\delta} - \sgn p + O(\gamma_\delta \delta^{-1}(1 + |p|)) \le -1 + Cc_0 (1 + |p|).
	\]
	Rearranging terms and shrinking $c_0$ gives the contradiction in this case.
\end{remark}

\subsection{Quantified convexity assumption}

{
We now assume some control on the convexity of $\Omega$, namely,}
\begin{equation}\label{A:uniformlyconvexdomain}
	\left\{
	\begin{split}
	&\text{there exist } \theta > 0 \text{ and } q \ge 2 \text{ such that, for all $x \in \partial \Omega$ and $y \in \Omega$}, \\
	&n(x) \cdot (x-y) \ge \theta |x-y|^q.
	\end{split}
	\right.
\end{equation}
\begin{lemma}\label{L:unifconvtestfn}
	Assume $\Omega$ satisfies \eqref{A:uniformlyconvexdomain}, $\ell(p) = |p|$, and $M \ge 1$. Then there exists universal $c_0 \in (0,1)$ depending on the parameters $\theta$ and $q$, such that, if $\gamma_\delta = c_0 M^{-1} \delta^{q + \frac{1}{2}}$, then, for all $\delta,\eps \in (0,1)$, $\sigma, \tau,\rho \in (-\gamma_\delta,\gamma_\delta)^m$, and $|x-y| \le M \delta^{1/2}$, \eqref{A:boundarycondition} holds. If $H^i_1,H^i_2$ are globally Lipschitz for all $i = 1,2,\ldots,m$, then the same is true with $\gamma_\delta = c_0 \delta^q$.
\end{lemma}

\begin{proof}
	Once again, we only prove the first inequality in \eqref{A:boundarycondition}, and we do so under the more general growth assumption \eqref{A:HLip}, because the argument runs similarly if $H^i_1,H^i_2$ are globally Lipschitz.
	
	Fix $x \in \partial \Omega$ and $y \in \oline{\Omega}$ satisfying $|x-y| \le M\delta^{1/2}$, and $\sigma,\tau,\rho \in (-\gamma_\delta,\gamma_\delta)^m$, and let $p$, $u$, and $v$ be the unique optimizers in the definition of $\Phi_{\delta,\eps}$. We want to prove that $(p+u) \cdot n(x) >0$. 
	
	Invoking Lemma \ref{L:puv} again gives
	\[
		u = \eps \left[ D\psi(x) + O(M\gamma_\delta \delta^{-1/2}) \right] \quad \text{and} \quad v = \eps \left[ D\psi(y) + O(M\gamma_\delta \delta^{-1/2}) \right].
	\]
	In particular, taking $c_0$ sufficiently small yields $u \cdot n(x) > 0$. Thus, if $p = 0$, we are done.
	
	Assume then that $p\neq 0$, and suppose for the sake of contradiction that $p \cdot n(x) < 0$. Then
	\[
		p = \frac{x-y}{\delta} -  \frac{p}{|p|} +O\left(\frac{M\gamma_\delta}{\delta^{3/2}}\right),
	\]
	and taking the scalar product with $\frac{p}{|p|}$ implies that $\frac{|x-y|}{\delta} \geq 1 - O\left(\frac{M\gamma_\delta}{\delta^{3/2}}\right)$, so, shrinking $c_0$ if necessary, $|x-y| \ge \delta/2$.

	On the other hand, by \eqref{A:uniformlyconvexdomain},
	\begin{align*}
		p\cdot n(x) &\geq  \frac{(x-y)\cdot n(x)}{\delta}  - \frac{p}{|p|} \cdot n(x) + O\left(\frac{M\gamma_\delta}{\delta^{3/2}}\right) \\
		&\geq \theta \frac{|x-y|^q}{\delta} + O\left(\frac{M\gamma_\delta}{\delta^{3/2}}\right) 
		\geq \frac{\theta}{2^q} \delta^{q-1} - O\left(\frac{M\gamma_\delta}{\delta^{3/2}}\right).
	\end{align*}
	The right-hand side is positive if $c_0$ is sufficiently small, and a contradiction is reached. 
\end{proof}


\subsection{Radial Hamiltonians}
We return to the setting where $\Omega$ is $C^1$ and convex, but we make the assumption that
\begin{equation}\label{A:Hradial}
	H(p) = h(|p|) \quad \text{for some } h\in C^2([0,\oo),\RR^m).
\end{equation}
Then $H$ as in \eqref{A:Hradial} satisfies \eqref{A:DCH} and \eqref{A:HLip}. Indeed, for $i = 1,2,\ldots,m$, we write
\[
	H^i_1(p) = h^i_1(|p|) := h^i(|p|) + C(|p| + |p|^2)  \quad \text{and} \quad H^i_2(p) = h^i_2(|p|) := C(|p| + |p|^2).
\]
The function $H^i_2$ is convex, and, for sufficiently large $C$ depending on $\norm{h^i}_{C^2}$,
\[
	D^2 H^i_1(p) = \left[ (h^i)''(|p|) + 2C \right] \frac{p}{|p|} \otimes \frac{p}{|p|} + \left[ \frac{ (h^i)'(|p|) + C}{|p|} + 2C \right] \left( \Id - \frac{p}{|p|} \otimes \frac{p}{|p|} \right) \ge 0.
\]
Let us write $G(p) = g(|p|)$, that is, for $r \ge 0$, $g(r) = \sum_{i=1}^m \left( h^i_1(r) + h^i_2(r) \right)$.

We once more consider the test function \eqref{multidtestfn} with $\ell(p)=|p|$, which we write as
\begin{equation}\label{radialHtestfn}
	\begin{split}
	\Phi_{\delta,\eps}&(x,y,\sigma,\tau,\rho) \\
	&:= 
	\sup_{p,u,v} \Bigg\{ (p+u) \cdot x - (p-v) \cdot y  - \delta \left( \frac{|p|^2}{2} + |p| \right) - \eps \psi^*\left( \frac{u}{\eps} \right) - \eps \psi^* \left( \frac{v}{\eps} \right) \\
	&+\sum_{i=1}^m \left( \sigma_i h^i(|p+u|) - \tau_i h^i(|p-v|) + \rho_i h^i(|p|) \right)\\
	& - \gamma_\delta \left( g(|p+u|) + g(|p-v|) + g(|p|) \right) \Bigg\}.
	\end{split}
\end{equation}

\begin{lemma}\label{L:radialHtestfn}
	Assume $H$ satisfies \eqref{A:Hradial}. Then there exist universal constants $\delta_0$, $\eps_0$, $c_0$, $\nu_0 \in (0,1)$ such that, if $\gamma_\delta = c_0 M^{-1} \delta^{3/2}$, then, for all $\delta \in (0,\delta_0)$, $\eps \in (0,\eps_0)$, $\sigma,\tau,\rho \in (-\gamma_\delta, \gamma_\delta)^m$, and $x,y \in \oline{\Omega}$ such that $|x-y| \le \min\left\{ M \delta^{1/2}, \nu_0 \right\}$, \eqref{A:boundarycondition} holds. If $h^i_1,h^i_2$ are globally Lipschitz for $i = 1,2,\ldots,m$, then the same is true with $\gamma_\delta = c_0 \delta$.
\end{lemma}
 
\begin{proof}	
Fix $x,y,\sigma,\tau,\rho$ as in the statement of the lemma, and let $p,u,v$ be the unique maximizers in the definition of $\Phi_{\delta,\eps}$. We need to show that $(p + u) \cdot n(x) > 0$.

Lemma \ref{L:puv} gives
\[
	u = \eps \left[ D\psi(x) + O(\gamma_\delta M\delta^{-1/2}) \right] \quad \text{and} \quad v = \eps \left[ D\psi(y) + O(\gamma_\delta M \delta^{-1/2}) \right],
\]
so that taking $c_0$ sufficiently small yields $u \cdot n(x) > \frac{\eps}{2}$, and, thus, for some universal $c > 0$,
\[
	\frac{u}{|u|} \cdot n(x) > c.
\]
If $\nu_0$ is sufficiently small (depending on $\norm{D^2\psi}_\oo$), then also $(u + v) \cdot n(x) > 0$. Therefore, if $p = 0$ or $p = v$, we are done, so assume otherwise.

Let us next rule out the case that $p = -u$. In view of Lemma \ref{L:puv} and the fact that $p \ne 0$,
\begin{equation}\label{pexpression}
	p = \frac{x-y}{\delta} - \frac{p}{|p|} + O\left( \frac{M\gamma_\delta}{\delta^{3/2}} \right),
\end{equation}
and so plugging in $p = -u$ gives
\[
	-u = \frac{x-y}{\delta} + \frac{u}{|u|} + O\left( \frac{M\gamma_\delta}{\delta^{3/2}} \right).
\]
Taking the scalar product with $n(x)$ yields
\[
	-\frac{\eps}{2} \ge -u \cdot n(x) \ge \frac{x-y}{\delta} \cdot n(x) + \frac{u}{|u|} \cdot n(x) + O\left( \frac{M\gamma_\delta}{\delta^{3/2}} \right) \ge c - O\left( \frac{M\gamma_\delta}{\delta^{3/2}} \right),
\]
which is a contradiction if $c_0$ and $\eps_0$ are sufficiently small. Therefore, $p \ne -u$.

Using now that $p \ne 0$, $p \ne -u$, and $p \ne v$, we may write
\begin{align*}
	p &= \frac{x-y}{\delta} - \frac{p}{|p|} - \left( \frac{\gamma_\delta}{\delta} g'(|p + u|) - \frac{1}{\delta} \sum_{i=1}^m \sigma_i (h^i)'(|p+u|) \right) \frac{p+u}{|p+u|} \\
	&- \left( \frac{\gamma_\delta}{\delta} g'(|p - v|) + \frac{1}{\delta} \sum_{i=1}^m \tau_i (h^i)'(|p - v|) \right) \frac{p - v}{|p - v|} \\
	&- \left( \frac{\gamma_\delta}{\delta} g'(|p|) - \frac{1}{\delta} \sum_{i=1}^m \rho_i (h^i)'(|p|) \right) \frac{p}{|p|}.
\end{align*}
Note that the three terms in parentheses are all nonnegative.

Assume for the sake of contradiction that $(p+u)\cdot n(x) \leq 0$. This implies that $p\cdot n(x) < -c\eps$. Moreover,
\[
	|v - u| \le \eps |D\psi(y) - D\psi(x)| + O(\eps M\gamma_\delta \delta^{-1/2}) \le O(\eps(\nu_0 + c_0 \delta)),
\]
and so, once more shrinking $\nu_0$ and $c_0$, we also have $(p + v) \cdot n(x) \le 0$. We conclude that
\[
	(p+u) \cdot n(x) \geq \epsilon (1-O(M\gamma_\delta \delta^{-1/2})), 
\]
which is a contradiction for sufficiently small $c_0$. We conclude that $(p+u) \cdot n(x) > 0$. The argument is similar for the other inequality.

\end{proof}


We now strengthen the assumption \eqref{A:DCH} for the radial Hamiltonians, and assume
\begin{equation}\label{A:convexH}
	H^i \text{ is convex for all } i = 1,2,\ldots,m.
\end{equation}
In this case, we will be able to lengthen some time intervals, from one side, for which the above properties of $\Phi_{\delta,\eps}$ hold. We shall use this observation to prove monotonicity in the path variable for \eqref{E:neumann} when $H$ is radial and convex.

We first note that, if each $H^i$ is convex, we may take $H^i_1 = H^i$ and $H^i_2 = 0$, so that $G = \sum_{i=1}^m H^i$. Then \eqref{A:Hquadgrowth} can be refined to, for some universal $C > 0$,
	\begin{equation}\label{newGHestimate}
		\left\{
		\begin{split}
		&0 \le G(p) + \sum_{i=1}^m \theta_i H^i(p) \le C|1 + \theta|(|p| + |p|^2) \\
		&\text{for all } p \in \RR^d \text{ and } \theta = (\theta_1, \theta_2, \ldots, \theta_m) \in [-1,\oo)^m.
		\end{split}
		\right.
	\end{equation}
The notation $1 + \theta$ means $(1 + \theta_1,1 + \theta_2, \ldots, 1 + \theta_m)$.
	
It is immediate that \eqref{newGHestimate} can be used to generalize Lemma \ref{L:quadphieps=0} and obtain the following result:

\begin{lemma}\label{L:phideltaconvex}
	The function $\phi_\delta$ belongs to $C^{1,1}(\RR^d \times (-\oo,3\gamma_\delta)^m$, and satisfies the same equations as in Lemma \ref{L:quadphieps=0} in that domain. Moreover, there exists a universal $C > 0$ such that, for all $\delta \in (0,1)$, $z \in \RR^d$, $\nu \in (0,1)$, and $\rho \in (- \nu,3\gamma_\delta)^m$,
	\[
		\frac{ \left[ (|z| - C(\delta + \nu))_+ \right]^2}{2(1 + C)(\delta + \nu)} 
		\le \phi_\delta(z,\rho)
		\le \frac{|z|^2}{2\delta}.
	\]
\end{lemma}

We similarly analyze the test function $\Phi_{\delta,\eps}$ from \eqref{radialHtestfn}, which now takes the form
\begin{equation}\label{convexradialHtestfn}
	\begin{split}
	\Phi_{\delta,\eps}(x,y,\sigma,\tau,\rho) 
	&:= 
	\sup_{p,u,v} \Bigg\{ (p+u) \cdot x - (p-v) \cdot y  - \delta \left( \frac{|p|^2}{2} + |p| \right) - \eps \psi^*\left( \frac{u}{\eps} \right) - \eps \psi^* \left( \frac{v}{\eps} \right) \\
	&-\sum_{i=1}^m \Big((\gamma_\delta - \sigma_i) h^i(|p+u|) + (\gamma_\delta + \tau_i) h^i(|p-v|) + (\gamma_\delta -  \rho_i) h^i(|p|) \Big) \Bigg\}.
	\end{split}
\end{equation}

\begin{lemma}\label{L:convexH}
	Assume that $H$ satisfies \eqref{A:Hradial} and \eqref{A:convexH}. Then the results of Lemma \ref{L:testfnproperties} continue to hold if, for any fixed $\nu \in (0,1)$, the variable $\rho$ is everywhere allowed to belong to $(-\nu,\gamma_\delta)^m$.
\end{lemma}

\begin{proof}
	In view of \eqref{newGHestimate}, arguing just as in Lemma \ref{L:C11solutions}, we have $\Phi_{\delta,\eps} \in C^{1,1}(\RR^d \times \RR^d \times (-\gamma_\delta,\gamma_\delta)^m \times (-\gamma_\delta,\gamma_\delta)^m \times (-\nu, \gamma_\delta)^m)$, with the equation in Lemma \ref{L:testfnproperties}\eqref{A:testfn} being satisfied in the new, larger domain. The estimate \eqref{newGHestimate} also leads easily to the bounds and limits in Lemma \ref{L:testfnproperties}\eqref{A:xybounded}, \eqref{A:Phieps=0} (note that we use $H(0) = 0$), and \eqref{A:PhiHessian}, again all satisfied on the new domains.
\end{proof}

Finally, we show that the boundary behavior from Lemma \ref{L:radialHtestfn} can also be established on a longer interval.

\begin{lemma}\label{L:convradialboundary}
	Assume that $H$ satisfies \eqref{A:Hradial} and \eqref{A:convexH}. Then the conclusions of Lemma \ref{L:radialHtestfn} continue to hold if we impose $|x-y| \le \nu_0$, $\gamma_\delta = c_0 \delta^2$, and $\rho$ is allowed to belong to $(-\oo,\gamma_\delta)$.
\end{lemma}

\begin{proof}
	First, for $x,y,\sigma,\tau,\rho$ as in the statement of the lemma, we note that, if $p,u,v$ are the unique maximizers in the definition of $\Phi_{\delta,\eps}(x,y,\sigma,\tau,\rho)$, then we have $|p| \le C/\delta$ for some universal $K > 0$, the argument following exactly as in Lemma \ref{L:boundedp} and using the lower bound in \eqref{newGHestimate}.
	
	We next note that the conclusions of Lemma \ref{L:puv} regarding $u$ and $v$ continue to hold, and, therefore, 
	\[
		u = \eps \left[ D\psi(x) + O(\gamma_\delta \delta^{-1}) \right] \quad \text{and} \quad v = \eps \left[ D\psi(y) + O(\gamma_\delta \delta^{-1}) \right].
	\]
	Arguing just as in Lemma \ref{L:radialHtestfn}, we thus have, for sufficiently small $c_0$ and $\nu_0$, $u \cdot n(x) \ge c\eps$, $\frac{u}{|u|} \cdot n(x) \ge c$, and $(u+v) \cdot n(x) > 0$, from which we can assume without loss of generality that $p \ne 0$ and $p \ne v$. Then, with a similar argument as for Lemma \ref{L:puv},
\[
	p = \frac{x-y}{\delta} - \left( 1 + \sum_{i=1}^m (\gamma_\delta - \rho_i) (h^i)'(|p|) \right) \frac{p}{|p|} + O\left( \frac{\gamma_\delta}{\delta^{2}} \right).
\]
Note that
\[
	 1 + \sum_{i=1}^m (\gamma_\delta - \rho_i) (h^i)'(|p|) \ge 0,
\]
so that we may argue as before that, if $\eps_0$ is sufficiently small, then $p \ne -u$. The rest of the argument follows in exactly the same way as the rest of the proof of Lemma \ref{L:radialHtestfn}.
\end{proof}

\section{Nonsmooth Hamiltonians: well-posedness}\label{sec:nonsmoothwellposed}

Making use of the test functions analyzed in the previous section, we now prove well-posedness results for the Neumann problem
\begin{equation}\label{E:general}
	\begin{dcases}
		du = F(D^2 u ,Du,u,x,t)dt + \sum_{i=1}^m H^i(Du) \cdot d\zeta^i & \text{in } \Omega \times (0,T], \\
		Du \cdot n = 0 & \text{on } \del \Omega \times (0,T],
	\end{dcases}
\end{equation}
with a non-$C^2$ Hamiltonian, making use of the test functions analyzed in the previous section. As mentioned there, the results here apply when $\Omega$ is a half-space, $\Omega$ has a quantified convexity assumption, or $H$ is radial.

Just as in the previous section, a constant $C$ is called universal if it depends only on $\Omega$, $F$, or $H$.

\subsection{A general estimate}

We first prove a general estimate for the difference of a sub- and super-solution, with two different driving paths and with a doubling in the space variable, with penalization provided by the function $\phi_\delta$ from \eqref{quadphieps=0}. This will be used to prove existence, uniqueness, and precise stability estimates in the path variable, as well as a uniform modulus of continuity in space for the solutions that does not depend on the path.

%

\begin{proposition}\label{P:keyestimate}
	Assume that $\Omega$ satisfies \eqref{A:C1convexdomain}, $H$ satisfies \eqref{A:DCH} and \eqref{A:HLip}, and $F$ satisfies \eqref{A:Fcts}, \eqref{A:Fmonotone}, and \eqref{A:F}. Let $\zeta,\eta \in C_0([0,T],\RR^m)$, and assume $u \in BUSC(\oline{\Omega} \times [0,T])$ and $v \in BLSC(\oline{\Omega} \times [0,T])$ are respectively a bounded sub- and super-solution of \eqref{E:general} corresponding to $\zeta$ and $\eta$. Then there exist universal constants $0 < c_0 \le 1 \le C_0$ such that, if
	\[
		M := C_0\left( \norm{u}_{\oo} + \norm{v}_{\oo} + 1 \right)^{1/2};
	\]
	if one of the following holds:
	\begin{enumerate}[(i)]
	\item\label{case:halfspace} $\Omega$ is a half space, $\delta \in (0,1)$, and $\gamma_\delta = c_0 M^{-1}\delta^{3/2}$;
	\item\label{case:unifconv} $\Omega$ satisfies the uniform convexity condition \eqref{A:uniformlyconvexdomain}, $\delta \in (0,1)$, and $\gamma_\delta = c_0M^{-1} \delta^{q + \frac{1}{2}}$, or
	\item\label{case:radial} $H$ satisfies the radial condition \eqref{A:Hradial}, $\delta \in (0, c_0 M^{-2})$, and $\gamma_\delta = c_0 M^{-1} \delta^{3/2}$;
	\end{enumerate}
	and if $T > 0$ is such that
	\begin{equation}\label{closepaths}
		\max_{t \in [0,T]} \max_{i=1,2,\ldots,m} |\zeta^i_t - \eta^i_t| < \gamma_\delta,
	\end{equation}
	then, for some $\omega_\delta > 0$ depending on $u$ and $v$ such that $\lim_{\delta \to 0} \omega_\delta = 0$,
	\begin{align*}
		\sup_{(x,y,t) \in \oline{\Omega} \times \oline{\Omega} \times [0,T]} &\Big\{ \big( u(x,t) - v(y,t) - \phi_\delta(x- y, \zeta_t - \eta_t) \big)_+ \Big\}\\
		& \le \sup_{x,y \in \oline{\Omega}} \Big\{ \big( u(x,0) - v(y,0)  - \phi_\delta(x-y,0) \big)_+ \Big\} + \omega_\delta t.
	\end{align*}
\end{proposition}

{
The proof of Proposition \ref{P:keyestimate} has many similarities with that of Proposition \ref{P:thmofsums} in the Appendix, but with some key differences due to the fact that the test-functions are not necessarily smooth.


\begin{proof}[Proof of Proposition \ref{P:keyestimate}]
	Without loss of generality, it suffices to consider the case where $u(x,0) - v(y,0) - \phi_\delta(x-y,0) \le 0$. One can always reduce to this case in general by defining $\tilde u(\cdot,t) = u(\cdot,t) - \sup_{x,y} (u(x,0) - v(y,0) - \phi_\delta(x-y,0))_+$ and noting that $\tilde u$ is again a sub-solution in view of the monotonicity of $F$ \eqref{A:Fmonotone}, while the lower bound in Lemma \ref{L:quadphieps=0} implies $\norm{\tilde u}_\oo \le 2 \norm{u}_\oo + \norm{v}_\oo + C$.

	Assume for the sake of contradiction that, for some $\mu > 0$,
	\[
		[0,T] \ni t \mapsto \sup_{x,y \in \oline{\Omega}} \left( \big(u(x,t) - v(y,t) - \phi_{\delta}(x-y,\zeta_t - \eta_t)\big)_+ \right) - \mu t
	\]
	attains a strict maximum at $\hat t > 0$. In particular, $u(x,t) - v(y,t) - \phi_\delta(x-y,\zeta_t - \eta_t)$ is not everywhere nonpositive and, for $t$ sufficiently close to $\hat t$,
	\begin{align*}
		\sup_{x,y \in \oline{\Omega}}& \left( \big( u(x,t) - v(y,t)) - \phi_{\delta}(x-y,\zeta_t - \eta_t)  \big)_+ \right)\\
		&= \sup_{x,y \in \oline{\Omega}} \left( u(x,t) - v(y,t) - \phi_{\delta}(x-y,\zeta_t - \eta_t) \right).
	\end{align*}
	Moreover, the upper and lower bound on $\phi_\delta$ imply that there exists $r_\delta > 0$ such that $r_\delta^2/\delta \xrightarrow{\delta \to 0} 0$ and, for all $t \in [0,T]$, the supremum may be restricted to $|x-y| < r_\delta$.

	Let $\Phi_{\delta,\eps}$ be the test function defined in \eqref{multidtestfn} with $\ell(p) = |p\cdot e_1|$ if case \eqref{case:halfspace} holds and $\ell(p) = |p|$ if \eqref{case:unifconv} or \eqref{case:radial} hold. Then Lemma \ref{L:quadphieps=0} and Lemma \ref{L:testfnproperties}\eqref{A:xybounded},\eqref{A:Phieps=0} imply that, for sufficiently small $\eps$, depending on $\delta$,
	\[
		\oline{\Omega} \times \oline{\Omega} \times [0,T] \ni (x,y,t) \mapsto \big( u(x,t) - v(y,t) - \Phi_{\delta,\eps}(x,y,\zeta_t - \zeta_{\hat t}, \eta_t - \eta_{\hat t}, \zeta_{\hat t} - \eta_{\hat t}) \big)_+ - \mu t
	\]
	attains a maximum at $(x_\eps,y_\eps,t_\eps) \in \oline{\Omega} \times \oline{\Omega} \times (0,T]$ with $\lim_{\eps \to 0} t_\eps = \hat t$, so that we can ensure that, for all sufficiently small $\eps$ depending on $\delta$,
	\begin{equation}\label{pospart}
		u(x_\eps,t_\eps) - v(y_\eps,t_\eps) - \Phi_{\delta,\eps}(x_\eps,y_\eps,\zeta_{t_\eps} - \zeta_{\hat t}, \eta_{t_\eps} - \eta_{\hat t}, \zeta_{\hat t} - \eta_{\hat t}) > 0
	\end{equation}
	and
	\[
		\max_{i=1,2,\ldots,m} |\zeta^i_{t_\eps} - \zeta^i_{\hat t}| \vee |\eta^i_{t_\eps} - \eta^i_{\hat t}| < \gamma_\delta.
	\]
	Moreover, 
	\begin{equation}\label{bigmaxball}
		\lim_{\eps \to 0} (\eps \psi(x_\eps) + \eps \psi(y_\eps) )= 0,
	\end{equation}
	and, for sufficiently small $\eps$, $C_0 \ge 1$ may be chosen (so as to determine $M$) such that $|x_\eps - y_\eps| \le M \delta^{1/2}$.
	
	Set
	\[
		p_\eps := D_x \Phi_{\delta,\eps}(x_\eps,y_\eps, \zeta_{t_\eps} - \zeta_{\hat t}, \eta_{t_\eps} - \eta_{\hat t}, \zeta_{\hat t} - \eta_{\hat t}),
	\]
	\[
		q_\eps := D_y \Phi_{\delta,\eps}(x_\eps,y_\eps,  \zeta_{t_\eps} - \zeta_{\hat t}, \eta_{t_\eps} - \eta_{\hat t}, \zeta_{\hat t} - \eta_{\hat t}),
	\]
	and
	\[
		A_{\delta,\eps} := 
		\frac{1}{\delta}
		\begin{pmatrix}
			\Id & -\Id \\
			-\Id & \Id
		\end{pmatrix}
		+ \eps \left(1 + \frac{1}{\kappa} \right)
		\begin{pmatrix}
			\Id & 0 \\
			0 & \Id
		\end{pmatrix}.
	\]
	Then, by Lemma \ref{L:testfnproperties}\eqref{A:PhiHessian}, for all $x,y \in \oline{\Omega}$,
	\begin{equation}\label{Taylor}
	\begin{split}
		\Phi_{\delta,\eps}(x,y, \zeta_{t_\eps} - \zeta_{\hat t}, \eta_{t_\eps} - \eta_{\hat t}, \zeta_{\hat t} - \eta_{\hat t})
		&\le \Phi_{\delta,\eps}(x_\eps, y_\eps,  \zeta_{t_\eps} - \zeta_{\hat t}, \eta_{t_\eps} - \eta_{\hat t}, \zeta_{\hat t} - \eta_{\hat t}) \\
		&+ p_\eps \cdot (x - x_\eps)
		+ q_\eps \cdot (y - y_\eps)\\
		&+ 
		\frac{1}{2}  A_{\delta,\eps} 
		\begin{pmatrix}
			x - x_\eps \\
			y - y_\eps
		\end{pmatrix}
		\cdot
		\begin{pmatrix}
			x - x_\eps \\
			y - y_\eps
		\end{pmatrix},
	\end{split}
	\end{equation}
	with equality if and only if $x = x_\eps$ and $y = y_\eps$. 
	
	We next claim, shrinking $c_0$ and $\eps$ if necessary, that
	\begin{equation}\label{boundaryholds?}
		p_\eps \cdot n(x_\eps) > 0 \quad \text{if } x_\eps \in \del \Omega
	\quad \text{and} \quad
		q_\eps \cdot n(y_\eps) > 0 \quad \text{if } y_\eps \in \del \Omega.
	\end{equation}
	Indeed, in either case \eqref{case:halfspace} or \eqref{case:unifconv}, the claim is a consequence of respectively Lemma \ref{L:halfspacetestfn} or Lemma \ref{L:unifconvtestfn}. In the radial case, that is, if \eqref{case:radial} holds, then the restriction on $\delta$ implies also that $|x-y| \le c_0^{1/2}$. Thus, further shrinking $c_0$, we have both $\delta \in (0,\delta_0)$ and $|x-y| \le \nu_0$, where $\delta_0$ and $\nu_0$ are as in Lemma \ref{L:radialHtestfn}. The claim is then proved by further shrinking $\eps$ if necessary.
	
	Appealing to Lemma \ref{L:Youngtrick}, we further develop the inequality \eqref{Taylor} with extra variables $\xi$ and $\eta$:
	\begin{equation}\label{moreTaylor}
	\begin{split}
		\Phi_{\delta,\eps}(x,y,\zeta_{t_\eps} - \zeta_{\hat t}, \eta_{t_\eps} - \eta_{\hat t}, \zeta_{\hat t} - \eta_{\hat t})
		&\le \Phi_{\delta,\eps}(x_\eps, y_\eps, \zeta_{t_\eps} - \zeta_{\hat t}, \eta_{t_\eps} - \eta_{\hat t}, \zeta_{\hat t} - \eta_{\hat t})\\
		&+ p_\eps \cdot (x - x_\eps - \xi) + q_\eps \cdot (y - y_\eps - \eta)\\
		&+ \left( \frac{1}{\delta} + \norm{ A_{\delta,\eps} }\right) \left( |x - x_\eps - \xi|^2 + |y - y_\eps - \eta|^2 \right) \\
		&+p_\eps \cdot \xi + q_\eps \cdot \eta + \frac{1}{2} ( A_{\delta,\eps} +  \delta A_{\delta,\eps}^2)
		\begin{pmatrix}
			\xi\\
			\eta
		\end{pmatrix}
		\cdot 
		\begin{pmatrix}
			\xi\\
			\eta
		\end{pmatrix},
	\end{split}
	\end{equation}
	with equality at $(x,y,\xi,\eta) = (x_\eps,y_\eps,0,0)$. 
	
	Let $S_\zeta(\cdot,t_\eps)$ and $S_\zeta(\cdot,t_\eps)$ denote the solution operators for the Hamilton-Jacobi part of the equation on $\RR^d$, driven by respectively $\zeta$ and $\eta$. Then, for any function $\psi(x,y) = f(x) - g(y)$, the solution of
	\[
		d\Psi = \sum_{i=1}^m \left( H^i(D_x\Psi) \cdot d\zeta^i - H^i(-D_y \Psi) \cdot d\eta^i \right), \quad \Psi(\cdot,\cdot,t_\eps) = \psi
	\]
	satisfies $\Psi(x,y,t) = S_\zeta(t,t_\eps)f(x) - S_\eta(t,t_\eps)g(y)$. Using the monotonicity of the solution operators, we have, for all $t$ close enough to $\hat t$ (including $t_\eps$),
	\begin{align*}
		\Phi_{\delta,\eps}&(x,y, \zeta_t - \zeta_{\hat t}, \eta_t - \eta_{\hat t}, \zeta_{\hat t} - \eta_{\hat t}) - \Phi_{\delta,\eps}(x_\eps,y_\eps, \zeta_{t_\eps} - \zeta_{\hat t}, \eta_{t_\eps} - \eta_{\hat t}, \zeta_{\hat t} - \eta_{\hat t})\\
		&\le S_\zeta(t_\eps,t) \left( p_\eps \cdot (\cdot - x_\eps) + \left( \frac{1}{\delta} + \norm{\tilde A_{\delta,\eps}} \right) |\cdot - x_\eps|^2 \right)(x - \xi)\\
		&- S_\eta(t_\eps,t) \left( -q_\eps \cdot (\cdot - y_\eps) - \left( \frac{1}{\delta} + \norm{\tilde A_{\delta,\eps}} \right)|\cdot - y_\eps|^2 \right)(y - \eta) \\
		&+ p_\eps \cdot \xi + q_\eps \cdot \eta + \frac{1}{2} B_{\delta,\eps}
		\begin{pmatrix}
			\xi\\
			\eta
		\end{pmatrix}
		\cdot 
		\begin{pmatrix}
			\xi\\
			\eta
		\end{pmatrix},
	\end{align*}
	where
	\[
		B_{\delta,\eps} :=  A_{\delta,\eps} +  \delta A_{\delta,\eps}^2 + 
		\eps
		\begin{pmatrix}
			\Id & 0\\
			0 & \Id
		\end{pmatrix},
	\]
	with equality only if $(x,y,t,\xi,\eta) = (x_\eps,y_\eps,t_\eps, 0,0)$. 
	
	We now set
	\[
		\Psi_+(x,t) = S_\zeta(t_\eps,t) \left( p_\eps \cdot (\cdot - x_\eps) + \left( \frac{1}{\delta} + \norm{ A_{\delta,\eps}} \right) |\cdot - x_\eps|^2 \right)(x),
	\]
	\[
		\Psi_-(y,t) = S_\eta(t_\eps,t) \left( -q_\eps \cdot (\cdot - y_\eps) - \left( \frac{1}{\delta} + \norm{ A_{\delta,\eps}} \right)|\cdot - y_\eps|^2 \right)(y),
	\]
	\[
		\oline{u}(\xi,t) := \max_{x \in \oline{\Omega}} \left\{ u(x,t) - \Psi_+(x - \xi,t) \right\},
	\]
	and
	\[
		\uline{v}(\eta,t) := \min_{y \in \oline{\Omega}} \left\{ v(y,t) - \Psi_-(y-\eta,t)  \right\}.
	\]
	Then
	\[
		\oline{u}(\xi,t) - \uline{v}(\eta,t) - p_\eps \cdot \xi - q_\eps \cdot \eta - \frac{1}{2} B_{\delta,\eps}
		\begin{pmatrix}
			\xi\\
			\eta
		\end{pmatrix}
		\cdot
		\begin{pmatrix}
			\xi\\
			\eta
		\end{pmatrix}
		 - \mu t
	\]
	attains a strict maximum at $(\xi,\eta,t) = (0,0,t_\eps)$. 
	
	Note that we have ensured that, in the definition of $\oline{u}(0,t_\eps)$, the maximum is reached only at $x = x_\eps$. Therefore, for $(\xi,t)$ in a neighborhood of $(0,t_\eps)$, the $\argmax$ defining $\oline{u}(\xi,t)$ contains only points close to $x_\eps$. We make the same comments about $y_\eps$ and the $\argmin$ of $\uline{v}(\eta,t)$ for $(\eta,t)$ close to $(0,t_\eps)$. Moreover, $p_\eps \cdot n(x) > 0$ and $q_\eps \cdot n(y)$ whenever $x_\eps$ or $y_\eps$ belong to the boundary of $\Omega$ and whenever $x$ is close to $x_\eps$ and $y$ is close to $y_\eps$.
	
	It follows from Definition \ref{D:nonsmoothH} that, in a neighborhood of $(0,t_\eps)$, $\oline{u}$ is a subsolution of $\del_t \oline{u} \le \mcl F(D^2 \oline{u}, D \oline{u}, \oline{u}, \xi,t)$ $\uline{v}$ is a supersolution of $\del_t \uline{v} \ge \mcl G(D^2 \uline{v}, D \uline{v}, \uline{v}, \eta,t)$, where
	\[
		\mcl F(X,p,r,\xi,t)
		= \inf_{x \in \argmax_{u(\cdot,t) - \Psi_+(\cdot - \xi,t) } } F(X,p,r + \Psi_+(x-\xi,t), x,t)
	\]
	and
	\[
		\mcl G(Y,q,s,\eta,t)
		= \sup_{y \in \argmin_{v(\cdot,t) - \Psi_-(\cdot - \eta,t) } } F(Y,q,s + \Psi_-(y - \eta,t),y,t).
	\]
	In particular, for some $C > 0$, $\del \oline{u}/dt \le C$ and $\del \uline{v}/dt \ge -C$ in a neighborhood of $(0,t_\eps$).
	
	We now appeal directly to \cite[Theorem 7]{CImax} and find that there exist $a,b \in \RR$ and $X_\eps,Y_\eps \in \mbb S^d$ such that $(a,p_\eps, X_\eps) \in \oline{\mcl P}^+ \oline{u}(0,t_\eps)$, $(b, q_\eps Y_\eps) \in \oline{\mcl P}^- \uline{v}(0,t_\eps)$ (the definitions of the limiting sub- and superjets being recalled in the Appendix below), $a - b = \mu$, and, for some universal constant $C > 0$,
	\[
		-\frac{C}{\delta}
		\begin{pmatrix}
			\Id & 0\\
			0 & \Id
		\end{pmatrix}
		\le
		\begin{pmatrix}
			X_\eps & 0 \\
			0 & -Y_\eps
		\end{pmatrix}
		\le
		B_{\delta,\eps} + B_{\delta,\eps}^2
		\le \frac{C}{\delta}
		\begin{pmatrix}
			\Id & - \Id \\
			- \Id & \Id
		\end{pmatrix}
		+
		C\eps
		\begin{pmatrix}
			\Id & 0 \\
			0 & \Id
		\end{pmatrix}.
	\]
	By the continuity of $F$, we thus have
	\begin{align*}
		\mu &\le F(X_\eps,p_\eps, u(x_\eps,t_\eps),  x_\eps,t_\eps) - F(Y_\eps,-q_\eps,v(y_\eps,t_\eps), y_\eps,t_\eps)\\
		&\le F(X_\eps,p_\eps, v(y_\eps,t_\eps) + \Phi_{\delta,\eps}(x_\eps,y_\eps,\zeta_{t_\eps} - \zeta_{\hat t}, \eta_{t_\eps} - \eta_{\hat t}, \zeta_{\hat t} - \eta_{\hat t}) ,x_\eps,t_\eps) \\
		& \qquad - F(Y_\eps,-q_\eps,v(y_\eps,t_\eps),y_\eps,t_\eps),
	\end{align*}
	where the last inequality follows from \eqref{A:Fmonotone} and \eqref{pospart}. Note that, in view of Lemma \ref{L:quadphieps=0}, Lemma \ref{L:testfnproperties}\eqref{A:Phieps=0}, and \eqref{bigmaxball}, 
	\[
		\Phi_{\delta,\eps}(x_\eps,y_\eps,\zeta_{t_\eps} - \zeta_{\hat t}, \eta_{t_\eps} - \eta_{\hat t}, \zeta_{\hat t} - \eta_{\hat t})
		\le \frac{r_\delta^2}{2\delta} + C\delta + \tilde \omega_{\delta,\eps},
	\]
	where, for fixed $\delta$, $\lim_{\eps \to 0} \tilde \omega_{\delta,\eps} = 0$. It then follows from assumption \eqref{A:F} on $F$ that, for some $R,C > 0$ independent of $\delta$ and $\eps$, for all sufficiently small $\eps$,
	\[
		\mu \le \omega_{M,C}\left(\frac{ r_\delta^2 + |p_\eps - q_\eps|^2}{\delta} + (1 + |p_\eps| + |q_\eps|)r_\delta + \frac{r_\delta^2}{2\delta} + C\delta + \tilde \omega_{\delta,\eps} + C\eps \right).
	\]	
	By Lemma \ref{L:quadphieps=0},  along a subsequence as $\eps \to 0$, $p_\eps$ and $- q_\eps$ converge to some $\hat p$ satisfying $|\hat p| \le C(1 + r_\delta/\delta)$. Therefore, sending $\eps \to 0$, we obtain, redefining $C$ if necessary,
	\[
		\mu \le \omega_{M,C}\left( Cr_\delta + C\frac{r_\delta^2}{\delta} + C\delta\right) =: \omega_\delta.
	\]
	We thus reach a contradiction for any $\mu > \omega_\delta$.
\end{proof}

\begin{remark}\label{R:omegadelta=0}
	If $F$ is independent of $u$ and $x$, then we may take $\omega_\delta = 0$ in the statement of Proposition \ref{P:keyestimate}. Indeed, after the use of \cite[Theorem 7]{CImax} in the above proof, we have $\mu \le F(X_\eps,p_\eps,t_\eps) - F(Y_\eps,q_\eps,t_\eps)$, and, upon sending $\eps \to 0$ along some subsequence, we have $X_\eps \to X$ and $Y_\eps \to Y$ where $X \le Y$. This gives the contradiction $\mu \le 0$.
\end{remark}
}

\subsection{Wellposedness}

We now use Proposition \ref{P:keyestimate} to prove the existence, uniqueness, and stability of solutions to \eqref{E:general}, assuming that $H$ satisfies \eqref{A:DCH} and \eqref{A:HLip}, $F$ satisfies \eqref{A:Fcts}, \eqref{A:Fmonotone}, and \eqref{A:F}, and $\Omega$ satisfies \eqref{A:C1convexdomain}, as well as one of the following: $\Omega$ is a half space, $\Omega$ satisfies \eqref{A:uniformlyconvexdomain}, or $H$ satisfies \eqref{A:Hradial}.

\subsubsection{The comparison principle}

The first use of Proposition \ref{P:keyestimate} is to prove the comparison principle.

\begin{theorem}\label{T:firstordercomparison}
	 Fix $\zeta \in C([0,T],\RR^m)$, and let $u \in BUSC(\oline{\Omega} \times [0,T])$ and $v \in BLSC(\oline{\Omega} \times [0,T])$ be respectively a bounded sub- and super-solution of \eqref{E:general}. Then
	\[
		\sup_{(x,t) \in \oline{\Omega} \times [0,T]} \Big\{ \big(u(x,t) - v(x,t)\big)_+\Big\} \le \sup_{x \in \oline{\Omega}} \Big\{\big( u(x,0) - v(x,0) \big)_+\Big\}.
	\]
\end{theorem}

\begin{proof}
Setting $\zeta = \eta$, we note that \eqref{closepaths} is satisfied. Then, for the values of $\delta$ and $\gamma_\delta$ specified by Proposition \ref{P:keyestimate}, we have
\begin{align*}
	\sup_{(x,y,t) \in \oline{\Omega} \times \oline{\Omega} \times [0,T]} &\Big\{ \big( u(x,t) - v(y,t) - \phi_\delta(x-y,0) \big)_+ \Big\}\\
	&le
	\sup_{x,y\in \oline{\Omega}} \left\{ \big( u(x,0) - v(y,0) - \phi_\delta(x-y,0) \big)_+ \right\} + \omega_\delta t.
\end{align*}
Sending $\delta \to 0$ and invoking Lemma \ref{L:quadphieps=0} immediately yields the result.
\end{proof}

\subsubsection{Existence and stability}

Just as in Section \ref{sec:smoothH}, the comparison principle implies the existence of a unique solution of the the initial value problem, as well the continuity of the solution map $\zeta \mapsto u$.

\begin{corollary}\label{C:firstorderwellposed}
	Under the same conditions as Theorem \ref{T:firstordercomparison}, for any $u_0 \in BUC(\oline{\Omega})$ and $\zeta \in C([0,T],\RR^m)$, there exists a unique solution of \eqref{E:general} with $u(\cdot,0) = u_0$. Moreover, if, for $n \in \NN$, $\zeta^n \in C([0,T],\RR^m)$, $u_0^n \in BUC(\oline{\Omega})$, and, as $n \to \oo$, $\zeta^n$ converges to $\zeta$ and $u_0^n$ converges to $u_0$ uniformly, then, if $u^n$ is the corresponding solution of \eqref{E:neumann}, then, as $n \to \oo$, $u^n$ converges locally uniformly to $u$.
\end{corollary}

\begin{proof}
	As in the proof of Theorem \ref{T:smoothexistence}, the argument rests on showing that the half-relaxed limit $u^\star$, which is a sub-solution, satisfies $u^\star(\cdot,0) \le u_0$ (and a similar statement for the lower half-relaxed limit $u_\star$). 
	
	Fix $x_0 \in \Omega$ and let $\phi \in C^2(\RR^d)$ be convex and such that $u_0 \le \phi$ and $u_0(x_0) = \phi(x_0)$. Then, by Lemma \ref{L:C11solutions}, there exists $\delta >0$, independent of $n$, and a family $(\Phi^n)_{n \in \NN} \subset C([0,\delta], C^{1,1}(\RR^d))$ of solutions of
	\[
		d\Phi^n = \sum_{i=1}^m H^i(D\Phi^n) \cdot d\zeta^{i,n} \quad \text{in } \RR^d \times [0,\delta]
		\quad \text{and} \quad \Phi^n(\cdot,0) = \phi.
	\]
	Moreover, there exists $C > 0$ independent of $n$ such that $0 \le D^2 \Phi^n \le C$. Therefore, for some $\tilde C > 0$ independent of $n$, $\Phi^n(x,t) + \tilde C t$ is a super-solution of \eqref{E:general} corresponding to the path $\zeta^n$. We conclude as in the proof of Theorem \ref{T:smoothexistence} that $u^\star(\cdot,0) \le \phi$, and, in particular, $u^\star(x_0,0) \le \phi(x_0) = u_0(x_0)$. The result now follows because $x_0$ was arbitrary.
\end{proof}

\subsection{Further results for homogenous $F$} We next prove some further properties of solutions under the additional assumption that $F$ depends only on $D^2 u$ and $Du$, that is, 
\begin{equation}\label{A:Fnonsmooth}
	F \in C(\mbb S^d \times \RR^d) \quad \text{is non-decreasing in } X \in \mbb S^d.
\end{equation}
We first record some simple bounds for the unique solution.

\begin{lemma}\label{L:solutionbound}
	Let $u_0 \in BUC(\oline\Omega)$ and $\zeta \in C_0([0,T],\RR^m)$, and let $u$ be the unique solution of \eqref{E:general} with $u(\cdot,0) = u_0$. Then
	\[
		\sup_{(x,t) \in \oline{\Omega} \times [0,T]} \left| u(x,t) - \sum_{i=1}^m H^i(0) \zeta^i(t) - F(0,0) t\right| \le \norm{u_0}_\oo.
	\]
\end{lemma}

\begin{proof}
	This is a consequence of the comparison principle Theorem \ref{T:firstordercomparison}, as well as the fact that $(x,t) \mapsto \pm \norm{u_0}_\oo + \sum_{i=1}^m H^i(0) \zeta^i(t) + F(0,0) t$ are solutions of \eqref{E:general}.
\end{proof}

\subsubsection{Path stability}

The existence result in the previous sub-section was derived from the stability of the solution operator for \eqref{E:general} in the path variable. In this section, we make this more quantitative when $F$ satisfies \eqref{A:Fnonsmooth}.
%
\begin{theorem}\label{T:firstorderstabilityestimate}
	Fix $u_0 \in BUC(\RR^d)$. Then there exists a modulus of continuity $\omega : [0,\oo) \to [0,\oo)$ such that, if $\zeta^1,\zeta^2 \in C_0([0,T],\RR^m)$ and if $u^1,u^2$ are the corresponding solutions of \eqref{E:general}, then
	\[
		\norm{u^1 - u^2}_{\oo,\oline{\Omega} \times [0,T]} \le \omega\left( \norm{\zeta^1 - \zeta^2}_{\oo,[0,T]} \right).
	\]
\end{theorem}

\begin{proof}
	We first note that we may assume that $H(0) = 0$ and $F(0,0) = 0$. Otherwise, we may replace $u^j$, $j = 1,2$, with
	\[
		\tilde u^j(x,t) = u^j(x,t) - \sum_{i=1}^m H^i(0) \zeta^{j,i}(t) - F(0,0)t.
	\]
	We also prove the result for $u_0 \in C^{0,1}(\oline{\Omega})$.
	The general result then follows by replacing $u_0$ by Lipschitz approximations and using the contractive property of the solution operator implied by the comparison principle.
	
	By Lemma \ref{L:solutionbound}, 
	\[
		\norm{u^1}_{\oo,\oline{\Omega} \times [0,T]} \vee \norm{u^2}_{\oo, \oline{\Omega} \times [0,T]} \le \norm{u_0}_{\oo,\oline\Omega},
	\]
	and so the value $M$ in Proposition \ref{P:keyestimate} depends only on $\norm{u_0}_\oo$.
	
	Let $\delta_0 >0$ be such that the conclusion of Proposition \ref{P:keyestimate} holds for $0 < \delta < \delta_0$. We first assume that $\zeta^1$ and $\zeta^2$ satisfy $\norm{\zeta^1 - \zeta^2}_{\oo,[0,T]} <  \gamma_{\delta_0}$, with $\gamma_\delta = c_0M^{-1} \delta^r$, the exact value of $r$ being specified in Proposition \ref{P:keyestimate}. 
By Lemma \ref{L:quadphieps=0} and Proposition \ref{P:keyestimate} (recall that $\omega_\delta = 0$ by Remark \ref{R:omegadelta=0}), we then have, for all $(x,t) \in \oline{\Omega} \times [0,T]$, for all $\delta \in (0,\delta_0)$ with $\gamma_{\delta} > \norm{\zeta^1 - \zeta^2}_{\oo,[0,T]}$,
	\begin{align*}
		u^1(x,t) - u^2(x,t) 
		&\le \phi_\delta(0,\zeta^1_t - \zeta^2_t) + \sup_{x,y \in \oline{\Omega}} \left\{ u^1(x,0) - u^2(y,0) - \phi_\delta(x-y,0) \right\}\\
		&\le C\delta + \sup_{x,y \in \oline{\Omega}} \left\{ L |x-y| - \frac{ (|x-y| - C\delta)_+^2}{\delta} \right\} \\
		&\le C_L \delta.
	\end{align*}
	We conclude that, for some constant $C > 0$ depending only on $u_0$, and after applying the same argument with the roles of $u^1$ and $u^2$ reversed,
	\[
		\norm{u^1 - u^2}_{\oo,[0,T]} \le C \norm{\zeta^1 -\zeta^2}_{\oo,[0,T]}^{1/r}
	\]
	whenever $\norm{\zeta^1 - \zeta^2}_{\oo,[0,T]} \le \gamma_{\delta_0}$. 
	
	{
	Suppose now that, for some $N = 2,3,\ldots$, $(N-1)\gamma_{\delta_0} < \norm{\zeta^1 - \zeta^2}_{\oo,[0,T]} \le N\gamma_{\delta_0}$. Then iterating the above result gives
	\begin{align*}
		\norm{u^1 - u^2}_{\oo,[0,T]} &\le CN\left( \frac{ \norm{\zeta^1 -\zeta^2}_{\oo,[0,T]} }{N} \right)^{1/r}\\
		&\le CN^{1 - \frac{1}{r}}  \norm{\zeta^1 -\zeta^2}_{\oo,[0,T]}^{1/r}.		
	\end{align*}
	We conclude since
	\[
		N \le 1 + \frac{ \norm{\zeta^1 - \zeta^2}_{\oo,[0,T]}}{c_0 M^{-1} \delta_0^r}.
	\]
	}
\end{proof}

\subsubsection{Continuity estimates}

Another consequence of the homogeneity assumption \eqref{A:Fnonsmooth} is the continuity of solutions of \eqref{E:neumann} in the space variable, depending only on the initial datum $u_0$, and not on the path $\zeta$ or time $t > 0$.

\begin{theorem}\label{T:firstordercontinuity}
	Fix $u_0 \in BUC(\oline{\Omega})$. Then there exists a modulus of continuity $\omega: [0,\oo) \to [0,\oo)$ such that, for any $\zeta \in C([0,T],\RR^m)$, if $u$ is the unique solution of \eqref{E:neumann} with $u(\cdot,0) = u_0$, then, for all $t \in [0,T]$, $u(\cdot,t)$ is $\omega$-continuous.
\end{theorem}

\begin{proof}
	We may assume without loss of generality that $u_0$ is Lipschitz with constant $L > 0$. The general result follows from approximating $u_0$ and using the contractive property of the equation. As in the proof of Theorem \ref{T:firstorderstabilityestimate}, we may also assume that $H(0) = 0$ and $F(0,0) = 0$.
	
	In view of Lemma \ref{L:quadphieps=0}, taking $u = v$ and $\zeta = \eta$ in Proposition \ref{P:keyestimate}, and noting $\omega_\delta = 0$ in view of Remark \ref{R:omegadelta=0}, yields, for all $t \in [0,T]$ and $\delta$ sufficiently small, depending possibly on $\norm{u_0}_\oo$,
	\begin{align*}
		u(x,t) - u(y,t) &\le \phi_\delta(x-y,0) + \sup_{x,y \in \oline{\Omega} } \left\{ u_0(x)  - u_0(y) - \phi_\delta(x-y,0) \right\}\\
		&\le C \frac{|x-y|^2}{\delta} + C \delta + \max_{r \ge 0} \left\{Lr - \frac{ C r^2}{2\delta} \right\}\\
		&= C \frac{|x-y|^2}{\delta}  + C\delta.
	\end{align*}
	If $x$ and $y$ are sufficiently close, we may take $\delta \approx |x-y|$ to conclude that $u$ is locally Lipschitz, and thus globally Lipschitz because the constant does not depend on $x$ or $y$.
\end{proof}

\subsubsection{Monotonicity in the path variable}

The comparison principle, Theorem \ref{T:firstordercomparison}, implies that the solution operator for \eqref{E:neumann} is order-preserving in the initial data. Under the additional assumption that $F$ satisfy \eqref{A:Fnonsmooth} and $H$ be radial and convex, it turns out that the same is true in the path variable.

\begin{theorem}\label{T:firstordermonotonicity}
	Assume that $H$ satisfies \eqref{A:Hradial} and \eqref{A:convexH}. Let $u_0 \in BUC(\oline{\Omega})$ and $\zeta^1, \zeta^2 \in C_0([0,T],\RR^m)$ satisfy $\zeta^{1,i} \le \zeta^{2,i}$ for all $i = 1,2,\ldots,m$, and let $u^1,u^2 \in BUC(\oline{\Omega} \times [0,T])$ be the corresponding solutions of \eqref{E:general}. Then $u^1 \le u^2$ on $\oline{\Omega} \times [0,T]$.
\end{theorem}

To prove Theorem \ref{T:firstordermonotonicity}, we need the following variation of Proposition \ref{P:keyestimate}:

\begin{proposition}\label{P:keyestimateradial}
	If $H$ satisfies \eqref{A:Hradial} and \eqref{A:convexH}, then the conclusions of Proposition \ref{P:keyestimate} continue to hold if the condition on $\gamma_\delta$ in \eqref{case:radial} is replaced with $\gamma_\delta = c_0 \delta^2$, and, for some sufficiently small universal $\theta_0 \in (0,1)$, \eqref{closepaths} is replaced with 	
	\begin{equation}\label{closeorderedpaths}
		\eta_t - \theta_0M^{-2} \le \zeta_t \le \eta_t + \gamma_\delta \quad \text{for all } t \in [0,T].
	\end{equation}
\end{proposition}

\begin{proof}
	Most of the proof goes just as in that of Proposition \ref{P:keyestimate}, in view of Lemma \ref{L:convexH}. The key difference is in establishing the boundary condition \eqref{boundaryholds?}. In view of the estimate for $\phi_\delta$ in Lemma \ref{L:phideltaconvex}, the points $x_\eps$ and $y_\eps$ satisfy, for sufficiently small $\eps$,
	\[
		|x_\eps-y_\eps| \le M(\delta + \eta_{t_\eps} - \zeta_{t_\eps})^{1/2} \le (c_0 + \theta_0)^{1/2}.
	\]
	Shrinking $\theta_0$ and $c_0$, if necessary, the right-hand side is less than $\nu_0$, where $\nu_0$ is as in Lemma \ref{L:convradialboundary}. Therefore, \eqref{boundaryholds?} holds, and the rest of the argument goes through just as in the proof of Proposition \ref{P:keyestimate}.
\end{proof}
	
\begin{proof}[Proof of Theorem \ref{T:firstordermonotonicity}]
	We have $H(0) = 0$, and we may assume also without loss of generality that $F(0,0) = 0$. Then, as in previous proofs, the value $M$ depends only on $\norm{u_0}_{\oo}$.
	
	By Proposition \ref{P:keyestimateradial}, there exists a constant $\tilde \theta_0 > 0$ depending on universal quantities and $\norm{u_0}_\oo$ such that, if $-\tilde \theta_0 \le \zeta^1 - \zeta^2 \le 0$, then, for all $\delta > 0$ sufficiently small, $\gamma_\delta = c_0 \delta^2$, and $(x,t) \in \oline{\Omega} \times [0,T]$,
	\[
		u^1(x,t) - u^2(x,t) \le \sup_{x,y} \left\{ u_0(x) - u_0(y) - \phi_\delta(x-y , \zeta_t - \eta_t) \right\}.
	\]
	Sending $\delta \to 0$, we have $u^1 \le u^2$. For arbitrary paths satisfying $\zeta^1 \le \zeta^2$, the result may be iterated.
\end{proof}

\section{Geometric equations}\label{sec:geo} We now turn to the study of geometric equations. We assume that $H$ is positively $1$-homogenous, that is,
\begin{equation}\label{A:Hgeometric}
	H(\lambda p) = \lambda H(p) \quad \text{for all } p \in \RR^d, \; \lambda > 0.
\end{equation}
Such an $H$ satisfies the difference-of-convex-functions property \eqref{A:DCH} if, for instance, $H$ is $C^2$ in a neighborhood of $S^{d-1}$ (see \cite[Appendix B]{LSS}), in which case $H^i_1$ and $H^i_2$ are also both $1$-homogenous (and therefore, so is $G$ in \eqref{Gfn}). In fact, such Hamiltonians are also globally Lipschitz. We then note that level-set equations dealing with first-order geometric motions are already covered by the results in Section \ref{sec:nonsmoothwellposed}. We focus here on second-order problems, for which 
\begin{equation}\label{A:Fgeometric}
	\left\{
	\begin{split}
	&F \in C\left(\mbb S^d \times (\RR^d \backslash \{0\}) \right) \quad \text{is nondecreasing in the $\mbb S^d$ variable and}\\
	&F(\lambda X + \mu p \otimes p, \lambda p) = \lambda F(X,p) \quad \text{for all } (X,p) \in \mbb S^d \times \RR^d, \; \lambda > 0, \text{ and } \mu \in \RR.
	\end{split}
	\right.
\end{equation}
Functions satisfying \eqref{A:Fgeometric} are allowed to have a singularity at $p = 0$. Note, however, that \eqref{A:Fgeometric} implies that $F(0,\cdot)$ is continuous at $p = 0$:
\begin{equation}\label{Fat0}
	F^*(0,0) = F_*(0,0) = 0.
\end{equation}
A standard example is
\[
	F(X,p) = \tr\left[ \left( \Id - \frac{p}{|p|} \otimes \frac{p}{|p|} \right) X \right],
\]
which is the nonlinearity from the level-set equation for perturbed mean curvature flow, and which is the focus of the application in the next section.

\subsection{Test functions}
{
We construct a particular test function much in the same way as in Section \ref{sec:testfn}. The key difference is that its Hessian vanishes, as $\eps \to 0$, whenever its gradient does (see Lemma \ref{L:geounifconvex}\eqref{geoHessian} below). This is done in order to deal with the singularity of the nonlinearity $F$.
}

We first define, for $\delta \in (0,1)$ and $\gamma_\delta$ as in \eqref{intervalofexistence},
\begin{equation}\label{geophidelta}
	\phi_{\delta}(z,\rho) = \sup_p \left\{ p \cdot z - \frac{3\delta}{4} |p|^{4/3} - \delta |p| + \sum_{i=1}^m \rho_i H^i(p) - 3\gamma_\delta G(p) \right\},
\end{equation}
and, for $\eps \in (0,1)$, $G$ as in \eqref{Gfn}, and $\psi$ as in \eqref{psi},

\begin{equation}\label{testfngeo}
\begin{split}
	\Phi_{\delta,\eps}(x,y,\sigma,\tau,\rho) = \sup_{p,u,v} &\Big\{ (p+u) \cdot x - (p-v) \cdot y - \frac{3\delta}{4} |p|^{4/3} - \delta|p|  - \eps \psi^*\left( \frac{u}{\eps} \right)  - \eps \psi^*\left( \frac{v}{\eps} \right) \\
& - \sum_{i=1}^m \left(  \sigma_i H^i(p+u) - \tau_i H^i(p-v) + \rho_i H^i(p) \right) \\
&- \gamma_\delta(G(p+u) + G(p-v) + G(p)) \Big\}.
\end{split}
\end{equation}
In view of the positive $1$-homogeneity of $H$ (and therefore $G$), we have, for some $C > 0$, for all $\rho \in (3\gamma_\delta,3\gamma_\delta)^m$, and for some $\nu_\rho \in S^{d-1}$,
\[
 	-C\gamma_\delta(p \cdot \nu_\rho - 1) \le 3\gamma_\delta G(p) - \sum_{i=1}^m \rho_i H^i(p) \le C\gamma_\delta|p|.
\]
We then have the following, which is proved exactly as in Lemma \ref{L:quadphieps=0}.

\begin{lemma}\label{L:quarticphieps=0}
The function $\phi_\delta$ in \eqref{testfngeo} belongs to $C^{1,1}(\RR^d \times (-3\gamma_\delta,3\gamma_\delta)^m)$, and there exists $C > 0$ such that, for all $\delta \in (0,1)$, $z \in \RR^d$, and $\rho \in (-3\gamma_\delta,3\gamma_\delta)^m$,
\begin{equation}\label{geophideltabounds}
	\frac{ [ (|z| - C\delta)_+]^4 }{4 \delta^3} \le \phi_\delta(z,\rho) \le \frac{ [ (|z| - \delta)_+]^4 }{4 \delta^3} 
\end{equation}
\end{lemma}

Analogously to Lemma \ref{L:testfnproperties}, we can prove the following properties of $\Phi_{\delta,\eps}$:
\begin{lemma}\label{L:geounifconvex}
	The function $\Phi_{\delta,\eps}$ defined in \eqref{testfngeo} satisfies the following:
	\begin{enumerate}[(i)]
	\item\label{geoeqns} For all $\delta,\eps \in (0,1)$, $\Phi_{\delta,\eps} \in C^{1,1}(\RR^d \times \RR^d \times (-\gamma_\delta,\gamma_\delta)^{3m})$, and, for all $i = 1,2,\ldots,m$,
	\[
		\frac{\partial \Phi_{\delta,\eps}}{\partial \sigma_i } = H^i(D_x \Phi_{\delta,\eps}) \quad \text{and} \quad \frac{\partial \Phi_{\delta,\eps}}{\partial \tau_i} = - H^i(-D_y \Phi_{\delta,\eps})
	\]
	in $\RR^d \times \RR^d \times (-\gamma_\delta,\gamma_\delta)^{3m}$.
	\item\label{geopenalize} For some $C > 0$ and for all $x,y \in \oline{\Omega}$, $\sigma,\tau,\rho \in (-\gamma_\delta,\gamma_\delta)^m$, and $\delta,\eps \in (0,1)$,
	\[
		\Phi_{\delta,\eps}(x,y,\sigma,\tau,\rho) \ge \eps \psi(x) + \eps \psi(y) - C\gamma_\delta.
	\]

	\item\label{geoeps=0} If $R > 0$ and $r_\eps > 0$ is such that $\lim_{\eps \to 0} r_\eps = 0$, then, for all $\delta > 0$,
	\begin{align*}
		&\lim_{\eps \to 0} \sup \Big\{ |\Phi_{\delta,\eps}(x,y,\sigma,\tau,\rho) - \phi_\delta(x-y,\sigma-\tau + \rho)| \\
	&\qquad  : |x-y| \le R, \; \eps\psi(x) + \eps \psi(y) \le r_\eps, \;  \sigma,\tau,\rho \in (-\gamma_\delta,\gamma_\delta)^m\Big\}= 0.
	\end{align*}
	\item\label{geoHessian} For $\delta,\eps > 0$, there exists a continuous, nonnegative symmetric matrix-valued map
	\[
		\oline{\Omega} \times \oline{\Omega} \times (-\gamma_\delta,\gamma_\delta)^3 \ni (x,y,\sigma,\tau,\rho) \mapsto A_{\delta,\eps}(x,y,\sigma,\tau,\rho)
	\]
	such that, for $\sigma,\tau,\rho \in (-\gamma_\delta,\gamma_\delta)^m$ and almost every $x,y \in \oline{\Omega}$,
	\[
		0 \le D^2_{(x,y)} \Phi(x,y,\sigma,\tau,\rho) \le A_{\delta,\eps}(x,y,\sigma,\tau,\rho),
	\]
	and, uniformly for bounded $x,y \in \RR^d$ and $\sigma,\tau,\rho \in (-3\gamma_\delta, 3\gamma_\delta)^3$,
	\[
		\limsup_{\eps \to 0} A_{\delta,\eps}(x,y,\sigma,\tau,\rho) \le \frac{3}{\delta} |D\phi_\delta(x-y,\sigma-\tau +\rho)|^{2/3}
	\begin{pmatrix}
		\Id & -\Id \\
		-\Id & \Id
	\end{pmatrix}.
	\]
	\end{enumerate}
\end{lemma}

\begin{proof}
	We first prove item \eqref{geoHessian}. This establishes the $C^{1,1}$-regularity, and part \eqref{geoeqns} is then a consequence of Lemma \ref{L:C11solutions}. Parts \eqref{geopenalize} and \eqref{geoeps=0} then follow exactly as in the proof of Lemma \ref{L:testfnproperties}.
		
	We rewrite, for some convex function $L: \RR^{3d} \to \RR$,
	\begin{align*}
		\Phi_{\delta,\eps}&(x,y,\sigma,\tau,\rho) \\
		&= \sup_{p,u,v} \Big\{ u \cdot x + v \cdot y - \frac{3\delta}{4} |p|^{4/3}  - \eps \psi^*\left( \frac{u-p}{\eps} \right)  - \eps \psi^*\left( \frac{v+p}{\eps} \right) 
 - L(p,u,v)  \Big\}.
	\end{align*}
	Define
	\[
		\mcl L(p,u,v) := \frac{3\delta}{4} |p|^{4/3}  + \eps \psi^*\left( \frac{u-p}{\eps} \right) + \eps \psi^*\left( \frac{v+p}{\eps} \right)+ L(p,u,v).
	\]
	Then $\mcl L$ is strictly convex and, for $p \ne 0$,
	\[
		D^2 \mcl L(p,u,v) \ge 
		\begin{pmatrix}
			\frac{\delta}{3} |p|^{-2/3} \Id & 0 & 0 \\
			0 & 0 & 0 \\
			0 & 0 & 0
		\end{pmatrix}
		+
		\frac{\kappa}{\eps}
		\left[
		\begin{pmatrix}
			\Id & -\Id & 0 \\
			-\Id & \Id & 0 \\
			0 & 0 & 0
		\end{pmatrix}
		+
		\begin{pmatrix}
			\Id& 0 & \Id\\
			0 & 0 & 0 \\
			\Id & 0 & \Id
		\end{pmatrix}
		\right]
		=: B_{\delta,\eps}(|p|).
	\]
	We compute, for $r \ge 0$,
	\[
		C_{\delta,\eps}(r) := B_{\delta,\eps}(r)^{-1} 
		= \frac{3r^{2/3}}{\delta}
		\begin{pmatrix}
			\Id & \Id & -\Id \\
			\Id & \Id & -\Id \\
			-\Id & -\Id & \Id
		\end{pmatrix}
		+ \frac{\eps}{\kappa}
		\begin{pmatrix}
			0 & 0 & 0 \\
			0 & \Id & 0 \\
			0 & 0 & \Id
		\end{pmatrix}.
	\]
	It follows that, for any $(x,y,\sigma,\tau,\rho) \in \oline{\Omega}^2 \times (-\gamma_\delta,\gamma_\delta)^3$, a unique maximum $(p,u,v) = (p,u,v)(x,y,\sigma,\tau,\rho)$ is attained, with $D\Phi_{\delta,\eps}(x,y,\sigma,\tau,\rho) = (u,v)$. Moreover, for bounded $x$ and $y$, the maximizers $p$, $u$, and $v$ belong to a bounded set, over which $\mcl L$ is uniformly convex. It then follows as in Lemma \ref{L:C11solutions} that $(p,u,v)$ is Lipschitz in $(x,y,\sigma,\tau,\rho)$.
	
	Fix $(x,y), (\hat x, \hat y) \in \oline{\Omega}^2$ and $\sigma,\tau,\rho \in (-\gamma_\delta,\gamma_\delta)^m$, and write
	\[
		(p,u,v) = (p,u,v)(x,y,\sigma,\tau,\rho) \quad \text{and} \quad (\hat p,\hat u,\hat v) = (p,u,v)(\hat x,\hat y, \sigma, \tau, \rho).
	\]
	Assume first that $p = p(x,y,\sigma,\tau,\rho) \ne 0$. Then, if $(\hat x, \hat y)$ is sufficiently close to $(x,y)$, it is also the case that $\hat p \ne 0$. We have
	\[
		\mcl L(\hat p, \hat u, \hat v) \ge \mcl L(p,u,v) + x \cdot (\hat u - u) + y \cdot (\hat v - v) + 
		\frac{1}{2} B_{\delta,\eps}(|p| \vee |\hat p|) 
		\begin{pmatrix}
			\hat p - p \\
			\hat u - u \\
			\hat v - v
		\end{pmatrix}
		\cdot
		\begin{pmatrix}
			\hat p - p \\
			\hat u - u \\
			\hat v - v
		\end{pmatrix}
	\]
	and
	\[
		\mcl L(p, u, v) \ge \mcl L(\hat p,\hat u,\hat v) + \hat x \cdot (u - \hat u) + \hat y \cdot (v - \hat v) + 
		\frac{1}{2} B_{\delta,\eps}(|p| \vee |\hat p|) 
		\begin{pmatrix}
			\hat p - p \\
			\hat u - u \\
			\hat v - v
		\end{pmatrix}
		\cdot
		\begin{pmatrix}
			\hat p - p \\
			\hat u - u \\
			\hat v - v
		\end{pmatrix}.
	\]
	Adding the two inequalities yields
	\[
		B_{\delta,\eps}(|p| \vee |\hat p|) 
		\begin{pmatrix}
			\hat p - p \\
			\hat u - u \\
			\hat v - v
		\end{pmatrix}
		\cdot
		\begin{pmatrix}
			\hat p - p \\
			\hat u - u \\
			\hat v - v
		\end{pmatrix}
		\le 
		(\hat u - u) \cdot (\hat x - x) + (\hat v - v) \cdot (\hat y - y),
	\]
	and so, just as in the proof of Lemma \ref{L:C11solutions},
	\begin{align*}
		(\hat u - u) \cdot (\hat x - x) + (\hat v - v) \cdot (\hat y - y) 
		&\le B_{\delta,\eps}(|p| \vee |\hat p|)^{-1}
		\begin{pmatrix}
			0 \\
			\hat x - x\\
			\hat y - y
		\end{pmatrix}
		\cdot
		\begin{pmatrix}
			0 \\
			\hat x - x\\
			\hat y - y
		\end{pmatrix}
		\\
		&=
		A_{\delta,\eps}(|p| \vee |\hat p|)
		\begin{pmatrix}
			\hat x - x\\
			\hat y - y
		\end{pmatrix}
		\cdot
		\begin{pmatrix}
			\hat x - x\\
			\hat y - y
		\end{pmatrix},
	\end{align*}
	where
	\[
		A_{\delta,\eps}(r) :=
		\frac{3r^{2/3}}{\delta}
		\begin{pmatrix}
			\Id & -\Id \\
			-\Id & \Id
		\end{pmatrix}
		+ \frac{\eps}{\kappa}
		\begin{pmatrix}
			\Id & 0 \\
			0 & \Id
		\end{pmatrix}.
	\]
	We conclude that
	\begin{align*}
		\limsup_{(\hat x, \hat y) \to (x,y)} & \big[ \left( D_x \Phi_{\delta,\eps}(\hat x, \hat y, \sigma,\tau,\rho) - D_x \Phi_{\delta,\eps}( x,  y, \sigma,\tau,\rho) \right)
		\cdot (\hat x - x) \\
		&+ 
		\left( D_y \Phi_{\delta,\eps}(\hat x, \hat y, \sigma,\tau,\rho) - 
		D_y \Phi_{\delta,\eps}( x,  y, \sigma,\tau,\rho) \right) \cdot (\hat y - y) \big]\\
		&\le 
		A_{\delta,\eps}(|p(x,y,\sigma,\tau,\rho)|)
		\begin{pmatrix}
			\hat x - x\\
			\hat y - y
		\end{pmatrix}
		\cdot
		\begin{pmatrix}
			\hat x - x\\
			\hat y - y
		\end{pmatrix},
	\end{align*}
	which means that
	\[
		D^2_{(x,y)} \Phi_{\delta,\eps}(x,y,\sigma,\tau,\rho) \le A_{\delta,\eps}(|p(x,y,\sigma,\tau,\rho)|).
	\]
	If $p(x,y,\sigma,\tau,\rho) = 0$, in which case $B(|p|) = B(0)$ is not well-defined, a similar argument may be applied as above by replacing $B(|p|)$ by $B(M^{-1})$ for arbitrarily large $M$, then letting $M \to \oo$ in the last step.
	
	Finally, uniformly for bounded $(x, y) \in \oline{\Omega}^2$ and $(\sigma,\tau,\rho) \in (-\gamma_\delta,\gamma_\delta)^3$, the unique maximizer $(p,u,v) = (p_\eps,u_\eps,v_\eps)$ for $\Phi_{\delta,\eps}$ satisfies
	\[
		\lim_{\eps \to 0} (p_\eps - u_\eps) = \lim_{\eps \to 0} (p_\eps + v_\eps) = 0 \quad \text{and} \quad
		\lim_{\eps \to 0} u_\eps = - \lim_{\eps \to 0} v_\eps = D\phi_\delta(x-y,\sigma - \tau + \rho),
	\]
	from which the result follows.

\end{proof}

As in Section \ref{sec:testfn}, we require certain strict inequalities to hold on the boundary, and we split into two cases, depending on whether $\Omega$ is more quantifiably convex or $H$ is radial.

\begin{lemma}\label{L:geoboundaryconvex}
	Assume that $\Omega$ satisfies \eqref{A:uniformlyconvexdomain}. Then there exists a universal $c_0 \in (0,1)$ such that, if $\gamma_\delta = c_0 \delta^q$, then, for all $\eps,\delta \in (0,1)$, $\sigma,\tau,\rho \in (-\gamma_\delta,\gamma_\delta)^m$, and $x,y \in \oline{\Omega}$, \eqref{A:boundarycondition} holds.
\end{lemma}

\begin{proof}
For $x \in \del \Omega$, $y \in \oline{\Omega}$, and $(\sigma,\tau,\rho) \in (-\gamma_\delta,\gamma_\delta)^3$, let $u,v,p$ be the unique maximizers in the definition in \eqref{testfngeo}. Then $D_x \Phi_{\delta,\eps}(x,y,\sigma,\tau,\rho) = p + u$. Recalling that $H$ and $G$ are uniformly Lipschitz, we see that $u = \eps D\Psi(x) + O(\eps\gamma_\delta)$, so, for $\delta \in (0,1)$ and $c_0$ sufficiently small, $u \cdot n(x) > 0$. 
	
	We may therefore assume without loss of generality that $p \ne 0$. Assume for the sake of contradiction that $(p+u) \cdot n(x) < 0$, which, in particular, implies that $p \cdot n(x) < 0$. We first note that
	\[
		x-y = \delta |p|^{-2/3} p + \delta \frac{p}{|p|} + O(\gamma_\delta),
	\]
	so that (upon taking the scalar product with $\frac{p}{|p|}$)
	\[
		\frac{|x-y|}{\delta} \ge 1 - O(\gamma_\delta/\delta).
	\]
	Shrinking $c_0$ if necessary, we have $|x-y| \ge c \delta$. On the other hand, because $p \cdot n(x) < 0$,
	\[
		\theta |x-y|^q \le (x-y) \cdot n(x) \le O(\gamma_\delta).
	\]
	Combining these two facts yields, for some universal $C > 0$, $\delta^q \le C c_0 \delta^q$, which is a contradiction if $c_0$ is sufficiently small. The other inequality is proved similarly.
	\end{proof}

We now drop the assumption \eqref{A:uniformlyconvexdomain}, letting $\Omega$ be an arbitrary $C^1$ convex domain, and assume that $H$ is radial. Because $H$ is one-homogenous, this simply reduces to taking $m = 1$, and, after a renormalization, $H(p) = |p|$. We then consider the test function
\begin{equation}\label{testfn|p|}
\begin{split}
	\Phi_{\delta,\eps}(x,y,\sigma,\tau,\rho) = \sup_{p,u,v} &\Big\{ (p+u) \cdot x - (p-v) \cdot y - \frac{3\delta}{4} |p|^{4/3} - \delta|p|  - \eps \psi^*\left( \frac{u}{\eps} \right)  - \eps \psi^*\left( \frac{v}{\eps} \right) \\
& - (\gamma_\delta - \rho) |p| - (\gamma_\delta - \sigma) |p+u| - (\gamma_\delta + \tau) |p-v|  \Big\},
\end{split}
\end{equation}
where $\psi$ is as in \eqref{psi}. We also note that $\phi_\delta$ takes the form
\begin{equation}\label{geophidelta|p|}
	\phi_{\delta}(z,\rho) = \sup_p \left\{ p \cdot z - \frac{3\delta}{4} |p|^{4/3} -(\delta + 3\gamma_\delta -\rho)|p| \right\} = \frac{1}{4\delta^3} \left[ (|z| - \delta - 3\gamma_\delta + \rho)_+ \right]^4.
\end{equation}

\begin{lemma}\label{L:geotestfn|p|}
	The function \eqref{testfn|p|} satisfies the conclusions of Lemma \ref{L:geounifconvex} with everywhere $\rho$ being allowed to belong to $(-\oo,\gamma_\delta)$. Moreover, there exist $c_0 , \nu_0,\eps_0 \in (0,1)$ sufficiently small that, if $\gamma_\delta = c_0 \delta$, then, for all $\delta \in (0,1)$, $\eps \in (0,\eps_0)$, $\sigma,\tau \in (-\gamma_\delta,\gamma_\delta)$, $\rho \in (-\oo,\gamma_\delta)$, and $x,y \in \oline{\Omega}$ with $|x-y| \le \nu_0$, 
	\[
		D_x \Phi_{\delta,\eps}(x,y,\sigma,\tau,\rho) \cdot n(x) > 0 \quad \text{if } (x,y) \in \partial \Omega \times\oline{\Omega}
	\]
	and
	\[
		D_y \Phi_{\delta,\eps}(x,y,\sigma,\tau,\rho) \cdot n(y) > 0 \quad \text{if } (x,y) \in \oline{\Omega} \times \partial \Omega.
	\]
\end{lemma}

\begin{proof}
	The argument that we may take $\rho \in (-\oo,\gamma_\delta)$ is just as in the proof of Lemma \ref{L:convexH}. 
	
	Assume now that $x \in \del \Omega$, $y \in \oline{\Omega}$ with $|x-y| \le \nu_0$, and $(\sigma,\tau) \in (-\gamma_\delta,\gamma_\delta)^2$ and $\rho \in (-\oo,\gamma_\delta)$. Let $u,v,p$ be the unique maximizers. Then $D_x \Phi_{\delta,\eps}(x,y,\sigma,\tau,\rho) = p + u$. We can write
	\[
		u = \eps \left[ D\psi(x) + O(\gamma_\delta) \right] \quad \text{and} \quad v = \eps \left[ D\psi(y) + O(\gamma_\delta) \right],
	\]
	and therefore, shrinking $\nu_0$, there exists $c > 0$ such that both $u \cdot n(x) \ge c\eps$ and $v \cdot n(x) \ge c\eps$. If $p = 0$ or $p = v$, we are done, and so we assume without loss of generality that $p \ne 0$ and $p \ne v$. Moreover, just as in the proof of Lemma \ref{L:radialHtestfn}, we have $\frac{u}{|u|} \cdot n(x) \ge c$ for a possibly different value of $c > 0$.
	
	We next rule out the equality $p = -u$. Note that there exists $w \in \RR^d$ and $C > 0$ independent of $\delta$ and $\eps$ such that $|w| \le C \gamma_\delta$, and
	\[
		p = \frac{1}{\delta^3} \left| x-y - \delta \frac{p}{|p|} + w \right|^2 \left( x-y - \delta \frac{p}{|p|} + w\right).
	\]
	Therefore, if $p = -u$, we have
	\begin{equation}\label{isp-u?}
		-u = \frac{1}{\delta^3} \left| x-y + \delta \frac{u}{|u|} + w \right|^2 \left( x-y + \delta \frac{u}{|u|} + w \right).
	\end{equation}
	We first estimate
	\[
		c\eps \le |u| = \frac{1}{\delta^3}  \left| x-y + \delta \frac{u}{|u|} + w \right|^3,
	\]
	so that
	\[
		 \left| x-y + \delta \frac{u}{|u|} + w \right| \ge c^{1/3} \delta \eps^{1/3}.
	\]
	On the other hand, taking the scalar product of \eqref{isp-u?} with $n(x)$ gives, for some constants $\tilde C, \tilde c > 0$,
	\[
		-c\eps \ge -u \cdot n(x) \ge \frac{1}{\delta^3} (c^{1/3} \delta \eps^{1/3})^2 (c\delta - w \cdot n(x))
		\ge \tilde c \eps^{2/3} (1 - C c_0).
	\]
	This is a contradiction if $c_0$ and $\eps_0$ are sufficiently small.

	Now using the fact that $p \ne 0$, $p \ne u$, and $p \ne -v$, we write
	\begin{align*}
		p = \frac{1}{\delta^3} &\left| x-y + \delta \frac{u}{|u|} + O(\gamma_\delta) \right|^2 \\
		&\times \left( x-y - (\delta + \gamma_\delta - \rho) \frac{p}{|p|} - (\gamma - \sigma) \frac{p+u}{|p+u|} - (\gamma + \tau) \frac{p-v}{|p-v|} \right).
	\end{align*}
	If $(p+u) \cdot n(x) \le 0$, then, as in the proof of Lemma \ref{L:radialHtestfn}, we have $p \cdot n(x) \le -c \eps$ and $(p+v) \cdot n(x) \le 0$. Taking the scalar product above with $n(x)$ then yields $p \cdot n(x) \ge 0$, which is a contradiction, and we conclude.
\end{proof}

\subsection{Well-posedness results}

We next prove results similar to Proposition \ref{P:keyestimate} for solutions of the geometric equation
\begin{equation}\label{E:geo}
	\begin{dcases}
		du = F(D^2 u, Du)dt + \sum_{i=1}^m H^i(Du) \cdot d\zeta^i & \text{in } \oline{\Omega} \times (0,T],\\
		Du \cdot n = 0 & \text{on } \partial \Omega \times (0,T).
	\end{dcases}
\end{equation}

We first consider the case of a domain with quantified convexity.
\begin{proposition}\label{P:geokeyestimate}
	Assume that $\Omega$ satisfies \eqref{A:C1convexdomain} and \eqref{A:uniformlyconvexdomain}; $H$ satisfies \eqref{A:DCH} and \eqref{A:Hgeometric}; and $F$ satisfies \eqref{A:Fgeometric}. If $\zeta,\eta \in C([0,T],\RR^m)$ and $u \in BUSC(\oline{\Omega} \times [0,T])$ and $v \in BLSC(\oline{\Omega} \times [0,T])$ are respectively a bounded sub- and super-solution of \eqref{E:geo} corresponding to $\zeta$ and $\eta$, then, for all $\delta \in (0,1)$, if
	\[
		\max_{t \in [0,T]} |\zeta_t - \eta_t| \le \gamma_\delta,
	\]
	then
	\begin{align*}
		&\sup_{(x,y,t) \in \oline{\Omega} \times \oline{\Omega} \times [0,T]} \Big\{ u(x,t) - v(y,t) - \phi_\delta(x-y,\zeta_t - \eta_t) \Big\}\\
		& \le \sup_{(x,y) \in \oline{\Omega}} \Big\{ u(x,0) - v(y,0) - \phi_\delta(x-y,\zeta_0 - \eta_0) \Big\}.
	\end{align*}
\end{proposition}

\begin{proof}
	If the result is false, then, for sufficiently small $\mu > 0$,
	\[
		[0,T] \ni t \mapsto \sup_{x,y \in \oline{\Omega}} \Big\{ u(x,t) - v(y,t) - \phi_\delta(x-y,\zeta_t - \eta_t) \Big\} - \mu t
	\]
	attains a maximum at $t_0 \in (0,T]$. 
	
	Arguing exactly as in the proof of Proposition \ref{P:keyestimate}, using Lemmas \ref{L:geounifconvex} and \ref{L:geoboundaryconvex}, we have the following: for $\eps > 0$, there exist $(x_\eps,y_\eps,t_\eps) \in \oline{\Omega} \times \oline{\Omega} \times [0,T]$ and, for $\lambda > 0$, $X_{\lambda,\eps},Y_{\lambda,\eps} \in \mbb S^d$ such that, along a particular subsequence, $\lim_{\eps \to 0} t_\eps = t_0$, $\lim_{\eps \to 0} (x_\eps - y_\eps) = z \in \RR^d$, and $\lim_{\eps \to 0} (X_{\lambda,\eps},Y_{\lambda,\eps}) = (X_\lambda,Y_\lambda)$, where
	\begin{equation}\label{geomatrix}
		- \left( \frac{1}{\lambda} + \norm{\tilde A_\delta(z)} \right)
		\begin{pmatrix}
			\Id & 0 \\
			0 & \Id
		\end{pmatrix}
		\le 
		\begin{pmatrix}
			X_\lambda & 0 \\
			0 & -Y_\lambda
		\end{pmatrix}
		\le 
		\tilde A_\delta(z) + \lambda \tilde A_\delta(z)^2;
	\end{equation}
	and
	\[
		\mu \le F^*(X_{\lambda,\eps}, p_\eps) - F_*(Y_{\lambda,\eps},- q_\eps).
	\]
	Here,
	\[
		p_\eps = D_x \Phi_{\delta,\eps}(x_\eps,y_\eps,\zeta_{t_\eps} - \zeta_{t_0},\eta_{t_\eps} - \eta_{t_0}, \zeta_{t_0} - \eta_{t_0} ),
	\]
	\[
		q_\eps = D_y \Phi_{\delta,\eps}(x_\eps,y_\eps,\zeta_{t_\eps} - \zeta_{t_0},\eta_{t_\eps} - \eta_{t_0}, \zeta_{t_0} - \eta_{t_0} ),
	\]
	and
	\[
		\tilde A_\delta(z) = \frac{3}{\delta} |D\phi_\delta(z, \zeta_{t_0} - \eta_{t_0} ) |^{2/3} 
		\begin{pmatrix}
			\Id & -\Id \\
			-\Id & \Id
		\end{pmatrix}.
	\]
	Sending $\eps \to 0$ along a particular subsequence yields
	\[
		\mu \le F^*(X_\lambda,D\phi_\delta(z, \zeta_{t_0} - \eta_{t_0} )) - F_*(Y_\lambda, D\phi_\delta(z, \zeta_{t_0} - \eta_{t_0} )). 
	\]
	We consider two cases: if $D\phi_\delta(z,\zeta_{t_0} - \eta_{t_0}) \ne 0$, then \eqref{geomatrix} implies that $X_\lambda \le Y_\lambda$ and so $\mu \le 0$, a contradiction. Otherwise, if $D\phi_\delta(z,\zeta_{t_0} - \eta_{t_0}) = 0$, then $\tilde A_\delta(z) = 0$, and we have
	\[
		\mu \le F^*(X_\lambda,0) - F_*(Y_\lambda,0).
	\]
	Then by \eqref{geomatrix}, as $\lambda \to +\oo$, $(X_\lambda,Y_\lambda) \to (0,0)$, which once again yields $\mu \le 0$ in view of \eqref{Fat0}.	
\end{proof}

The same result holds when $H(p) = |p|$ and $\Omega$ is merely $C^1$ and convex, and we are able to improve in terms of the condition on $\zeta$ and $\eta$:
\begin{proposition}\label{P:|p|keyestimate}
	Assume that $\Omega$ satisfies \eqref{A:C1convexdomain}; $H(p) = |p|$; and $F$ satisfies \eqref{A:Fgeometric}. If $\zeta,\eta \in C([0,T],\RR^m)$ and $u \in BUSC(\oline{\Omega} \times [0,T])$ and $v \in BLSC(\oline{\Omega} \times [0,T])$ are respectively a bounded sub- and super-solution of \eqref{E:geo} corresponding to $\zeta$ and $\eta$, then, for some universal $c_0 \in (0,1)$, if 
	\[
		0 < \delta < c_0\left(1 + \norm{u}_{\oo,\oline{\Omega} \times [0,T]} + \norm{v}_{\oo,\oline{\Omega} \times [0,T]} \right)^{-3},
	\]
	$\gamma_\delta = c_0 \delta$, and
	\[
		\eta_t - \frac{\nu_0}{4} \le \zeta_t \le \eta_t + \gamma_\delta \quad \text{for all } t \in [0,T],
	\]
	where $\nu_0$ is as in Lemma \ref{L:geotestfn|p|}, then
	\begin{align*}
		&\sup_{(x,y,t) \in \oline{\Omega} \times \oline{\Omega} \times [0,T]} \Big\{ u(x,t) - v(y,t) - \phi_\delta(x-y,\zeta_t - \eta_t) \Big\} \\
		&\le \sup_{(x,y) \in \oline{\Omega}} \Big\{ u(x,0) - v(y,0) - \phi_\delta(x-y,\zeta_0 - \eta_0) \Big\}.
	\end{align*}
\end{proposition}

\begin{proof}
	The proof is almost identical to the proof of Proposition \ref{P:geokeyestimate}, this time applying Lemma \ref{L:geotestfn|p|} instead of Lemma \ref{L:geoboundaryconvex} (in particular, note that the lower bound on $\zeta - \eta$ need not shrink with $\delta$). 
	
	The only difference is that the points $x_\eps,y_\eps \in \oline{\Omega}$ must satisfy $|x_\eps - y_\eps| \le \nu_0$ in order for the conclusions of Lemma \ref{L:geotestfn|p|} to hold. Therefore, in order to carry out the arguments as in the proof of Proposition \ref{P:geokeyestimate}, it is sufficient to show that, if $|x-y| > \frac{\nu_0}{2}$, then, for all $t \in [0,T]$,
	\[
		u(x,t) - v(y,t) - \phi_\delta(x-y,\zeta_t -\eta_t) < \sup_{x,y \in \oline{\Omega}} \left\{ u(x,t) - v(y,t) - \phi_\delta(x-y,\zeta_t -\eta_t) \right\}.
	\]
	Indeed, by \eqref{geophidelta|p|}, if $x,y \in \oline{\Omega}$ and $|x-y| > \frac{\nu_0}{2}$, then, if $c_0$ is sufficiently small, then, for some universal $c_1 > 0$,
	\begin{align*}
		u(x,t) & - v(y,t) - \phi_\delta(x-y,\zeta_t -\eta_t)\\
		 &\le u(x,t) - v(y,t) - \frac{1}{4\delta^3} \left( |x-y| - \delta - 3\gamma_\delta + \zeta_t - \eta_t \right)_+^4\\
		&\le u(x,t) - v(y,t) - \frac{1}{4\delta^3} \left( \frac{\nu_0}{4} - 4\delta \right)_+^4\\
		&\le u(x,t) - v(y,t) - \frac{c_1}{\delta^3},
	\end{align*}
	while
	\[
		\sup_{x,y \in \oline{\Omega}} \left\{ u(x,t) - v(y,t) - \phi_\delta(x-y,\zeta_t -\eta_t) \right\}
		\ge u(x,t) - v(x,t).
	\]
	The inequality thus holds if $c_0$ is sufficiently small.
\end{proof}

Just as in Section \ref{sec:nonsmoothwellposed}, Propositions \ref{P:geokeyestimate} and \ref{P:|p|keyestimate} lead to the well-posedness of the level-set equation \eqref{E:geo}.

\begin{theorem}\label{T:geowellposed}
	Assume that $\Omega$ satisfies \eqref{A:C1convexdomain}, $H$ satisfies \eqref{A:DCH} and \eqref{A:Hgeometric}, $F$ satisfies \eqref{A:Fgeometric}, and one of the following: $\Omega$ satisfies \eqref{A:uniformlyconvexdomain}, or $H$ is radial.
	
	\begin{enumerate}[(i)]
	\item\label{geocomparison} If $\zeta \in C([0,T],\RR^m)$ and $u \in BUSC(\oline{\Omega} \times [0,T])$ and $v \in BLSC(\oline{\Omega} \times [0,T])$ are respectively a bounded sub- and super-solution of \eqref{E:geo}, then
	\[
		\sup_{(x,t) \in [0,T]} \Big\{ u(x,t) - v(x,t) \Big\} \le \sup_{x \in \oline{\Omega}} \Big\{ u(x,0) - v(x,0) \Big\}.
	\]
	\item\label{geosolution} If $u_0 \in BUC(\oline{\Omega})$ and $\zeta \in C([0,T],\RR^m)$, then there exists a unique solution $u \in BUC(\oline{\Omega} \times [0,T])$ of \eqref{E:geo} with $u(\cdot,0) = u_0$. Moreover, $\norm{u}_{\oo,\oline{\Omega} \times [0,T]} \le \norm{u_0}_{\oo,\oline{\Omega}}$, and the modulus of continuity of $u(\cdot,t)$ is independent of $t > 0$ and $\zeta$.
	\item\label{geostable} For any $u_0 \in BUC(\oline{\Omega})$, there exists a modulus of continuity $\omega: [0,\oo) \to [0,\oo)$ such that, for any $\zeta^1,\zeta^2 \in C_0([0,T],\RR^m)$, if $u^1,u^2 \in BUC(\oline{\Omega} \times [0,T])$ are the corresponding solutions of \eqref{E:geo}, then
	\[
		\norm{u^1 - u^2}_{\oo,\oline{\Omega} \times [0,T]} \le \omega\left( \norm{\zeta^1 - \zeta^2}_{\oo,[0,T]} \right).
	\]
	\end{enumerate}
\end{theorem}

%

We finish with a monotonicity property for solutions of \eqref{E:geo}, when $H(p) = |p|$.

\begin{theorem}\label{T:geomonotonicity}
	Assume $H(p) = |p|$, fix $u_0 \in BUC(\oline{\Omega})$ and $\zeta, \eta \in C([0,T],\RR)$, and assume that $\zeta_\cdot - \zeta_0 \le \eta_\cdot - \eta_0 $. Then, if $u$ and $v$ are the corresponding solutions of \eqref{E:geo}, we have $u \le v$ in $\oline{\Omega} \times [0,T]$.
\end{theorem}

\begin{proof}
	Upon adding a constant to $\zeta$ or $\eta$, we may assume without loss of generality that $\zeta_0 = \eta_0$. It also suffices to prove the result when $\zeta$ and $\eta$ satisfy $- \frac{\nu_0}{4} \le \zeta - \eta \le 0$, where $\nu_0$ is as in Proposition \ref{P:|p|keyestimate}, since the result can be iterated for $\zeta$ and $\eta$ an arbitrarily large distance away.
	
	In that case, Proposition \ref{P:|p|keyestimate} gives, for all sufficiently small $\delta$,
	
	
	\[
		u(x,t) - v(x,t) \le \sup_{x,y \in \oline{\Omega}} \left( u_0(x) - u_0(y) - \phi_\delta(x-y,0 ) \right).
	\]
	Sending $\delta \to 0$ gives $u \le v$.
\end{proof}

As a limiting case of the above result, we can obtain an inequality relating solutions of \eqref{E:geo} with solutions obtained from solving the equations involving $F$ and $H$ separately. This will be useful in Section \ref{sec:mcf} below.

Let $S_F(t)$, $S_{\pm H}(t)$, $t \ge 0$, denote the viscosity-solution semigroups associated to the Neumann problems for the nonlinearities $F$ and $\pm H$ on $\oline{\Omega}$.

\begin{corollary} \label{cor:geomonotonicity}
Assume $\Omega$ satisfies \eqref{A:C1convexdomain}, $F$ satisfies \eqref{A:Fgeometric}, and $H(p) = |p|$. Fix $u_0 \in BUC(\oline{\Omega})$, $\xi \in C([0,T],\RR)$, and let $u \in BUC(\oline{\Omega} \times [0,T])$ be the solution of \eqref{E:geo}. Then, for all $t \geq 0$,
\begin{align*}
	&S_{H}\left(\xi(t)- \min_{[0,t]} \xi\right)\circ S_F(t) \circ S_{-H}\left(-\min_{[0,t]} \xi\right) u_0 \\
	&\leq u(t,\cdot) \leq  S_{-H}\left(\max_{[0,t]} \xi - \xi(t)\right) \circ S_F(t)  \circ S_{H}\left(\max_{[0,t]} \xi\right) u_0.
\end{align*}
\end{corollary}

\begin{proof}
The first inequality follows from applying (formally) Theorem \ref{T:geomonotonicity} with $\zeta$ given by
\begin{equation*}
\zeta(s) = \begin{cases}
\xi(0), & s=0 \\ \min_{[0,t]} \xi, &s \in (0,t), \\ \xi(t), & s=t.
\end{cases}
\end{equation*} 
Since $\zeta$ is not continuous, one cannot apply this result directly, but taking smooth approximations $(\zeta^n)_{n \in \NN}$ such that $\zeta \le \zeta^n \le \xi$ and, as $n \to \oo$, $\zeta^n$ decreases pointwise to $\zeta$. If, for $n \in \NN$, $v^n$ is the solution of \eqref{E:geo} corresponding to $\zeta^n$, then $v^n \le u$ by Theorem \ref{T:geomonotonicity}. The convergence as $n \to \infty$ of $v^n(t, \cdot)$ to the left-hand-side of the above inequality can then be obtained by the stability of viscosity solutions and a simple time-reparametrization argument (as in \cite[Proposition 4.9]{GG19}).
\end{proof}

\section{Long time behaviour of perturbed mean curvature flow} \label{sec:mcf}

Throughout this section we fix $\Omega = D \times \RR$, where $D$ is a bounded, $C^1$, convex domain in $\RR^{d-1}$. We will write elements of $\Omega$ as $x = (x',x_d)$, $x' \in D$ and $x_d \in \RR$. 

We consider the level set equation for perturbed mean curvature flow in $\Omega$ with right angle boundary equation, namely, the initial value problem
\begin{equation}\label{E:mcf}
\begin{dcases}
	d u =   \tr\left[ D^2 u \left( \Id - \frac{Du \otimes Du}{\left|Du\right|^2}\right) \right] dt + | Du |\cdot d\xi(t) \mbox{ on } \Omega \times [0,+\oo),\\
	u(x,0) = u_0(x) \mbox{ on } \Omega,\\
	\partial_n u = 0 \mbox{ on } \times \partial \Omega \times [0,+\oo),
\end{dcases}
\end{equation}
where $u_0$ is a continuous function on $\Omega$, such that 
\begin{equation*}
\alpha \leq u_0 \leq \beta, \;\;\;\;u_0(x',x_d) = \alpha \mbox{ for } x_d \leq a, \;\;\;u_0(x',x_d) = \beta \mbox{ for } x_d \geq b
\end{equation*}
where $\alpha< \beta$ and $a<b$ are fixed. 

By Theorem \ref{T:geowellposed}, there exists a unique solution of \eqref{E:mcf} which satisfies, for $t \ge 0$,
\begin{equation} \label{eq:alphabeta}
u(\cdot,t) = \alpha \mbox{ on } D \times (-\infty, a - \xi(t)], \;\;\;u(\cdot,t) = \beta \mbox{ on } D \times [b-\xi(t),+\infty).
\end{equation}


Our main result in this section is then the following.

\begin{theorem}
Let $\xi=B(\omega) \in C([0,\infty),\RR)$ be a Brownian motion sample path. Then, almost surely, there exists a non-decreasing continuous function $\hat{v}: \RR \to \RR$, such that, letting
\[v(x,t) = \hat{v}(x_d+ B(t)), \]
it holds that
\begin{equation} \label{eq:ConvMCFfunction}
\lim_{t \to \infty} \left\| u(\cdot,t) - v(\cdot,t) \right\|_{\infty} = 0,
\end{equation}

In addition, for each $c \in (\alpha,\beta)$, it holds that
\begin{equation} \label{eq:ConvMCFlevelset}
\lim_{t \to \infty} \dist_{H}\left(\Gamma^c_{u(\cdot,t)}, \Gamma^c_{v(\cdot,t)}\right)  = 0,
\end{equation}
where $\Gamma^c_w := \left\{ x \in \Omega, \;\; w(x) = c\right\}$ are the level sets and $\dist_{H}$ is the Hausdorff distance.
\end{theorem}

In the deterministic case ($B \equiv 0$), a similar result was obtained by \cite{GOS99}, and in fact our proof heavily relies on their result. Due to the noise term, our conclusions are slightly stronger:  we obtain that the limit is monotone in $x_d$, so that in particular almost every level set will converge to a hyperplane perpendicular to the boundary of $D$. In addition, we obtain convergence of all level sets in Hausdorff distance, whereas in \cite{GOS99}, this was only proven for level sets such that the limit is a (finite union of) hyperplane(s). 

The idea of proof is as follows:

\begin{itemize}
\item First, on arbitrary long intervals, the noise will be small. Hence we can directly appeal to the deterministic result (and the continuity of the solution map) to obtain that $u$ is close to a stationary solution after these intervals (and remains so at later times), and this stationary solution is of the form $v=v(x_d)$.
\item Second, due to the $dB$ term, any stationary solution will be non-decreasing for large times. This is due to the fact that large negative excursions of $B$ will ``fill up'' any hole in the sub-level sets.

\item Finally, we need to prove the convergence of level sets in Hausdorff distance. One direction is clear from convergence of $u$, and for the other direction we need to check that the ``fat'' level sets cannot have holes (of a non-negligible size) for large times. { Such holes are unstable in the sense that if the driving noise has a large increment over a small interval, the hole becomes large and invades the level set (see Lemma \ref{lem:hole} below for a precise statement). Using the strong Markov property and the full support of Brownian paths, by a Borel-Cantelli argument such events will happen almost surely for arbitrary large times, so that holes cannot survive in the large time limit.} In this part of the argument we need to use the monotonicity result (Corollary \ref{cor:geomonotonicity}) which allows us to compare the level sets of solutions to \eqref{E:mcf} with level sets following successively the deterministic and stochastic parts of the equation.   \end{itemize}

{ It is natural to ask to which extent the Brownian structure is actually needed, and if the results would still hold, for example, with $B$ a sample path of fractional Brownian motion (fBm). The convergence of the function $u$ \eqref{eq:ConvMCFfunction}, only requires the first two steps above (more precisely, this requires the finiteness of the sequences $t_n$, $t'_n$ defined in \eqref{eq:deftn}-\eqref{eq:deft2n}), and will hold for example for (almost) any fBm path with arbitrary regularity. The proof of the convergence in Hausdorff distance \eqref{eq:ConvMCFlevelset}, however, is more probabilistic and requires the strong Markov property, so that it is not clear if it could be extended to fBm paths.
}

We now pass to the details of the proof.

\begin{lemma} Let 
\[ \mathcal{K}(u_0) = \left\{ u(\cdot,\cdot - \xi(t),t), \; u \mbox{ is the solution of \eqref{E:mcf} for some } \xi \in C([0,t]) \right\}.\]
Then $\mathcal{K}(u_0)$ is compact in $C(\Omega)$.
\end{lemma}

\begin{proof}
Note that the elements of $\mathcal{K}(u_0)$ are constant outside of $D \times [a,b]$ and take values in $[\alpha, \beta]$. Compactness is then a consequence of Theorem \ref{T:geowellposed}, since
elements of $\mathcal{K}(u_0)$ share the same modulus of continuity.
\end{proof}

\begin{lemma} \label{lem:lt1}
For each $\varepsilon>0$, there exists $T>0$ sufficiently large and $\eta>0$ sufficiently small such that for each $u \in \mathcal{K}(u_0)$, there exists $v(x)=\hat{v}(x_d)$ with $ \| S^{T,\xi}(u) - v\| \leq \varepsilon$ for each $\|\xi\|\leq \eta$.
\end{lemma}

\begin{proof}
For a fixed $u$, and $\xi=0$, this is the main result in \cite{GOS99}. The existence of $T$ and $\eta$ which are uniform over elements of $\mathcal{K}(u_0)$ follows from the compactness of that set and continuity of $S_F$.
\end{proof}

We also have the following result.

\begin{lemma} \label{lem:lt2}
Let $v_0(x) = \hat{v}_0(x_d)$, where $\hat{v}_0$ is continuous with values in $[\alpha,\beta]$, $\hat{v}_0 = \alpha$ on $(-\infty, a]$ and $= \beta$ on $[b,+\infty)$.

Let $v$ be the solution to \eqref{E:mcf} starting from $v_0$. Then $v(x,t) = \hat{v}(x_d,t)$ where $\hat{v}=\hat{v}(r,t)$ solves
$$d\hat{v} = \left| \partial_{r} \hat{v} \right| \cdot d{\xi}(t) \mbox{ on }\RR \times (0,\infty), \;\; \hat{v}(\cdot,0)=\hat{v}_0.$$
In addition, let $\hat{t} = \inf\left\{ t , \;\; \max_{[0,t]} \xi - \min_{[0,t]} \xi  = \frac{b-a}{2}\right\}$. Then
\begin{equation*}
\hat{v}(\cdot,\hat{t}) \mbox{ is non-decreasing}
\end{equation*} 
and 
\begin{equation} \label{eq:afterthat}
\forall t \geq \hat{t}, \;\;\;\;\hat{v}(r,t) = \hat{v}( r + \xi(t)-\xi(\hat{t}),\hat t).
\end{equation}

\end{lemma}

\begin{proof}
First, one checks that $\hat{v}(x_d,t)$ satisfies \eqref{E:mcf} (this is classical if $\xi$ is smooth, and the general case follows by stability), and it follows that it is equal to $v$.

Let us now prove monotonicity of $\hat{v}(\cdot,\hat{t})$. Assume first that $\xi$ attains its maximum at time $\hat{t}$, and let $t_{\ast}$ in $[0,\hat{t}]$ such that $\xi(t_{\ast}) = \min_{[0,\hat{t}]} \xi$. It follows from Corollary 2.8 in \cite{GGLS} and the representation of the deterministic semigroup for $H(p)=|p|$ that 
\[ \hat{v}(r,\hat{t})  = \max_{ |s-r| \leq \frac{b-a}{2} } \hat{v}(s,t_{\ast}) .\]
Using the fact that $\hat{v}(\cdot,t_{\ast})$ is equal to $\alpha$ or $\beta$  outside  of an interval of length $(b-a)$, one checks that the right-hand side above is nondecreasing in $r$. The case where $\xi$ attains its minimum at time $\hat{t}$ is similar, replacing the maximum by a minimum in the formula above.

Finally, \eqref{eq:afterthat} follows by noting that the right-hand side is a nondecreasing solution to $d\hat{v} = (\partial_r \hat{v}) \cdot d{\xi}(t) = |\partial_r \hat{v}| \cdot d{\xi}(t) $. \end{proof}

We can now proceed with the proof of \eqref{eq:ConvMCFfunction}. Let $T_n$, $\eta_n$ be obtained from applying Lemma \ref{lem:lt1} to a fixed sequence $\varepsilon_n \to 0$. We then let 
\begin{equation} t_n = \inf \left\{ t \geq T_n : \osc_{[t-T_n,t]} B \leq \eta_n \right\}, \label{eq:deftn} \end{equation}
and
\begin{equation} t'_n = \inf \left\{ t \geq t_n : \osc_{[t_n,t]} B \geq (b-a)/2 \right\}. \label{eq:deft2n} \end{equation}
and note that $t_n, t'_n$ are finite a.s. by an easy consequence of the Borel-Cantelli lemma. Now Lemma \ref{lem:lt1} guarantees that at time $t_n$, $u$ is $\varepsilon_n$-close of a function of $x_d$, and by Lemma \ref{lem:lt2} at time $t'_n$ it is $\varepsilon_n$-close to a function $\hat{v}_n$ which is a nondecreasing function of $x_d$.

%
Now we turn to the proof of \eqref{eq:ConvMCFlevelset}. 

We let $S_{\pm H}$, $S_F$ be the viscosity semigroups (with homogeneous Neumann boundary conditions) associated to $\pm H(p)= \pm |p|$ and $F(p,A) =\tr\left[ A \left( \Id - \frac{p \otimes p}{\left|p\right|^2}\right) \right]$. In view of Corollary \ref{cor:geomonotonicity}, we have the following: if $u$ is the solution to \eqref{E:mcf}, then

\begin{equation} \label{eq:mono}
u(\cdot,t) \geq S_{H}\left(\xi(t)- \min_{[0,t]} \xi\right) S_F(t) S_{-H}\left(-\min_{[0,t]} \xi\right) u _0
\end{equation}

Further note that $S_{\pm H}$ admit the following simple expression
\begin{equation} \label{eq:SH}S_{H}(t) 
u(x) = \sup_{|y-x| \leq t, y \in \Omega} u(y), \;\;\;\;\;S_{-H}(t) u(x) = \inf_{|y-x| \leq t, y \in \Omega} u(y).
\end{equation}
This follows from the representation of $S_{\pm H}$ in terms of control problems with reflecting trajectories (e.g. \cite{Lions85}), and the convexity of $\Omega$ ensures that in fact reflection is never optimal.

%
%
\begin{lemma} \label{lem:hole}
For all $r >0$, there exists $h , \varepsilon >0$, such that if for some $y \in \Omega$, $c \in (\alpha,\beta)$,
\[ u(\cdot,t) > c \mbox{ on } B(y,r) \cap \Omega , \;\;\;  \min_{[t,t+h]} B_{t,\cdot} \geq - \varepsilon , \;\;\;\;B_{t,t+h}\geq\varepsilon^{-1},\]
then
\[ u(\cdot,t+h) > c \mbox{ on } \left\{ (x',x_d) \in \Omega , \;\;\; x_d - B_{t,t+h} \geq y_d - \frac{r}{2} \right\} .\]
\end{lemma}

\begin{proof}
The estimate \eqref{eq:mono} together with \eqref{eq:SH} implies that
$$ u(\cdot,t) > c \mbox{ on } B(y,r) \Rightarrow  u(\cdot,t+h) > c \mbox{ on } B(y,R)$$
with
$$ R = B_{t,t+h}-\min_{[t,t+h]} B_{t,\cdot} + \phi(r + \min_{[t,t+h]} B_{t,\cdot}, h) $$
where $\phi$ is the solution to $\partial_1 \phi = - \frac{d}{\phi}$ with $\phi(0,r)=r$ (which is how the radius of a sphere evolves with mean curvature flow), under the condition that $\phi(r +\min_{[t,t+h]} B_{t,\cdot}, h) \geq 0$.
Then it is clear that for $h$ and $\min_{[t,t+h]} B_{t,\cdot}$ small enough, $R \geq \frac{r}{2} + B_{t,t+h} $, and then that, since $D$ is bounded, for  $R$ large enough, 
\begin{equation*}
 \left\{x_d \geq y_d + R - \frac{r}{2} \right\} \subset B(y,R)  \;\; \cup  \left( D \times [\beta - B(t+h),+\infty) \right), 
\end{equation*}
and the claim follows.
\end{proof}

We now resume the proof of \eqref{eq:ConvMCFlevelset}. Recall that 
\begin{equation} \label{eq:distH}
\dist_{H}\left(\Gamma^c_{u(\cdot,t)}, \Gamma^c_{v(\cdot,t)}\right) = \sup_{x \in \Gamma^c_{u(\cdot,t)} } \dist\left(x, \Gamma^c_{v(\cdot,t)}\right) + \sup_{ x \in \Gamma^c_{v(\cdot,t)}} \dist\left(x, \Gamma^c_{u(\cdot,t)}\right)
\end{equation} 
where $\dist(x,A) = \inf_{y \in A} |x-y|$. The convergence of the first term to $0$ for arbitrary $c$ is a simple consequence of the uniform convergence of $u$ to $v$ and the fact that $v$ is a continuous, monotonous function of $x_d$. In order to show that the second term also converges to $0$, we set
$$m^+(t,c) = \sup \left\{ x_d \in \RR : \exists x' \in D, \; u(x',x_d,t)=c \right\},$$
and $$m^-(t,c) = \inf\left\{ x_d \in \RR : \exists x' \in D, \; u(x',x_d,t)=c \right\}.$$

Note that for each $c$, $m^+(t,c) - m^{-}(t,c)$ is nonincreasing in $t$ and converges as $t \to \infty$. In addition, one has 
\[\Gamma^{c}_{v(\cdot,t)} \subset D \times [m^{-}(t,c) , m^{+}(t,c)]. \]
Indeed, assume for instance that $c = v(x,t) = \hat{v}(x_d + B(t))$ but $m^+(t,x) < x_d$. This means that $\inf_{ y \in D} u(y,x_d,t) = c + \delta$ for some $\delta >0$, and by comparison with stationary solutions this yields
\[\forall s \geq t, \quad u(x - B_{t,s},s) \geq c + \delta , \;\;\;\; v(x - B_{t,s},s) = c, \]
which contradicts uniform convergence of $u$ to $v$.

In order to prove that the second term of \eqref{eq:distH} converges to $0$, it will therefore be sufficient to prove that
\begin{equation} \label{eq:78}
	\begin{split}
	&\forall r>0, \exists T(r) ,\; \forall t \geq T(r), \forall y=(y',y_d) \\
	&\mbox{ with } m^{-}(t,c) \leq y_d \leq m^{+}(t,c), \;\; \; B(y,r) \cap \Gamma^c_{u(\cdot,t)} \neq \emptyset.
	\end{split}
\end{equation}

In order to prove \eqref{eq:78} we proceed by contradiction, and assume that there exists a sequence of stopping times $\tau_k \to \infty$ and $c_k \in [\alpha,\beta]$ s.t. 
$$ \exists y^{(k)}=(y'^{(k)},y_d^{(k)}) \mbox{ with } m^{-}(\tau_k,c_k) \leq y_d^{(k)} \leq m^{+}(\tau_k,c_k), \;\; \; B(y^{(k)},r) \cap \Gamma^{c_k}_{u(\cdot,\tau_k)} = \emptyset.$$
We also assume (w.l.o.g. by symmetry) that $u(\cdot,\tau_k) > c_k$ on $B(y^{(k)},r)$. 
We then fix $h,\varepsilon>0$ given by Lemma \ref{lem:hole}, and let
$$A_k = \left\{\min_{[\tau_k,\tau_k+h]} B_{\tau_k,\cdot} \geq - \varepsilon , \;\;\;\; B_{\tau_k,\tau_k+h}\geq \varepsilon^{-1} \right\}.$$
By the strong Markov property and the Borel-Cantelli lemma applied to the events $A_k$, almost surely, there exists a subsequence $k'$ s.t. 
\[ u(\cdot,\tau_{k'}+h) > c_{k'} \mbox{ on } \left\{ (x',x_d), x_d - B_{\tau_{k'},\tau_{k'}+h} \geq y_d - \frac{r}{2} \right\} \]
which further implies that
\begin{equation} \label{eq:mtau}
(m^+ - m^-)(\tau_{k'}+h, c_{k'}) \leq (m^+ - m^-)(\tau_{k'}, c_{k'}) - \frac{r}{2}.
\end{equation}

Further note that if $r> m^+(t,c) - m^{-}(t,c)$, then $ \forall y=(y',y_d) \mbox{ with } m^{-}(t,c) \leq y_d \leq m^{+}(t,c), \;\; \; B(y,r) \cap \Gamma^c_{u(t,\cdot)} \neq \emptyset$ 
(this follows from a simple continuity argument since in that case $u(y',y_d+r,t) > c > u(y',y_d-r,t)$). In addition, the fact that $u$ converges uniformly to a monotonous function of $x_d$ implies that 
\[ \left\{ c :  \lim_{t \to \infty} m^+(t,c) - m^{-}(t,c) >r \right\} \mbox{ is finite,} \]
so that taking a further subsequence if necessary we also have $c_{k'} \equiv c$. But then \eqref{eq:mtau} is in contradiction with the fact that $(m^+ - m^-)(t,c)$ converges as $t\to \infty$. This concludes the proof of \eqref{eq:78}.

\section{Possible extensions and open questions} \label{sec:cl}

{
We discuss in this section possible relaxations of the assumptions we made throughout the paper, and outline the 
difficulties that would arise.}

\subsection{On equations with $x$-dependent Hamiltonians}
We focus in this paper on spatially homogeneous Hamiltonians $H$. The case of $x$-dependent $H$ is much more demanding technically, even without boundary conditions, and requires some care on the assumptions, see for instance \cite{FGLS,Shomog,Snotes,GGLS}. Unlike in the $x$-independent case, where we can deal with any continuous signal $\zeta$, the regularity of the signal plays a role in the $x$-dependent case (this is not surprising, since this is already the case for ODEs, as made very explicit in Lyons' rough path theory \cite{Ly98}). We therefore expect that extending the results of this paper to $x$-dependent Hamiltonians would be a highly nontrivial task.

In relation to this question, an important remark concerns the case of linear transport equations, for which specific difficulties arise in the Neumann case. Indeed, recall that the characteristics corresponding to these boundary conditions are solutions to reflected ODEs. It was shown recently by the first author \cite{Gas21} that such equations are ill-posed (i.e. may have multiple solutions) in the case of rough driving signals (less regular than Brownian motion). At the PDE level, the example from that paper shows that solutions $u=u(x,y,t)$ to
\begin{align*}
du = - u_x dt + (x u_y + y u_x) d\zeta \mbox{ on }  \RR_+ \times \RR \times (0,T), \\
u(\cdot,\cdot,0)= u_0, \;\;\;\; u_x = 0 \mbox{ on }  \{0\} \times \RR \times (0,T)
\end{align*}
can develop discontinuities immediately (if $\zeta$ is rough enough), even if $u_0$ is smooth. While the case of transport equations is rather special (the Hamiltonian $H$ is far from coercive), this example may nevertheless hint at specific difficulties that could arise when trying to extend the results in this paper to $x$-dependent equations, {which could not be limited to the technicalities already present in the full space or periodic case. (Note that in the deterministic viscosity theory, such discontinuities only happen in the case of Dirichlet boundary conditions, for which comparison theorems are known to be much more delicate to establish than for the Neumann case).}

On the positive side, we nevertheless remark that it is possible to obtain the continuity of the solution map $\zeta \mapsto u$ in some particular nonlinear $x$-dependent examples. In the case of one-dimensional $\zeta$, and convex $H$, it was shown in \cite{GGLS} (in the full space case) using the control problem associated to the PDE, that the map $\zeta \mapsto u$ is uniformly continuous (see also \cite[Appendix A]{Shomog2}). This gives in particular a simple way of obtaining the continuity of the map $\zeta\mapsto u$ (under more general assumptions than what is usually needed for the PDE theory). One could obtain similar results in the case of Neumann boundary conditions (the control problem now corresponding to so-called reflected ODEs), for some rather special choices of Hamiltonians (namely, those for which the optimal trajectories are in fact never reflected). We refrain from giving more details here, but it would apply for instance to Hamiltonians of the form $H(x,Du) = a(x) |Du |$, which arise in front propagation.

{
\subsection{On more general boundary conditions}

In the deterministic case, it is known that comparison holds for much more general Neumann boundary conditions than the homogeneous case considered here. These can take the form
\begin{equation} \label{eq:genBC} G(x, Du) = 0 \;\;\; on \;\;\; \partial\Omega, \end{equation}
where $G$ only needs to satisfy a condition of the type $D_p G \cdot n \geq c > 0$ on $\partial \Omega$. This includes for instance the case of so-called capillary boundary conditions, where
\begin{equation} \label{eq:capBC}
 G(x, p) = p \cdot n(x) - \theta \left| p \right|,
 \end{equation}
with $\theta \in (0,1)$, which are relevant when considering the motion of hypersurfaces which intersect the boundary of the domain at a prescribed angle $\arccos(\theta)$. (We refer to e.g. \cite{Bar93,Bar99, IS04} for precise statements).

In this paper, we only consider the case of homogeneous linear Neumann conditions. The test functions that we construct are indeed restricted to this case, and their definition would need significant modifications to treat \eqref{eq:genBC} (or even linear but not homogeneous Neumann conditions). Again, we note that this is a delicate issue due to the $x$-dependence implicit in the fact that we work in a domain.

An exception would be the case of the half-space (already considered in subsection \ref{subsec:halfspace} above). In that case, since $n$ is constant on the boundary, it should be rather straightforward to combine our construction with classical arguments (such as \cite{Bar93}) to obtain well-posedness under boundary conditions of the form \eqref{eq:capBC} (note that $x$-dependence then disappears).

\subsection{More general domains}

In this text, we only consider convex domains, and in the case where the Hamiltonian is non-smooth (and not radial), we actually restrict ourselves to either the half-space or strictly convex domains. It seems likely that the main results should still hold for arbitrary $C^1$ convex domains, possibly with a suitable modification of the test-functions, but we have not been able to find a proof of this fact.

Another important question is whether convexity of the domain is actually necessary. As explained in the introduction, our proof completely breaks down if this assumption is dropped. It is not clear whether this is purely a technical point, or if in fact there are non-convex domains in which comparison does not hold for the kind of problems we consider here.

}
\begin{appendix}
\section*{Properties of pathwise viscosity solutions}

In this appendix, we isolate and prove some of the general properties and results concerning pathwise viscosity solutions used throughout the paper. The proofs of the statements here contain idiosyncrasies particular to the setting of the Neumann problem; however, ignoring these particular points, they are also valid in the full-space setting, and so this appendix acts additionally as a source for detailed explanations of some of the unique aspects of the pathwise viscosity solution theory, whose statements are found in many other works on the subject.

It turns out that, if $H$ is smooth and satisfies \eqref{A:HC2} (and, therefore, also \eqref{A:DCH}), then Definitions \ref{D:smoothH} and \ref{D:nonsmoothH} are equivalent. We omit the proof of this fact, since it is similar to the proof of Proposition \ref{P:nonsmoothdefconsistent} below. Therefore, throughout this secton, we work only with Definition \ref{D:nonsmoothH}.

We often use the following lemma, which is used throughout the literature on viscosity solutiosn of second-order equations in order to linearize certain quadratic expressions.

\begin{lemma}\label{L:Youngtrick}
	Let $n \in \NN$, $\mcl A \in \mbb S^n$, $X,\Xi \in \RR^n$, and $\delta > 0$. Then
	\[
		\frac{1}{2} \mcl A X \cdot X  \le \frac{1}{2} (\mcl A + \delta \mcl A^2)\Xi \cdot \Xi + \frac{1}{2} \left( \norm{\mcl A} + \frac{1}{\delta} \right)|X - \Xi|^2.
	\]
\end{lemma}

{
\begin{proof}
	A Taylor expansion around $\Xi$ yields
	\[
		\frac{1}{2} \mcl A X \cdot X \le \frac{1}{2} \mcl A \Xi \cdot \Xi + \mcl A \Xi \cdot (X - \Xi) + \frac{1}{2} \norm{\mcl A} |X - \Xi|^2.
	\]
	By Young's inequality,
	\[
		\mcl A \Xi \cdot (X - \Xi) \le \frac{\delta}{2} \norm{\mcl A \Xi}^2 + \frac{1}{2\delta} |X - \Xi|^2,
	\]
	and the result follows from the fact that $\norm{\mcl A \Xi}^2 = \mcl A^2\Xi \cdot \Xi$.
\end{proof}
}

{
The next proposition that the definition of pathwise viscosity solutions (with Neumann boundary conditions) is consistent with the classical one when the driving path is smooth. 
}
\begin{proposition}\label{P:nonsmoothdefconsistent}
	Assume $H$ satisfies \eqref{A:DCH}, $\zeta \in C^1([0,T],\RR^m)$, and $u \in (USC,LSC)$ $UC(\RR^d \times [0,T])$. Then $u$ is a (sub-, super-) solution of \eqref{E:neumann} in the sense of Definition \ref{D:nonsmoothH} if and only if $u$ is a classical viscosity (sub-, super-) solution of \eqref{E:neumann}.
\end{proposition}

\begin{proof}
	We prove the equivalence for sub-solutions only; the other statements have similar proofs.
	
	Assume first that $u$ is a classical viscosity sub-solution, let $\Phi \in UC(\RR^d \times [0,T])$ be a pathwise viscosity solution of \eqref{E:pathwiseHJnonsmooth} (because $\zeta \in C^1([0,T],\RR^m)$, this is equivalent to being a classical viscosity solution; see \cite{Snotes}), and define $\oline{u}$ as in Definition \ref{D:nonsmoothH}. Let $\phi: \RR^d \times [0,T] \to \RR$ be $C^2$ in space and $C^1$ in time, and assume that $\oline{u}(\xi,t) - \phi(\xi,t)$ attains a local maximum at $(\xi_0,t_0) \in \RR^d \times [0,T]$. Using standard arguments from the theory of viscosity solutions, we may assume, without loss of generality, that the maximum is strict and $t_0 < T$.
	
	Fix $x_0 \in A^+_{\xi_0,t_0}$ and assume either that $x_0 \notin \partial \Omega$, or that $x_0 \in \partial \Omega$ and
	\begin{equation}\label{strictboundarycase}
	D\phi(\xi_0,t_0) \cdot n(x_0) > 0.
	\end{equation}
	Fix $\eps > 0$ and $\delta > 0$. Then
	\[
		(x,\xi,s,t) \mapsto u(x,s) - \Phi(x - \xi,t) - \phi\left( \xi, \frac{s+t}{2} \right) - \frac{\eps}{2} |x - x_0|^2 - \frac{|s-t|^2}{2\delta}
	\]
	attains a local maximum at some $(x_\delta,\xi_\delta,s_\delta,t_\delta) \in \oline{\Omega} \times \RR^d \times [0,T] \times [0,T]$ such that, for any fixed $\eps > 0$,
	\begin{equation}\label{maxconverges}
		\lim_{\delta \to 0} (x_\delta,\xi_\delta,s_\delta,t_\delta) = (x_0,\xi_0,t_0,t_0).
	\end{equation}
	In particular, $u(x,s) - \psi_\delta(x,s)$ attains a local maximum at $(x_\delta,s_\delta)$, where
	\[
		\psi_\delta(x,s) := \phi\left(x - x_\delta + \xi_\delta, \frac{s+t_\delta}{2} \right) + \frac{\eps}{2} |x - x_0|^2 + \frac{|s - t_\delta|^2}{2\delta}.
	\]
	If $x_0 \notin \partial \Omega$, then, for sufficiently small $\delta$, $x_\delta \notin \partial \Omega$. Otherwise, if $x_0 \in \partial \Omega$ and $x_\delta \in \partial \Omega$, then, in view of \eqref{strictboundarycase} and \eqref{maxconverges}, for sufficiently small $\delta$ and $0 < \eps < 1$,
	\[
		D\psi_\delta(x_\delta,s_\delta) \cdot n(x_\delta) = \left( D\phi\left(\xi_\delta, \frac{s_\delta + t_\delta}{2} \right) + \eps(x_\delta - x_0) \right) \cdot n(x_\delta) > 0,
	\]
	and, therefore,the definition of viscosity solutions yields
	\begin{align*}
		\frac{\partial \psi_\delta(x_\delta,s_\delta)}{\partial s} &\le F^*(D^2 \psi_\delta(x_\delta,s_\delta), D\psi_\delta(x_\delta,s_\delta), , u(x_\delta,s_\delta),x_\delta,t_\delta) \\
		&+ \sum_{i=1}^m H^i(D\psi_\delta(x_\delta,s_\delta)) \dot \zeta(s_\delta),
	\end{align*}
	that is,
	\begin{equation}\label{usubineq}
		\begin{split}
		&\frac{s_\delta - t_\delta}{\delta} + \frac{1}{2} \frac{\partial \phi}{\partial t}\left( \xi_\delta, \frac{s_\delta + t_\delta}{2} \right) \\
		&\le F^*\left( D^2\phi\left( \xi_\delta, \frac{s_\delta + t_\delta}{2} \right) + \eps \Id, D\phi\left( \xi_\delta, \frac{s_\delta + t_\delta}{2} \right) + \eps(x_\delta - x_0), u(x_\delta,s_\delta), x_\delta,s_\delta\right)\\
		&+ \sum_{i=1}^m H^i\left(D\phi\left( \xi_\delta, \frac{s_\delta + t_\delta}{2} \right) + \eps(x_\delta - x_0) \right)\dot \zeta^i(s_\delta).
		\end{split}
	\end{equation}
	We also have that
	\[
		(\eta,t) \mapsto \Phi(\eta,t) + \phi\left( x_\delta - \eta, \frac{s_\delta + t}{2} \right) + \frac{|s_\delta - t|^2}{2\delta}
	\]
	attains a minimum at $(\eta,t) = (x_\delta - \xi_\delta,t_\delta)$, which means that
	\begin{equation}\label{Phisupineq}
		\frac{s_\delta - t_\delta}{\delta} - \frac{1}{2} \frac{\partial \phi}{\partial t}\left( \xi_\delta, \frac{s_\delta+t_\delta}{2} \right) \ge \sum_{i=1}^m H^i\left(D\phi\left( \xi_\delta, \frac{s_\delta + t_\delta}{2} \right) \right)\dot \zeta^i(t_\delta).
	\end{equation}
	Subtracting \eqref{Phisupineq} from \eqref{usubineq}, sending first $\delta \to 0$ and then $\eps \to 0$, we obtain the deisred inequality
	\[
		 \frac{\partial \phi}{\partial t}( \xi_0, t_0) \le F^*(D^2 \phi(\xi_0,t_0), D\phi(\xi_0,t_0), u(x_0,t_0),x_0,t_0).
	\]
	Conversely, assume that $u$ satisfies the sub-solution inequality from Definition \ref{D:nonsmoothH}. Let $\phi: \RR^d \times [0,T] \to \RR$ be $C^2$ in $x$ and $C^1$ in $t$, and assume that $u(x,t) - \phi(x,t)$ attains a local maximum in $\oline{\Omega} \times [0,T]$ at $(x_0,t_0)$. As is standard from the viscosity solution theory, we may assume without loss of generality that $\phi$ takes the form, for some $p \in \RR^d$, $A \in \mbb S^d$, and $a \in \RR$,
	\[
		\phi(x,t) = p \cdot (x - x_0) + \frac{1}{2} A(x - x_0) \cdot (x - x_0) + a(t - t_0),
	\]
	and the maximum is strict over $\oline{\Omega} \times [0,t_0]$. We shall also assume that $x_0 \in \partial \Omega$ (otherwise the argument below is identical) and suppose that $p \cdot n(x_0) > 0$.
	
	Lemma \ref{L:Youngtrick} yields, for any $\eps > 0$ and $\xi \in \RR^d$,
	\begin{align*}
		\frac{1}{2} A(x - x_0) \cdot (x - x_0) 
		&\le \frac{1}{2} (A + \eps A^2) \xi \cdot \xi + \frac{\lambda_\eps}{2}|x - x_0 - \xi|^2,
	\end{align*}
	where $\lambda_\eps := \norm{A} + \eps^{-1}$. Defining also $A_\eps := A + \eps A^2 + \eps \Id$, we find that
	\[
		(x,\xi,t) \mapsto u(x,t) - p \cdot (x - x_0 - \xi) - \frac{\lambda_\eps}{2} |x - x_0 - \xi|^2 - p \cdot \xi - \frac{1}{2} A_\eps \xi \cdot \xi - a(t - t_0)
	\]
	attains a strict maximum in $\oline{\Omega} \times \RR^d \times [0,t_0]$ at $(x_0,0,t_0)$.
	
	Define
	\[
		\psi_\eps(x) := p \cdot (x - x_0) + \frac{\lambda_\eps}{2} |x - x_0|^2,
	\]
	which is convex, and, for $\gamma > 0$, define
	\[
		\psi_{\eps,\gamma}(x) := (\psi_\eps^* + \gamma G)^*(x) = \sup_{q \in \RR^d} \left\{ q \cdot x - \psi_\eps^*(q) - \gamma G(q) \right\},
	\]
	where $G$ is as in \eqref{Gfn}. Then, for fixed $\eps > 0$, as $\gamma \to 0$, $\psi_{\eps,\gamma}$ converges locally uniformly to $\psi_{\eps}$, and so
	\[
		(x,\xi,t) \mapsto u(x,t) - \psi_{\eps,\gamma}(x - \xi) - p \cdot \xi - \frac{1}{2} A_\eps \xi \cdot \xi - a(t - t_0)
	\]
	attains a maximum at some $(x_\gamma, \xi_\gamma,t_\gamma) \in \oline{\Omega} \times \RR^d \times [0,t_0]$ such that $\lim_{\gamma \to 0} (x_\gamma, \xi_\gamma,t_\gamma) = (x_0,0,t_0)$.
	
	Let $h_\gamma > 0$ be such that $\sup_{|s - t| \le h_\gamma} \max_{i=1,2,\ldots,m} |\zeta^i_s - \zeta^i_t| < \gamma$, fix $0 < h < h_\gamma$, and define
	\[
		\Phi(x,t) := \sup_{q \in \RR^d} \left\{ q\cdot x - \psi_\eps^*(q) - \gamma G(q) + \sum_{i=1}^m H^i(q) (\zeta^i_t - \zeta^i_{t_\gamma - h} )\right\}.
	\]
	Then Lemma \ref{L:C11solutions} implies that $\Phi \in C^{1,1}(\RR^d \times [t_\gamma - h, t_\gamma])$ and
	\[
		\frac{\partial \Phi}{\partial t} = \sum_{i=1}^m H^i(D\Phi)\dot \zeta^i(t) \quad \text{in } \RR^d \times [t_\gamma - h, t_\gamma].
	\]
	Moreover, as $\gamma \to 0$, $\Phi$ converges locally uniformly to $\psi_\eps$.
	
	Let $\mu  \in C^1([t_\gamma - h, t_\gamma])$ and suppose that
	\[
		u(x,t) - \Phi(x - \xi,t) - p \cdot \xi - \frac{1}{2} A_\eps\xi \cdot \xi - \mu(t)
	\]
	attains a maximum in $\oline{\Omega} \times \RR^d \times [t_\gamma- h , t_\gamma]$ at some $(\tilde x_\gamma, \tilde \xi_\gamma, \tilde t_\gamma)$ such that $\tilde t_\gamma > t_\gamma - h$. Note first that, if $\tilde x_\gamma \in \partial \Omega$ and $\gamma$ is sufficiently small, then the gradient in $\xi$ satisfies $(p + A_\eps \tilde \xi_\gamma) \cdot n(\tilde x_\gamma) > 0$. Therefore, by Definition \ref{D:nonsmoothH}, 
	\[
		\mu'(\tilde t_\gamma) \le F^*(A_\eps, p + A_\eps \tilde \xi_\gamma, u(\tilde x_\gamma,\tilde t_\gamma),\tilde x_\gamma,\tilde t_\gamma) \le \sup_{(x,\xi) \in a(\tilde t_\gamma)} F^*(A_\eps, p + A_\eps \xi, u(x,\tilde t_\gamma), x,\tilde t_\gamma),
	\]
	where, for $t \in [t_\gamma - h , t_\gamma]$,
	\[
		a(t) := \argmax_{(x,\xi)} \left\{ u(x,t) - \Phi(x - \xi,t) - p \cdot \xi - \frac{1}{2} A_\eps\xi \cdot \xi \right\}.
	\]
	This is a contradiction if 
	\[
		\mu(t) > \int_{t_\gamma - h}^t \sup_{(x,\xi) \in a(s)} F^*(A_\eps, p + A_\eps \xi, u(x,s), x,s)ds,
	\]
	and therefore we find that
	\begin{align*}
		u(x_\gamma,t_\gamma) &- \Phi(x_\gamma - \xi_\gamma,t_\gamma) - p \cdot \xi_\gamma - \frac{1}{2} A_\eps \xi_\gamma \cdot \xi_\gamma\\
		&\le \sup_{(x,\xi) \in \oline{\Omega} \times \RR^d} \left\{ u(x,t_\gamma- h) - \psi_{\eps,\gamma}(x-\xi) - p \cdot \xi - \frac{1}{2} A_\eps \xi \cdot \xi \right\} \\
		&+ \int_{t_\gamma - h}^{t_\gamma}\sup_{(x,\xi) \in a(s)} F^*(A_\eps, p + A_\eps \xi,u(x,s), x,s)ds.
	\end{align*}
	We also have, for all $(x,\xi) \in \RR^d \times \oline{\Omega}$,
	\[
		u(x,t_\gamma - h) - \psi_{\eps,\gamma}(x - \xi) - p \cdot \xi - \frac{1}{2} A_\eps\xi \cdot \xi + ah
		\le u(x_\gamma,t_\gamma) - \psi_{\eps,\gamma} (x_\gamma - \xi_\gamma) - \frac{1}{2} A_\eps \xi_\gamma \cdot \xi_\gamma.
	\]
	Combining the last two inequalities and rearranging terms yields
	\[
		ah \le \int_{t_\gamma - h}^{t_\gamma}\sup_{(x,\xi) \in a(s)} F^*(A_\eps, p + A_\eps \xi, u(x,s),x,s)ds
		+ \Phi(x_\gamma - \xi_\gamma,t_\gamma) - \psi_{\eps,\gamma}(x_\gamma - \xi_\gamma).
	\]
	Dividing by $h$ and sending $h \to 0$ gives
	\[
		a \le \sup_{(x,\xi) \in a(t_\gamma)} F^*(A_\eps, p + A_\eps,\xi, u(x,t_\gamma), x,t_\gamma) + \sum_{i=1}^m H^i(D\psi_{\eps,\gamma}(x_\gamma - \xi_\gamma,t_\gamma)) \dot \zeta^i(t_\gamma).
	\]
	We conclude, upon sending $\gamma \to 0$ and then $\eps \to 0$, that $a \le F^*(A,p, u(x_0,t_0),x_0,t_0) +\sum_{i=1}^m H^i(p)\dot \zeta^i(t_0)$, as desired.
\end{proof}

We next discuss the stability of sub- and super-solutions under half-relaxed limits. A consequence of the following result is the local-uniform stability of solutions.

{
Given a sequence $(u^n)_{n \in \NN} : \RR^d \times [0,T] \to \RR$, we define the upper- and lower- half-relaxed limits
\begin{align*}
	&u^\star(x,t) := \lim_{N \to \oo} \sup\left\{ u^n(y,s) : n \ge N, \; |y-x| + |s-t| \le \frac{1}{N} \right\}
	 \quad \text{and} \\ 
	&u_\star(x,t) := \lim_{N \to \oo} \inf\left\{ u^n(y,s) : n \ge N, \; |y-x| + |s-t| \le \frac{1}{N} \right\}.
\end{align*}
Note the difference between $u^\star$ and $u^*$ (resp. $u_\star$ and $u_*$), where the latter is the upper (lower) semi-continuous envelope of the fixed function $u$. Roughly speaking, in the half-relaxed limits above, the $\limsup$- and $\liminf$- operations are performed simultaneously in $\RR^d \times [0,T]$ and as $n \to \oo$ (see \cite[Section 6]{CIL}).
}

We then have the following stability result.

\begin{proposition}\label{P:halfrelaxed}
	Assume $H$ satisfies \eqref{A:DCH}, $\left\{(\zeta^n)_{n \in \NN},\zeta \right\}\subset C_0([0,T],\RR^m)$,  and, as $n \to \oo$, $\zeta^n$ converges uniformly to $\zeta$. Let $(u^n)_{n \in \NN} \subset USC(\RR^d \times [0,T])$ (resp. $LSC(\RR^d \times [0,T])$) be a sequence of sub- (resp. super-)solutions of \eqref{E:neumann} corresponding to the sequence $(\zeta^n)_{n \in \NN}$. Then $u^\star$ (resp. $u_\star$) belongs to $USC(\RR^d \times [0,T])$ (resp. $LSC(\RR^d \times [0,T])$) if it is finite, and is a sub- (resp. super-)solution of \eqref{E:neumann} corresponding to $\zeta$.
\end{proposition}

\begin{proof}
	We prove only the sub-solution statement, as the super-solution one is similar. Let $\Phi$ be as in Definition \ref{D:nonsmoothH}, define
	\[
		\oline{u}(\xi,t) := \sup_{x \in \oline{\Omega}} \left\{ u^\star(x,t) - \Phi(x-\xi,t) \right\}, \quad (\xi,t) \in \RR^d \times [0,T],
	\]
	and assume that, for some smooth $\phi$, $\oline{u}(\xi,t) - \phi(\xi,t)$ attains a strict maximum at $(\xi_0,t_0)$. 
	
	Choose $x_0 \in A^+_{\xi_0,t_0}$ and $M > 0$, and let $\tilde \Phi$ be the solution of \eqref{E:pathwiseHJnonsmooth} such that $\tilde \Phi(\cdot,t_0) = \Phi(\cdot,t_0) + M|\cdot - x_0|$. Then $\tilde \Phi \ge \Phi$, and if $M$ is suffiiently large, then
	\[
		\lim_{|x| \to \oo}\inf_{t \in [0,T]} \tilde \Phi(x,t) = +\oo,
	\]
	which is a consequence of the estimates in \cite[Theorem 7.2]{Snotes}. We then have that
	\[
		u^\star(x,t) - \tilde \Phi(x-\xi,t) - \phi(\xi,t)
	\]
	attains a strict maximum at $(x_0,\xi_0,t_0)$.
	
	For $n \in \NN$, let $\Phi^n$ solve
	\[
		d\Phi^n = \sum_{i=1}^m H^i(D\Phi^n) d \zeta^{n,i} \quad \text{in } \RR^d \times [0,T], \quad \Phi^n(\cdot,t_0) = \tilde \Phi(\cdot,t_0).
	\]
	Then
	\[
		\lim_{|x| \to +\oo} \inf_{n \in \NN} \inf_{t \in [0,T]} \Phi^n(x,t) = +\oo,
	\]
	and $\lim_{n \to \oo} \Phi^n = \tilde \Phi$ locally uniformly in $\RR^d \times [0,T]$.
	
	Let $(x_n,\xi_n,t_n)$ be a maximum point of $u^n(x,t) - \Phi^n(x-\xi,t) - \phi(\xi,t)$, and let $(y,\eta,s) \in \oline{\Omega} \times \RR^d \times [0,T]$ be a limit point as $n \to \oo$ (the set of limit points is nonempty in view of the uniform-in-$n$ growth of $\Phi^n$ as $|x| \to \oo$). Let $(\tilde x_n, \tilde t_n)$ be any sequence such that
	\[
		\lim_{n \to \oo} (\tilde x_n,\tilde t_n) = (x_0,t_0) \quad \text{and} \quad \lim_{n \to \oo} u^n(\tilde x_n,\tilde t_n) = u^\star(x_0,t_0).
	\]
	Sending $n \to \oo$ in the inequality
	\[
		u^n(\tilde x_n,\tilde t_n) - \Phi^n(\tilde x_n - \xi_0,\tilde t_n) - \phi(\xi_0,\tilde t_n)
		\le u^n(x_n,t_n) - \Phi^n(x_n - \xi_n, t_n) - \phi(\xi_n,t_n)
	\]
	gives
	\[
		u^\star(x_0,t_0) - \Phi(x_0 - \xi_0,t_0) - \phi(\xi_0,t_0)
		\le u^\star(y,s) - \Phi(y - \xi_0,s) - \phi(\xi_0,s),
	\]
	which, in view of the strictness of the maximum, means that $(y,s) = (x_0,t_0)$, and thus the full sequence $(x_n,t_n)$ converges, as $n \to \oo$, to $(x_0,t_0)$. 
	
	Applying the sub-solution property to each $u^n$, noting that, for sufficiently large $n$, $D\phi(\xi_n,t_n) \cdot n(x_n) > 0$, we find that
	\[
		\phi_t(\xi_n,t_n) \le F^*(D^2 \phi(\xi_n,t_n), D\phi(\xi_n,t_n), u^n(x_n,t_n),x_n,t_n),
	\]
	and so sending $n \to \oo$ gives the desired solution inequality
	\[
		\phi_t(\xi_0,t_0) \le F^*(D^2 \phi(\xi_0,t_0), D\phi(\xi_0,t_0) , u^\star(x_0,t_0),x_0,t_0).
	\]
\end{proof}

{
The final result of the Appendix is a doubling variables lemma, which is the second-order analogue of Lemma \ref{L:doubled}. Before we state it, we introduce the notions of parabolic sub- and super-jets, as well as their limiting counterparts (see \cite{CIL} for example): for $(x,t) \in \oline\Omega \times (0,T)$ and $u \in USC(\oline\Omega \times (0,T))$,
\begin{align*}
	\mcl P^+ u(x,t) := &\Big\{ (a,p,X) \in \RR \times \RR^d \times \mbb S^d \\
	&: u(x',t') \le u(x,t) + p \cdot (x' - x) + \frac{1}{2} X(x' - x) \cdot (x' - x) \\
	& \quad + a(t' - t) + o( |x' - x|^2 + |t' - t|) \text{ as } (x',t') \to (x,t) \Big\};
\end{align*}
for $u \in LSC(\oline\Omega \times (0,T))$,
\begin{align*}
	\mcl P^- u(x,t) := &\Big\{ (a,p,X) \in \RR \times \RR^d \times \mbb S^d \\
	&: u(x',t') \ge u(x,t) + p \cdot (x' - x) + \frac{1}{2} X(x' - x) \cdot (x' - x) \\
	&  \quad+ a(t' - t) + o( |x' - x|^2 + |t' - t|) \text{ as } (x',t') \to (x,t) \Big\};
\end{align*}
and for $u \in USC(\oline\Omega \times (0,T))$ (resp. $LSC(\oline\Omega \times (0,T))$),
\begin{align*}
	\oline{\mcl P}^\pm u(x,t) &:= \Big\{ (a,p,X) \in \RR \times \RR^d \times \mbb S^d : \exists  \big( (x_n,t_n) \big)_{n \in \NN} \subset \Omega \times (0,T) \text{ and, for $n \in \NN$, } \\
	&(a_n,p_n,X_n) \in \mcl P^\pm u(x_n,t_n) \text{ such that } \lim_{n \to \oo} (x_n,t_n,a_n,p_n,X_n) = (x,t,a,p,X) \Big\}.
\end{align*}

We will make use of the following version of the Ishii lemma from the theory of viscosity solutions:

\begin{lemma}\label{L:maxprinciple}
	Let $u \in USC(\RR^d \times [0,T])$ and $v \in LSC(\RR^d \times [0,T])$, assume there exists $C > 0$ such that, in the sense of distributions,
	\[
		D^2 u \ge -C \Id, \quad D^2 v \le C \Id, \quad \frac{\partial u}{\partial t} \le C, \quad \text{and} \quad \frac{\partial v}{\partial t} \ge -C.
	\]
	Define $w(x,y,t) = u(x,t) - v(y,t)$, and assume that, for some $(x_0,y_0,t_0) \in \RR^d \times \RR^d \times [0,T]$, $a \in \RR$, $p,q \in \RR^d$, and $A \in \mbb S^{2d}$, we have $\big( a, (p,q), A \big) \in \mcl P^+ w(x_0,y_0,t_0)$. Then there exist $X,Y \in \mbb S^d$ and $\alpha,\beta \in \RR$ such that $(\alpha,p,X) \in \oline{\mcl P}^+ u(x_0,t_0)$, $(\beta,-q,Y) \in \oline{\mcl P}^-v(y_0,t_0)$, $\alpha - \beta = a$, and
	\[
		-C
		\begin{pmatrix}
			 \Id & 0 \\
			0 & \Id
		\end{pmatrix}
		\le
		\begin{pmatrix}
			X & 0 \\
			0 & -Y
		\end{pmatrix}
		\le
		A.
	\]
\end{lemma}

We omit the proof of Lemma \ref{L:maxprinciple}, because it follows exactly as in \cite[Theorem 7]{CImax}. The only difference is that $u$ and $-v$ are already assumed to be semi-convex, and it is not required to regularize $u$ or $v$ in the space variable via $\sup$- (resp. $\inf$-) convolution. This accounts for the upper and lower bounds for the matrix in the statement of the lemma.

\begin{proposition}\label{P:thmofsums}
	Assume that $u \in USC(\oline{\Omega} \times [0,T])$ and $v \in LSC(\oline{\Omega} \times [0,T])$ are respectively a sub- and super-solution of \eqref{E:smoothgeneraleq}. For $0 \le a < b \le T$, let $\Phi \in C([a,b],C^2(\oline{\Omega}))$ be a solution of \eqref{E:doubled} and $\psi \in C^1([a,b])$, and assume that
	\[
		\oline{\Omega} \times \oline{\Omega} \times [a,b] \ni (x,y,t) \mapsto u(x,t) - v(y,t) - \Phi(x,y,t) -\psi(t)
	\]
	attains a maximum at $(x_0,y_0,t_0)$ with $t_0 > a$,
	\begin{equation}\label{A:D>0}
		\left\{
		\begin{split}
		&D_x \Phi(x_0,y_0,t_0) \cdot n(x) > 0 \quad \text{if } x_0 \in \del \Omega, \quad \text{and}\\
		&D_y \Phi(x_0,y_0,t_0) \cdot n(y) > 0 \quad \text{if } y_0 \in \del \Omega.
		\end{split}
		\right.
	\end{equation}
	Set $\phi := \Phi(\cdot,\cdot,t_0)$. Then, for every $\delta >0$, there exist $X_\delta,Y_\delta \in \mathbb{S}^d$ such that
	\begin{equation}\label{E:lemmamatrix}
		\begin{split}
		- \left( \norm{D^2\phi(x_0,y_0)} + \frac{1}{\delta} \right)
		\begin{pmatrix}
			\Id & 0 \\
			0 & \Id
		\end{pmatrix}
		&\le
		\begin{pmatrix}
			X_\delta & 0 \\
			0 & -Y_\delta
		\end{pmatrix}\\
		&\le
		D^2\phi(x_0,y_0) + \delta (D^2 \phi(x_0,y_0))^2 
		\end{split}
	\end{equation}
	and
	\[
		\psi'(t_0) \le F(X_\delta,D_x \phi(x_0,y_0), u(x_0,t_0), x_0, t_0) - F(Y_\delta, -D_y \phi(x_0,y_0), v(y_0,t_0), y_0, t_0).
	\]
\end{proposition}

The proof of Proposition \ref{P:thmofsums} is similar to the uniqueness proof in \cite{LSunique} (see also \cite{Snotes}), but we have succeeded in simplifying many of the arguments. Additional technicalities also arise in order to deal with the boundary condition.

\begin{proof}[Proof of Proposition \ref{P:thmofsums}]
	
	By subtracting a constant, we may also assume without loss of generality that $\phi(x_0,y_0) = 0$.  Set $p := D_x \phi(x_0,y_0)$, $q := D_y \phi(x_0,y_0)$, $A := D^2 \phi(x_0,y_0)$, and $a := \psi'(t_0)$, and fix $\eps > 0$. Then, for some $r>0$ and for all $(x,y) \in B_r(x_0) \times B_r(y_0)$,
	\[
		\phi(x,y) \le p\cdot (x - x_0) + q \cdot (y - y_0) + \frac{1}{2} (A+ \eps I) (x - x_0,y-y_0) \cdot (x-x_0,y-y_0).
	\]
	In view of \eqref{A:D>0}, if $x_0 \in \del \Omega$, then $p \cdot n(x_0) > 0$, and if $y_0 \in \del \Omega$, then $q \cdot n(y_0) > 0$.
	
	We use the inequality in Lemma \ref{L:Youngtrick} arising from Young's inequality, with $n = 2d$, $\mcl A = A + \eps I := A_\eps$, $X = (x-x_0,y-y_0)$, and $\Xi = (\xi,\eta)$. Setting $A_{\delta,\eps} = A_\eps + \delta A_\eps^2$, this yields, for all $(x,y,\xi,\eta) \in B_r(x_0) \times B_r(y_0) \times \RR^d \times \RR^d$,
	\begin{align*}
		\phi(x,y) 
		&\le p\cdot (x -x_0 - \xi) + q \cdot ( y - y_0- \eta) \\
		&+ \frac{1}{2} \left(\norm{A_\eps} + \frac{1}{\delta} \right)( |x - x_0-\xi|^2 + |y -y_0 - \eta|^2) \\
		&+ p\cdot \xi + q \cdot \eta + \frac{1}{2} A_{\delta,\eps} (\xi,\eta) \cdot (\xi, \eta).
	\end{align*}	
	We introduce the shorthand
	\[
		S_{\pm}(t,t_0) := \prod_{i=1}^m S^i_{\pm}(\zeta^i_t - \zeta^i_{t_0}) \quad \text{and} \quad
		S_d(t,t_0) := \prod_{i=1}^m S^i_d(\zeta^i_t - \zeta^i_{t_0}).
	\]
	If $f,g \in C^2(\RR^d)$ and $\psi(x,y) := f(x) + g(y)$, then, for all $(x,y) \in \RR^d \times \RR^d$ and $t$ sufficiently close to $t_0$, depending on $\norm{D^2f}_\oo$ and $\norm{D^2 g}_\oo$,
	\[
		S_d(t,t_0)\tilde \psi(x,y) = S_+(t,t_0)f(x) + S_-(t,t_0)g(y).
	\]
	Then, if $h$ is sufficiently small, we can define, for $(x,y, \xi, \eta,t) \in B_r(x_0) \times B_r(y_0) \times \RR^d \times \RR^d \times (t_0 - h, t_0 + h)$,
	\[
		\left\{
		\begin{split}
		&\Phi_+(x,\xi,t) := S_+(t,t_0) \left(p \cdot (\cdot - x_0 - \xi) + \frac{1}{2} \left( \norm{A_\eps} + \frac{1}{\delta} \right) |\cdot - x_0 - \xi|^2 \right)(x),\\
		&\Phi_-(y,\eta,t) := S_-(t,t_0) \left( q \cdot (\cdot - y_0 - \eta) + \frac{1}{2} \left( \norm{A_\eps} + \frac{1}{\delta} \right)|\cdot - y_0 - \eta|^2 \right)(y),\\
		&\oline{u}(\xi,t) := \sup_{x \in B_r(x_0)} \left( u(x,t) - \Phi_+(x,\xi,t) \right), \quad \text{and} \\
		&\uline{v}(\eta,t) := \inf_{y \in B_r(y_0)} \left( v(y,t) + \Phi_-(y,\eta,t) \right).
		\end{split}
		\right.
	\]	
	It then follows that
	\[
		\oline{u}(\xi,t) - \uline{v}(\eta,t) - p\cdot \xi - q\cdot \eta - \frac{1}{2}  A_{\delta,\eps} (\xi,\eta) \cdot (\xi,\eta) - a(t- t_0) - \frac{\eps}{2} |t-t_0|^2
	\]
	attains a local maximum at $(\xi,\eta,t) = (0,0,t_0)$ in $\RR^d \times \RR^d \times (t_0 - h,t_0 + h)$.
	
		If $x_0 \in \del \Omega$ (so that $p \cdot n(x_0) > 0$), then, shrinking $h$ and $r$ if necessary, we have, for $(x,\xi, t) \in B_r(x_0) \times B_r(0) \times (t_0 - h,t_0 + h)$, the strict inequality $D_x \Phi_+(x,\xi,t) \cdot n(x) > 0$, with a similar comment for when $y_0 \in \del \Omega$.
	
	Both $D^2\Phi_+$ and $D^2\Phi_-$ are continuous on $\RR^d \times \RR^d \times (t_0 - h, t_0 +h)$, and, therefore,
	\begin{align*}
		C_h := &\sup_{(x,\xi,t) \in B_r(x_0) \times \RR^d \times (t_0 - h, t_0 + h)} \norm{D_\xi^2 \Phi_+(x,\xi,t)}\\
		&\vee \sup_{(y,\eta) \in B_r(y_0) \times \RR^d \times (t_0 - h, t_0 + h) } \norm{D_\eta^2 \Phi_-(y,\eta,t)}
		< \oo,
	\end{align*}
	with $\lim_{h \to 0} C_h = \norm{A_\eps} + \frac{1}{\delta}$. Then $\oline{u}$ and $\uline{v}$ are respectively semiconvex and semiconcave in the spatial variable, and, in the distributional sense, $D_\xi^2 \overline{u} \ge - C_h \Id$ and $D_\eta^2 \underline{v} \le C_h \Id$.
	
	Next, observe that, if $h$ is sufficiently small, then, for all $t \in (t_0 - h, t_0 + h)$, the supremum and infimum in the definitions of respectively $\oline{u}$ and $\uline{v}$ are achieved for some $x(t) \in B_r(x_0)$ and $y(t) \in B_r(y_0)$. This is because, for $t = t_0$, the extrema are attained uniquely at respectively $x = x_0$ and $y = y_0$ because of the addition of $\eps \Id$ to $A$. 
	
%
	Fix $s \in (t_0 - h,t_0+h)$ and $ \xi \in B_r(0)$, and assume, for some $\alpha > 0$, that $\oline{u}( \xi,t) - \alpha t$ attains a maximum at some $\bar t \in (s,t_0+h]$. Then, for some $\bar x \in B_r(x_0)$,
	\[
		u(x,t) - \Phi(x, \xi,t) - \alpha t
	\]
	attains a local maximum at $(\bar x, \bar t)$. Note that, if $x_0 \in \del \Omega$ and $\bar x \in \del \Omega$, then we have ensured that $D_x \Phi(\bar x, \xi, \bar t) \cdot n(\bar x) > 0$, and, therefore, Definition \ref{D:smoothH} yields
	\[
		\alpha \le F(D^2_x \Phi(\bar x,  \xi,\bar t), D_x \Phi(\bar x,  \xi, \bar t), u(\bar x, \bar t), \bar x, \bar t).
	\]
	In view of \eqref{A:Fcts}, this is a contradiction for sufficiently large $\alpha$, depending on bounds for $u$, $D\Phi$, and $D^2 \Phi$. It follows that, for some $\alpha_0 > 0$, $\del\oline{u}/\del t \le \alpha_0$ as a distribution on $B_r(0) \times (t_0 - h, t_0 + h)$, and, similarly, there exists $\beta_0>0$ such that $\del \uline{v}/\del t \ge -\beta_0$ on $B_r(0) \times (t_0 - h, t_0 + h)$.
	
	It is now a consequence of Lemma \ref{L:maxprinciple} that there exist $X_{\delta,\eps},Y_{\delta,\eps} \in \mbb S^d$ and $\alpha,\beta \in \RR$ such that $(\alpha,p,X_{\delta,\eps}) \in \oline{\mcl P}^+ u(0,t_0)$, $(\beta,-q,Y_{\delta,\eps}) \in \oline{\mcl P}^- v(0,t_0)$, $\alpha - \beta = a$, and
	\begin{equation} \label{E:matrix}
		-C_h
		\begin{pmatrix}
			\Id & 0 \\
			0 & \Id
		\end{pmatrix}
		\le
		\begin{pmatrix}
			X_{\delta,\eps} & 0 \\
			0 & -Y_{\delta,\eps}
		\end{pmatrix}
		\le
		A_{\delta,\eps}.
	\end{equation}
	Then, for all $n \in \NN$, there exist $(\xi_n,\eta_n,s_n,t_n) \in \RR^d \times \RR^d \times (t_0 - h, t_0 + h) \times (t_0 - h, t_0 + h)$, $(\alpha_n,p_n, X_{\delta,\eps,n}) \in \mcl P^+ \oline{u}(\xi_n,s_n)$, and $(\beta_n,q_n, Y_{\delta,\eps,n}) \in \mcl P^-\uline{v}(\eta_n,t_n)$ such that
	\[
		\lim_{n \to \oo} (X_{\delta,\eps,n},Y_{\delta,\eps,n},p_n, q_n, \alpha_n, \beta_n, \xi_n, \eta_n, s_n,t_n) = (X_{\delta,\eps},Y_{\delta,\eps},p,q,\alpha,\beta,0,0,t_0,t_0).
	\]
%
	Let
	\[
		x_n \in \argmax_{x \in \oline{B_r(x_0)}} \left( u(x,s_n) - \Phi_+(x,\xi_n,s_n) \right)
	\quad \text{and} \quad
		y_n \in \argmin_{y \in \oline{B_r(y_0)}} \left( v(y,t_n) + \Phi_-(y,\eta_n,t_n) \right).
	\]
	Observe that $(x_n,y_n) \to (0,0)$ as $n \to \oo$. Indeed, if $(\oline{x}, \oline{y})$ is a limit point of $((x_n,y_n))_{n \in \NN}$, then
	\[
		u(x, t_0) - v(y,t_0) - p \cdot x - q \cdot y - \frac{1}{2}\left(\norm{A_\eps} + \frac{1}{\delta} \right)(|x|^2 + |y|^2)
	\]
	attains a maximum in $\oline{B_r(x_0)} \times \oline{B_r(y_0)}$ at $(\oline{x}, \oline{y})$, and, therefore, $(\oline{x}, \oline{y}) = (0,0)$. Thus, if $n$ is large enough, $x_n \in B_r(x_0)$ and $y_n \in B_r(y_0)$. Moreover, if $x_0 \in \del \Omega$, then, for $n$ sufficiently large, if also $x_n \in \del \Omega$, then $D_x \Phi_n^+(x_n,s_n) \cdot n(x_n) > 0$, and a similar remark holds if $y_0 \in \del \Omega$ and $y_n \in \del \Omega$ for some sufficiently large $n$.
	
	Definition \ref{D:nonsmoothH} now yields
	\[
		\alpha_n - \beta_n \le F(X_{\delta,\eps,n},p_n, u(x_n,s_n),x_n,s_n) - F(Y_{\delta,\eps,n},q_n, v(y_n,t_n),y_n, t_n),
	\]
	and sending $n \to \oo$ then implies that
	\[
		a \le F(X_{\delta,\eps},p,u(0,t_0),0,t_0) - F(Y_{\delta,\eps},q,v(0,t_0),0,t_0).
	\]
	As $\eps \to 0$, along an appropriate subsequence, $X_{\delta,\eps}$ and $Y_{\delta,\eps}$ converge to some matrices $X_\delta,Y_\delta \in \mbb S^d$ satisfying \eqref{E:matrix} with right-hand side $A_\delta$. The proof is finished upon sending $h \to 0$.

\end{proof}
}

\end{appendix}
%
%

\begin{funding}
The first author was partially supported by the ANR via the project ANR-16-CE40-0020-01.
The second author was partially supported by NSF DMS 1902658 and NSF DMS 1840314
\end{funding}



\bibliographystyle{imsart-number} 
\bibliography{neumannrefs}       


\end{document}